\documentclass{psp}
\usepackage{tikz}

\usepackage{amsmath}
\usepackage{amsfonts}
\usepackage{amssymb}
\usepackage{subfigure} 
\usepackage{graphicx}
\usepackage{color}
\usepackage{enumerate}
\usepackage{lineno}
\usepackage{pinlabel}
\usepackage{hyperref}
\usepackage{comment}
\usepackage{enumitem}
\usepackage{framed}

\ifprodtf \else
  \checkfont{msam10}
  \iffontfound
    \IfFileExists{amsfonts.sty}
      {\typeout{^^JFound AMS Symbol fonts on the system, using the
                'amsfonts' package.^^J}%
       \usepackage{amsfonts}%
         
         \let\geq=\geqslant
      }
      {\providecommand\mathbb[1]{\mathsf{##1}}
       \providecommand\mathfrak[1]{\mathcal{##1}}}
  \else
    \providecommand\mathbb[1]{\mathsf{#1}}
    \providecommand\mathfrak[1]{\mathcal{#1}}
  \fi

\fi

\ifprodtf \else
  \IfFileExists{amsbsy.sty}
    {\typeout{^^JFound the 'amsbsy' package on the system, using it.^^J}%
     \usepackage{amsbsy}}
    {}

\fi

\newtheorem{theorem}{Theorem}
\newtheorem{proposition}[theorem]{Proposition}
\newtheorem{lemma}[theorem]{Lemma}

\newtheorem{corollary}[theorem]{Corollary}

\newtheorem{definition}[theorem]{Definition}

\newtheorem{remark}[theorem]{Remark}

\newcommand{\ba}{\backslash}
\newcommand{\R}{\mathbb{R}}
\newcommand{\G}{\mathbb{G}}
\newcommand{\V}{\mathcal{V}}
\newcommand{\B}{\mathcal{B}}
\newcommand{\T}{\mathbb{T}}
\newcommand{\h}{\mathbb{H}}

\newcommand{\bs}{\mathrm{bs}}
\newcommand{\bp}{\mathrm{bp}}
\newcommand{\olc}{\mathrm{olc}}
\newcommand{\olh}{\mathrm{olh}}
\newcommand{\nl}{\mathrm{nl}}
\newcommand{\ol}{\mathrm{ol}}
\newcommand{\br}{\mathrm{b}}

\begin{document}

\title[Types of embedded graphs and their Tutte polynomials]{Types of embedded graphs and their Tutte polynomials}

\author[S.~Huggett and I.~Moffatt]{
\uppercase{Stephen Huggett} \\ Centre for Mathematical Sciences,  University of Plymouth,  Plymouth,  PL4 8AA
\addressbreak email: {\tt s.huggett@plymouth.ac.uk}
\and\ 
\uppercase{Iain Moffatt} \\
Department of Mathematics,   Royal Holloway, University of London, \addressbreak  Egham, TW20 0EX
\addressbreak  email: {\tt iain.moffatt@rhul.ac.uk}
}


\volume{121}
\pubyear{1997}
\setcounter{page}{1}
\receivedline{Received \textup{xxxx};
              revised \textup{xxxx}}
\maketitle

\begin{abstract}
We take an elementary and systematic approach to the problem of extending the Tutte polynomial to the setting of embedded graphs. 
Four notions of embedded graphs arise naturally when considering deletion and contraction operations on graphs on surfaces. We give a description of each class in terms of coloured ribbon graphs. We then identify a universal deletion-contraction invariant (i.e., a `Tutte polynomial') for each class. We relate these to graph polynomials in the literature, including the Bollob\'as--Riordan, Krushkal, and Las~Vergnas polynomials, and give state-sum formulations, duality relations, deleton-contraction relations, and quasi-tree expansions for each of them.
\end{abstract}

\section{Introduction}

\subsection{Overview}
The phrase `topological Tutte polynomial'  refers to an analogue of the Tutte polynomial for graphs in surfaces or for related objects. Our aim here is to define topological Tutte polynomials by (i) starting from first principles, and (ii) proceeding in a canonical way.

To do so we take as our starting point the slightly vague but uncontroversial notion of a Tutte polynomial of an object as `something defined through a deletion-contraction relation' like the one for the classical Tutte polynomial (given below in \eqref{d1}). Crucial to this philosophy is that the deletion-contraction procedure should terminate on trivial objects (such as edgeless graphs) just as the classical Tutte polynomial does. This requirement constitutes a key difference between the approach here and those taken by  M.~Las~Vergnas \cite{Las78}, B.~Bollob\'as and O.~Riordan \cite{BR01,BR02}, and V.~Krushkal \cite{Kr} whose topological Tutte polynomials do not have this property. (However, we will see that these polynomials can be obtained by restricting the domains of those constructed here.)

To satisfy our second requirement (that we work canonically) we proceed by decoupling the definition of the Tutte polynomial of a graph from the specific language of graphs (such as loops, bridges, etc.), and formulating it in a way that depends on the existence of (an appropriate) deletion and contraction. Since graphs on surfaces also have concepts of deletion and contraction, this formulation enables us to obtain topological Tutte polynomials.

However, there are different notions of how to delete and contract edges in an embedded graph, and a requirement that the domain be minor-closed means that each notion results in a different `Tutte polynomial'. Our canonical approach thus leads to a family of four topological Tutte polynomials:
\begin{center}
\begin{tabular}{| l | l |}
\hline
\textbf{Polynomial} & \textbf{Tutte polynomial of graphs $\dots$}
\\
\hline
$T_{ps}(G;x,y,a,b)$ & $\dots$ embedded in pseudo-surfaces \\
\hline
$T_{cps}(G;x,y,z)$& $\dots$ cellularly embedded  in pseudo-surfaces \\
\hline
$T_s(G;x,y,z)$& $\dots$ embedded in surfaces \\
\hline
$T_{cs}(G;x,y)$& $\dots$ cellularly embedded in surfaces\\
\hline
\end{tabular}
\end{center}
For each of these polynomials we provide (i) `full' deletion-contraction procedures that terminate on edgeless embedded graphs, (ii) a state-sum formulation, (iii) an activities expansion, (iv) a universality theorem, and (v) a duality formula. These are all summarised in Section~\ref{djkda}.

Furthermore, the most common topological Tutte polynomials from the literature (namely the Las~Vergnas polynomial \cite{Las78}, the Bollob\'as--Riordan polynomial \cite{BR01,BR02}, and the Krushkal polynomial \cite{But,Kr}) can each be recovered from the above family by restricting domains (see Section~\ref{dhjk}). However, we emphasise that the polynomials presented here have `full' deletion-contraction relations that take edgeless graphs as the base, while the polynomials from the literature do not. 

This paper has the following structure.
The remainder of this section outlines our approach and philosophy. 
Section~\ref{s.3} describes various notions of embedded graphs and their minors, and introduces a description of each of these as (coloured) ribbon graphs.  
Section~\ref{s.5} introduces the family of topological Tutte polynomials.
Section~\ref{s.7} is concerned with activities (or tree and quasi-tree) expansions.

We assume a familiarity with basic graph theory and of the elementary parts of the topology of surfaces. Given $A\subseteq E$, we use $A^c$ to denote its complement $E\ba A$. For notational simplicity, in places we denote sets of size one by their unique element, for example writing $E\ba e$ in place of $E \ba \{e\}$. Initially we use the phrase `graphs on surfaces' fairly loosely, but make precise what we mean in Section~\ref{s.3}.

\subsection{A review of the standard definitions of the Tutte polynomial of a graph}\label{s.1}

Unsurprisingly, given the wealth of its applications, there are many formulations, and indeed definitions, of the Tutte polynomial. Among these, there are three that can be regarded as the standard definitions. We take these as our starting point. Throughout this section we let $G=(V,E)$ be a (non-embedded) graph, and note that graphs here may have loops and multiple edges.

Our first definition of the Tutte polynomial is the recursive \emph{deletion-contraction definition}. This defines the \emph{Tutte polynomial}, $T(G;x,y)\in \mathbb{Z}[x,y]$ as the graph polynomial defined recursively by the \emph{deletion-contraction} relations
\begin{equation}\label{d1}
T(G;x,y) =\begin{cases}   xT(G / e;x,y) &  \text{if $e$ is a bridge,} \\
 yT(G\backslash e;x,y ) &  \text{if $e$ is a loop,} \\
 T(G\backslash e;x,y ) + T(G/ e;x,y ) &  \text{if $e$ is an ordinary edge,}\\
   1 &  \text{if }E(G)=\emptyset.
 \end{cases}
 \end{equation}
Here, $G\ba e$ denotes the graph obtained from $G$ by deleting the edge $e$, and $G/ e$ the graph obtained by contracting $e$. An edge $e$ of $G$ is a \emph{bridge} if its deletion increases the number of components of the graph, a \emph{loop} if it is incident with exactly one vertex, and is \emph{ordinary} otherwise.

The deletion-contraction relations determine a polynomial, since their repeated application to edges in $G$ enables us to express the Tutte polynomial of $G$ as a  $\mathbb{Z}[x,y]$--linear combination of edgeless graphs, on which $T$ has the value 1. Such a repeated application of the deletion-contraction relations \emph{does} require a choice of the order of edges. $T(G;x,y)$ is independent of this choice, so it is \emph{well-defined}, but not trivially so. We will come back to this point shortly.

Our second standard definition is the \emph{state-sum definition}. This defines the Tutte polynomial as the graph polynomial $T(G;x,y)\in \mathbb{Z}[x,y]$ defined by
 \begin{equation}\label{d2}
     T(G;x,y) = \sum_{A\subseteq E} (x-1)^{r(E)-r(A)}  (y-1)^{|E|-r(A)},
 \end{equation}
where  $r(A)$ denotes the \emph{rank} of the spanning subgraph $(V,A)$ of $G$, which can be defined as the number of edges in a maximal spanning forest of $(V,A)$. 
Note that $r(A)=|V|-k(A)$ where $k(A)$ is the number of connected components of $(V,A)$.

Our third standard definition, the \emph{activities definition}, writes the Tutte polynomial as a bivariate generating function. 
Fix a linear order of the edges of $G$, and for simplicity assume that $G$ is connected.
Suppose that $F$ is a  spanning tree of $G$ and let $e$ be an edge of $G$.
If $e\notin F$, then the graph $F\cup e$ contains a unique cycle, and we say that $e$ is \emph{externally active} with respect to $F$ if it is the smallest edge in this cycle.
If $e\in F$, then we say that  $e$ is \emph{internally active} if $e$ is the smallest edge of $G$ that can be added to $F\ba e$ to recover a spanning tree of $G$.

Then the activities definition of the Tutte polynomial of a connected graph is the graph polynomial $T(G;x,y)\in \mathbb{Z}[x,y]$ obtained by fixing a linear order of $E$ and setting
\begin{equation}\label{d3}
  T(G;x,y)=\sum\limits_{F}x^{\mathrm{IA}(F)} y^{\mathrm{EA}(F)},
\end{equation}
where the sum is over all  spanning trees $F$ of $G$, and where $\mathrm{IA}(F)$ (respectively $\mathrm{EA}(F)$) denotes the number internally active (respectively, externally active) edges of $G$ with respect to $F$ and the linear ordering of $E$.

The activities definition requires a choice of edge order and it is far from obvious (and quite wonderful) that the result of $\eqref{d3}$ is independent of this choice.

Any one of \eqref{d1}--\eqref{d3} can be taken to be \emph{the} definition of the Tutte polynomial, with the other two being recovered as theorems.
However, the easiest way to prove the equivalence of all three definitions is to show that the sums in \eqref{d2} and \eqref{d3} both satisfy the relations in \eqref{d1}. Since \eqref{d2} is clearly independent of choice of edge order it follows that all three expressions are, and we can take any as the definition.

\subsection{Choosing the fundamental definition}
Suppose we are seeking to define a Tutte polynomial for a different class of objects (such as graphs embedded in surfaces).
The first step is to decide what we mean by the expression `a Tutte polynomial'. For this we need to decide which definition of the Tutte polynomial we regard as being the fundamental one. Here we choose the deletion-contraction definition given in \eqref{d1} as the most fundamental, for the following reasons.

As combinatorialists our interest in the Tutte polynomial lies in the fact that it contains a vast amount of combinatorial information about a graph. The reason for this is that the Tutte polynomial stores all graph parameters which satisfy the deletion-contraction relations, as follows.
\begin{theorem}[Universality]\label{thm1}
 Let $\mathcal{G}$ be a minor-closed class of graphs. Then there is a unique map $U: \mathcal{G}\rightarrow \mathbb{Z}[x,y,a,b,\gamma]$ such that
 \begin{equation}\label{EMM:eq:U}
U(G) =\begin{cases}   x\,U(G/ e) &  \text{if $e$ is a bridge,} \\
 y\,U(G\backslash e ) &  \text{if $e$ is a loop,} \\
 a\,U(G\backslash e ) + b\,U(G/ e ) &  \text{if $e$ is ordinary edge,}\\
   \gamma^n &  \text{if } E(G)=\emptyset \text{ and } v(G)=n.
 \end{cases}
 \end{equation}
Moreover
 \begin{equation}\label{EMM:eq:U2}
      U(G)= \gamma^{k(G)}a^{n(G)}
          b^{r(G)}
               T\left( G;\frac{x}{b}, \frac{y}{a}\right) .
 \end{equation}
\end{theorem}
For us, this is the salient feature of the Tutte polynomial, and we take it as the fundamental definition. We believe this choice is uncontroversial, but highlight some interesting recent work of
A.~Goodall, T.~Krajewski, G.~Regts and L.~Vena \cite{Goo} in which they defined a polynomial of graphs on surfaces as an amalgamation of a flow polynomial and tension polynomial for graphs on surfaces.

Before moving on, let us comment that we will meet generalisations of the Tutte polynomial in  universal forms, i.e., in a form analogous to  $U(G)$ of \eqref{EMM:eq:U}. In these we will be able to spot that we can reduce the number of variables to obtain an analogue of $T(G)$, but there will be choices in how this can be done. We will make such choices in a way that results in the polynomial having the cleanest duality relation, i.e., one that is closest to that for the Tutte polynomial, which states that for a plane graph $G$, 
 \begin{equation}\label{eld}
               T\left( G^*;x,y\right) =  T\left( G;y,x\right)  .
 \end{equation}

\subsection{Extending the Tutte polynomial}\label{s.2}

Our interest here is in extending the definition of the Tutte polynomial from graphs to other classes of combinatorial objects. Specifically here we will consider graphs on surfaces, although the general theory we describe does extend to other settings.

The definition of a `Tutte polynomial' requires three things:
\begin{enumerate}[label={T\arabic*}]
    \item \label{t1} A class of objects. (We are constructing a Tutte polynomial for this class.) 
    \item \label{t2} A notion of deletion and contraction for this class. The class must be closed under these operations, and we require that every object can be reduced to a trivial object (here edgeless graphs) using them. 
    \item \label{t3} A canonical way to fix  the cases of the deletion-contraction definition. (That is, a canonical way to determine the analogues of bridges, loops, and ordinary edges.)
\end{enumerate}

Although our procedures here apply more generally (see Remark~\ref{rem1} on page~\pageref{rem1}) we restrict our enquiries to polynomials of graphs on surfaces. So far in this discussion we have been intentionally vague about what we mean when we say `graphs on surfaces'. Our reason for doing this is that exactly what we mean by the phrase is highly dependent upon our choice of deletion and contraction, and so answers to \ref{t1} and \ref{t2} are highly dependent upon each other. For example, consider the torus with a graph drawn on it consisting of one vertex and two loops, a meridian and a longitude. If we delete the longitude, should the result be a single loop not cellularly embedded on the torus, or a single loop cellularly embedded on the sphere? (In the latter case we would have removed a handle as well as the edge it carried.)  
We now move to the problem of making precise what we mean by `graphs on surfaces', obtaining suitable constructions to satisfy \ref{t1} and \ref{t2}.

\section{Topological graphs and their minors}\label{s.3}

\subsection{A review of the topology of surfaces}

A \emph{surface} $\Sigma$ is a compact topological space in which distinct points have distinct neighbourhoods, and each point has a neighbourhood homeomorphic to an open disc in $\R^{2}$. Surfaces need not be connected. 
If the connected surface $\Sigma$ is orientable, then it is homeomorphic to a sphere or the connected sum of tori. If it is not orientable, then it is homeomorphic to  the connected sum of real projective planes.

We will also need \emph{surfaces with boundary}, which are surfaces except that they also have some points---the boundary points---all of whose neighbourhoods are homeomorphic to half of an open disc $\{(x,y)\in\R^{2}:x^{2}+y^{2}<1,x\geq 0\}$. Each component of the boundary of a surface is homeomorphic to a circle. Given a surface with boundary $\Sigma'$, we can obtain a surface $\Sigma$ by `capping' each boundary component. This just means identifying each of the circles with the boundaries of (disjoint) closed discs. Then the genus of $\Sigma'$ is defined to be that of $\Sigma$.

The number $n$ of tori or real projective planes is called the \emph{genus} of the surface. The genus of the sphere is zero. Together, genus, orientability, and number of boundary components completely classify connected surfaces with boundary.

Surfaces can be thought of as spheres with $n$ handles. Here a \emph{handle} is an annulus $S^1\times I$, where $S^1$ is a circle and $I$ is the unit interval. By adding a handle to a surface $\Sigma$, we mean that we remove the interiors of two disjoint discs from $\Sigma$, and identify each resulting boundary component  with a distinct boundary component of $S^1\times I$.  Adding a handle to $\Sigma$ yields its connected sum with either a torus or a Klein bottle, depending upon how the handle is attached. The inverse process is removing a handle.

\subsection{Graphs on surfaces and their generalizations}

\begin{definition}
A \emph{graph $G$ on a surface} $\Sigma$ consists of a set $V$ of points on $\Sigma$ and another set $E$ of simple paths joining these points and only intersecting each other at the points.

We say that $G\subset\Sigma$ is an \emph{embedding} of the abstract graph $G=(V,E)$, whose incidence relation comes from the paths in the obvious way.
\end{definition}

\begin{definition}
The graph $G\subset\Sigma$ is \emph{cellularly embedded} if $\Sigma\setminus G$ consists of discs.
\end{definition}

\begin{samepage}
\begin{definition}
Two embedded graphs $G\subset\Sigma$ and $G'\subset\Sigma'$ are \emph{equivalent} if there is a homeomorphism from $\Sigma$ to $\Sigma'$ inducing an isomorphism between $G$ and $G'$. When the surfaces are orientable this homeomorphism should be orientation preserving. 
\end{definition}
\end{samepage}

We will consider all graphical objects up to equivalence. (Note that a given graph will in general have many inequivalent embeddings.)

Topological graph theory is mostly (but not exclusively) concerned with cellularly embedded graphs. However, we will see that we  have to relax this restriction when we consider deletion and contraction.

Let $G$ be a graph cellularly embedded in the surface $\Sigma$, with $e\in E(G)$. We want to define $(G\subset\Sigma)\ba e$ in the natural way to be the result of removing the path $e$ from the graph (but leaving the surface unchanged). The difficulty is that it may result in a graph which is not cellularly embedded (compare Figures \ref{del.a1}, \ref{del.b1} and \ref{del.a2}, \ref{del.b2}).

There is a choice:
\begin{enumerate}[label={D\arabic*}]
    \item abandon the cellular embedding condition, or
    \item `stabilize' by removing the redundant handle as in Figure~\ref{del2}. (We make this term precise later in Definition~\ref{hfdjskfbu}.)
\end{enumerate}

Contraction leads to another dichotomy. We want to define $(G\subset\Sigma)/e$ as the image of $G\subset\Sigma$ under the formation of the topological quotient $\Sigma/e$ which identifies the path $e$ to a point, this point being a new vertex. If we start with a graph embedded (cellularly or not) in a surface then sometimes we obtain another graph embedded in a surface (as in  Figures~\ref{del.a1} and \ref{cont.b1}), but sometimes pinch points are created (as in Figures~\ref{cont.a2} and~\ref{cont.b2}).

Again we have a choice:
\begin{enumerate}[label={C\arabic*}]
    \item allow pinch points and work with pseudo-surfaces, or
    \item `resolve' pinch points as in Figure~\ref{cont.c2}. (We make this term precise below.)
\end{enumerate}

Here a \emph{pseudo-surface} is the result of taking topological quotients by a finite number of paths in a surface. A pseudo-surface may have \emph{pinch points}, i.e., those having neighbourhoods not homeomorphic to discs. If there are no pinch points then the pseudo-surface is a surface. A \emph{graph on a pseudo-surface} is the result of taking topological quotients by a finite number of edge-paths, starting with a graph on a surface. Note that pinch points are always vertices of such a graph.

By \emph{resolving a pinch point} we mean the result of the following process. Delete a small neighbourhood of the pinch point. This creates a number of boundary components. Next, by forming the topological quotient space, shrink each boundary component to a point and make this point a vertex. Note that resolving a pinch point in a graph embedded in a pseudo-surface results in another graph embedded in a pseudo-surface, but these may have different underlying abstract graphs.

\begin{definition}
Let $G$ be a graph embedded in a pseudo-surface $\check{\Sigma}$. Its \emph{regions} are subsets of the pseudo-surface corresponding to the components of its complement, $\check{\Sigma} \ba G$. If all of the regions are homeomorphic to discs we say that the graph is \emph{cellularly embedded} in a pseudo-surface and call its regions \emph{faces}. This terminology applies to graphs in surfaces since a surface is a special type of pseudo-surface.
\end{definition}

\begin{definition}
Two graphs embedded in pseudo-surfaces are \emph{equivalent} if there is a homeomorphism from one pseudo-surface to the other inducing an isomorphism between the graphs. When the pseudo-surfaces are orientable this homeomorphism should be orientation preserving.
\end{definition}

\begin{figure}[ht]
\centering
\subfigure[$G\subset\Sigma$]{
\labellist
\small\hair 2pt
\pinlabel $e$ at 66 50
\endlabellist
\includegraphics[width=40mm]{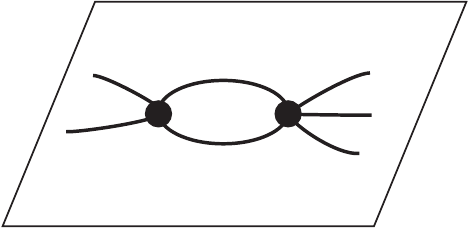}
\label{del.a1}
}
\hspace{1cm}
\subfigure[$G\ba e\subset\Sigma$]{
\includegraphics[width=40mm]{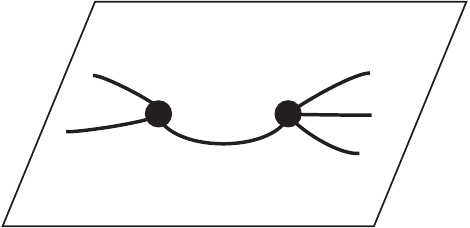}
\label{del.b1}
}
\hspace{1cm}
\subfigure[$G/e\subset\Sigma/e$]{
\includegraphics[width=40mm]{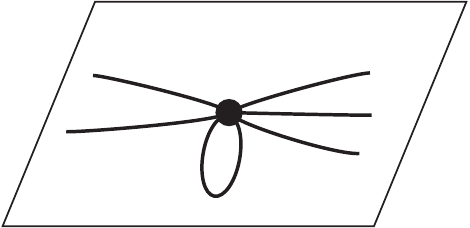}
\label{cont.b1}
}
\caption{Deletion and contraction preserving cellular embedding}
\label{del1}
\end{figure}

\begin{figure}[ht]
\centering
\subfigure[$G\subset\Sigma$]{
\labellist
\small\hair 2pt
\pinlabel $e$ at 35 28
\endlabellist
\includegraphics[width=40mm]{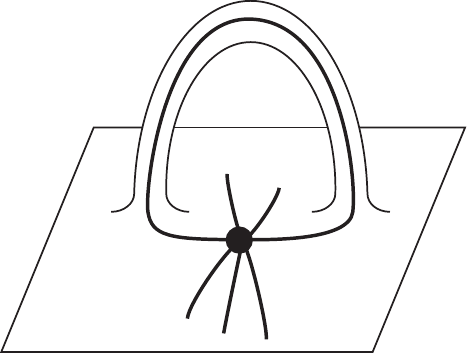}
\label{del.a2}
}
\hspace{1cm}
\subfigure[$G\ba e\subset\Sigma$]{
\includegraphics[width=40mm]{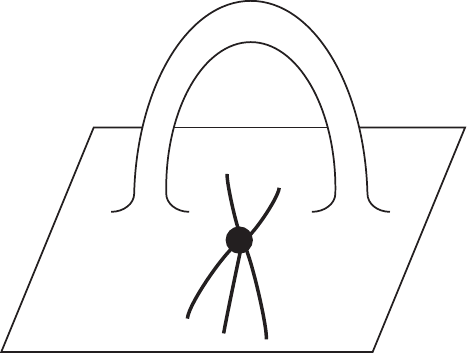}
\label{del.b2}
}
\hspace{1cm}
\subfigure[Removing the handle]{
\includegraphics[width=40mm]{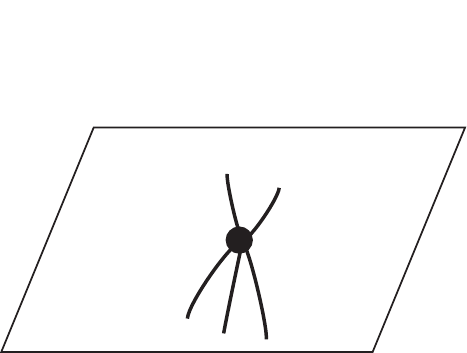}
\label{del.c2}
}
\caption{Deletion leaving a redundant handle}
\label{del2}
\end{figure}

\begin{figure}[ht]
\centering
\subfigure[$G\subset\Sigma$]{
\labellist
\small\hair 2pt
\pinlabel $e$ at 107 27
\endlabellist
\includegraphics[height=16mm]{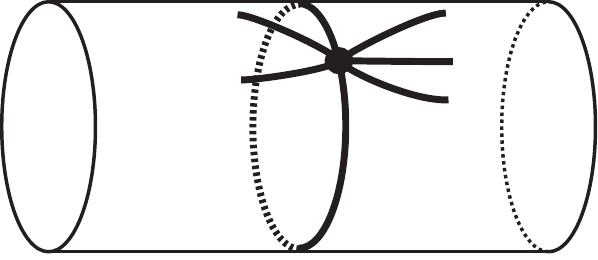}
\label{cont.a2}
}
\hspace{1cm}
\subfigure[$G/e\subset\Sigma/e$]{
\includegraphics[height=16mm]{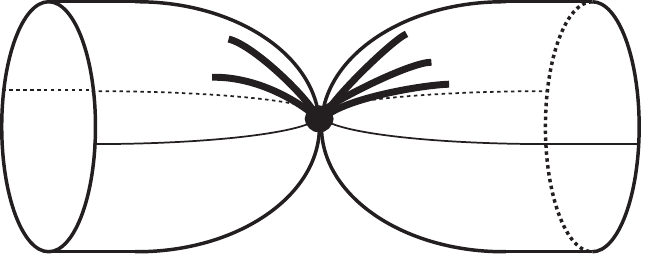}
\label{cont.b2}
}
\hspace{1cm}
\subfigure[Resolving the pinch point]{
\includegraphics[height=16mm]{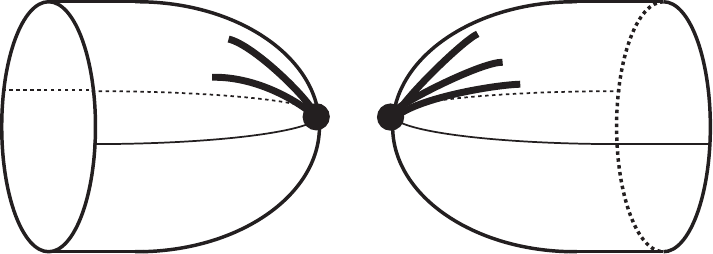}
\label{cont.c2}
}
\caption{Contraction making a pseudo-surface}
\label{cont2}
\end{figure}

Returning to the problem of constructing Tutte polynomials via deletion and contraction,  we see that in order to define a `Tutte polynomial of graphs on surfaces' we are immediately forced by \ref{t1} and \ref{t2} to consider four classes of objects, as follows:

\bigskip
\begin{center}
\begin{tabular}{l|l|l}
 & D1 & D2 \\ \hline
 C1 & graphs embedded in pseudo-surfaces & graphs cellularly embedded \\
    & (need not be cellular) & in pseudo-surfaces \\ \hline
 C2 & graphs embedded in surfaces & graphs cellularly embedded \\
    & (need not be cellular) & on surfaces\\
\end{tabular}
\end{center}
\bigskip

Having established these four cases, it is now convenient for our purposes to switch to the formalism of ribbon graphs.

\subsection{Ribbon graphs}\label{s.4}

A {\em ribbon graph} $\G=\left(V(\G),E(\G)\right)$ is a surface with boundary, represented as the union of two sets of discs---a set $V(\G)$ of {\em vertices} and a set $E(\G)$ of {\em edges}---such that: (1) the vertices and edges intersect in disjoint line segments; (2) each such line segment lies on the boundary of precisely one vertex and precisely one edge; and (3) every edge contains exactly two such line segments. 

We let $F(\G)$ denote the set of boundary components of a ribbon graph $\G$. 

Two ribbon graphs $\G$ and $\G'$ are \emph{equivalent} is there is a homeomorphism from $\G$ to $\G'$ (orientation preserving when $\G$ is orientable) mapping $V(\G)$ to $V(\G')$ and $E(\G)$ to $E(\G')$. In particular, the homeomorphism preserves the cyclic order of half-edges at each vertex.

Let $\G$ be a ribbon graph and $e\in E(\G)$. Then $\G\ba e$ denotes the ribbon graph obtained from $\G$ by \emph{deleting} the edge $e$.
If $u$ and $v$ are the (not necessarily distinct) vertices incident with $e$, then $\G/e$ denotes the ribbon graph obtained as follows: consider the boundary component(s) of $e\cup\{u,v\}$ as curves on $\G$. For each resulting curve, attach a disc (which will form a vertex of $\G/e$) by identifying its boundary component with the curve. Delete $e$, $u$ and $v$ from the resulting complex, to get the ribbon graph $\G/e$. We say $\G/e$ is obtained from $\G$ by {\em contracting} $e$. See Figure~\ref{tablecontractrg} for the local effect of contracting an edge of a ribbon graph.

\begin{table}[ht]
\centering
\begin{tabular}{|c||c|c|c|}\hline
& non-loop & non-orientable loop & orientable loop \\ \hline
\raisebox{6mm}{$\G$} 
&\includegraphics[scale=.25]{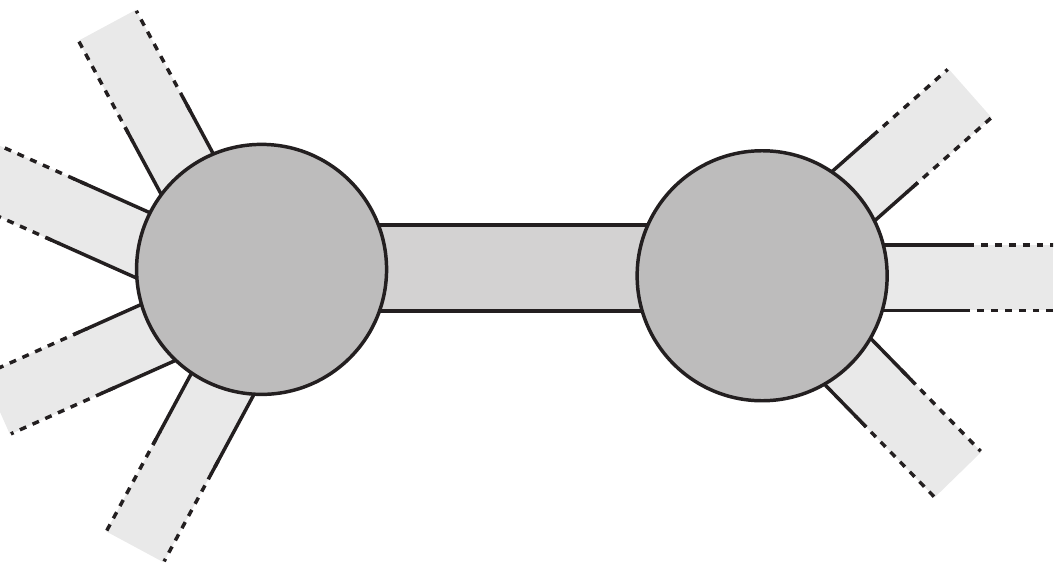} &\includegraphics[scale=.25]{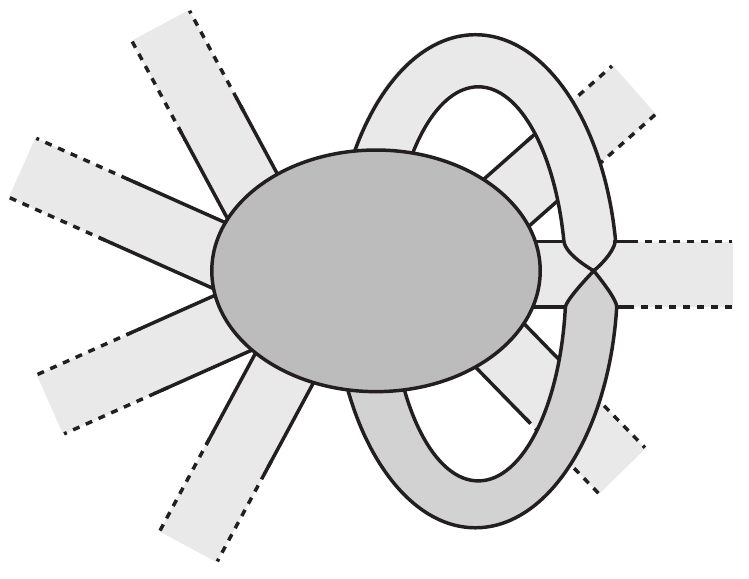} &\includegraphics[scale=.25]{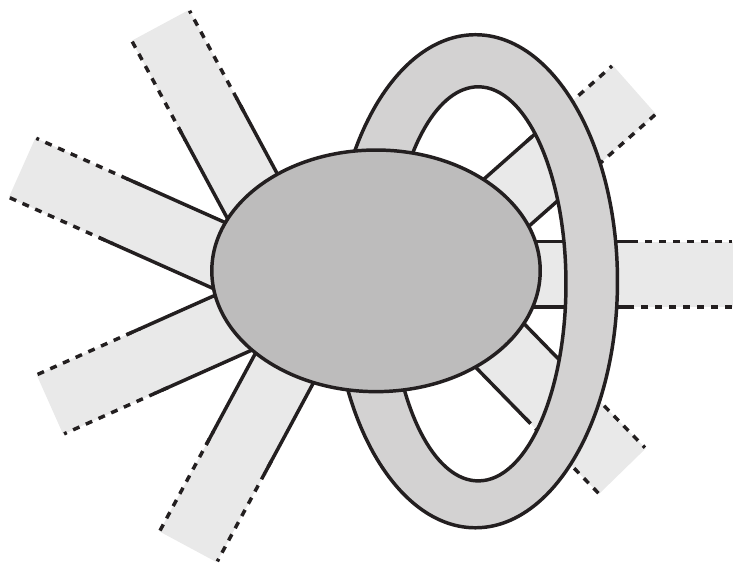}
\\ \hline
\raisebox{6mm}{$\G/e$} 
&\includegraphics[scale=.25]{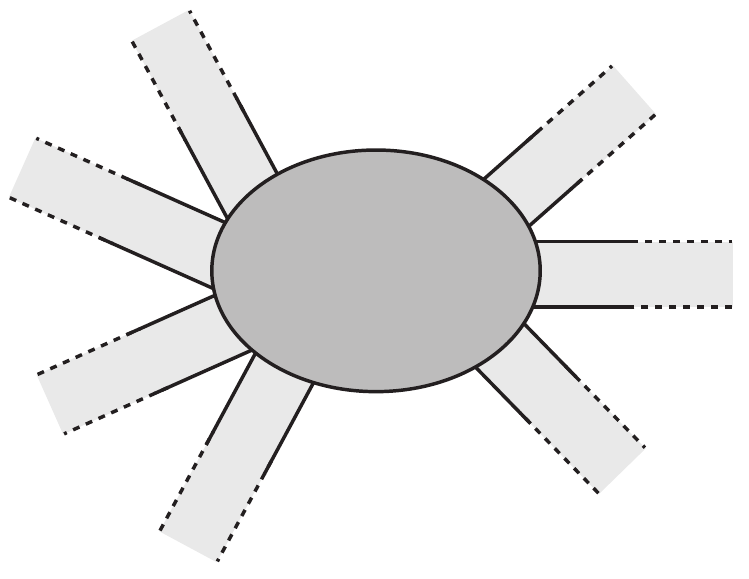} &\includegraphics[scale=.25]{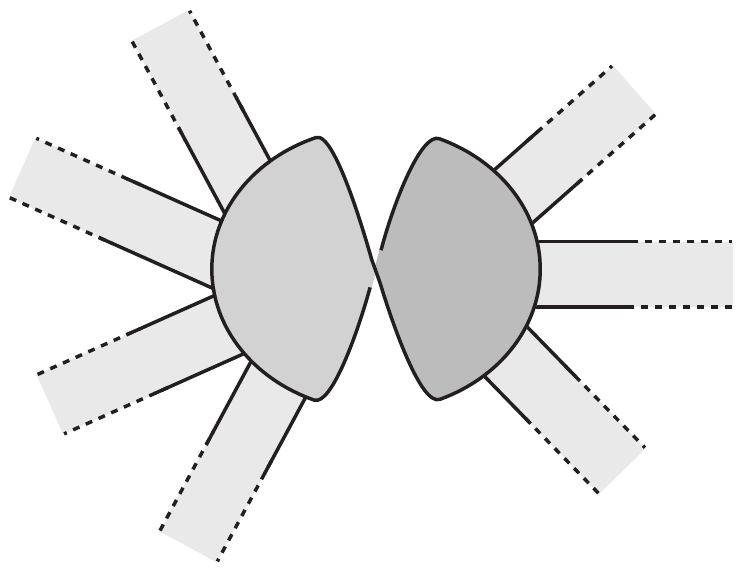}&\includegraphics[scale=.25]{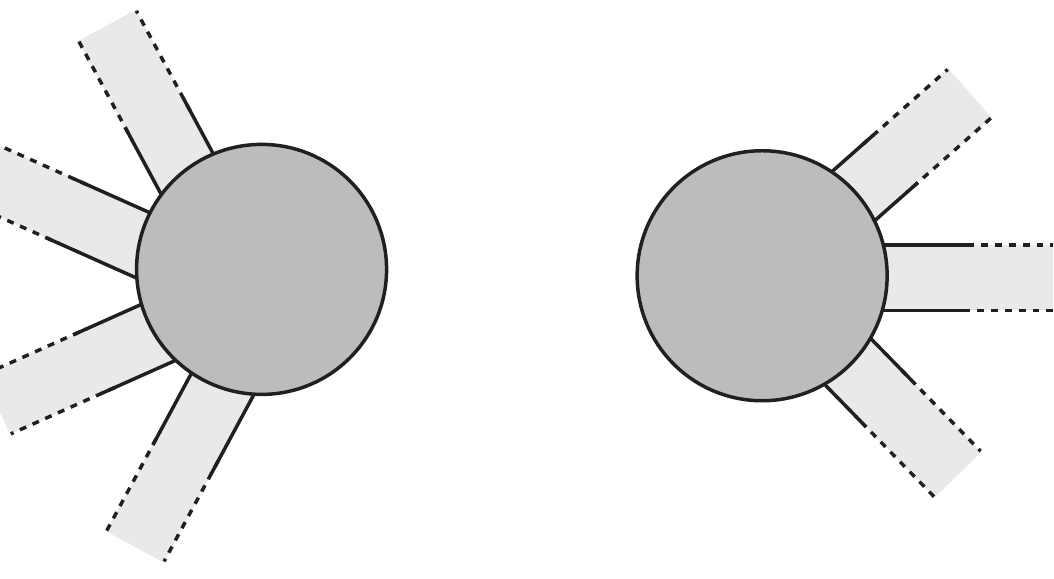} \\ \hline
\end{tabular}
\caption{Contracting an edge of a ribbon graph.}
\label{tablecontractrg}
\end{table}

\begin{definition}
A \emph{vertex colouring} of a ribbon graph $\G$ is a mapping from $V(\G)$ to a colouring set. Equivalently, it is a partition of $V(\G)$ into \emph{colour classes}. The colour class of the vertex $v$ is denoted $[v]_{\G}$.
\end{definition}

\begin{definition}
Two vertex-coloured ribbon graphs $\G$ and $\G'$ are \emph{equivalent} if they are equivalent as ribbon graphs, with the mapping $V(\G)\rightarrow V(\G')$ preserving (vertex) colour classes.
\end{definition}

Now we define deletion and contraction for vertex-coloured ribbon graphs. In fact, deletion is clear. For contraction, of an edge $e$, suppose that $e=(u,v)$ with colour classes $[u]_{\G},[v]_{\G}$. We obtain the ribbon graph $\G/e$, with colour classes determined as follows.
\begin{itemize}
    \item If the contraction does not change the number of vertices and creates a vertex $p$ (in which case $u=v$), then   
    $$[p]_{\G/e}=[u]_{\G}\cup\{p\}\setminus\{u\}= [v]_{\G}\cup\{p\}\setminus\{v\}.$$
    \item If the contraction merges $u$ and $v$ into a single vertex $p$, then
    $$[p]_{\G/e}=[u]_{\G}\cup [v]_{\G}\cup\{p\}\setminus\{u,v\}.$$
    \item If the contraction creates vertices $p,q$ (in which case $u=v$), then
    \begin{eqnarray*}[p]_{\G/e}=[q]_{\G/e}&=&[u]_{\G}\cup\{p,q\}\setminus\{u\} \\
    &=&[v]_{\G}\cup\{p,q\}\setminus\{v\}.
    \end{eqnarray*}
\end{itemize}
The local effect of contraction on vertex colour classes is shown in Table~\ref{tablevp}. 
 Note that in the case in which  contraction merges two colour classes, the effect on the graph is global in the sense that 
 all vertices in those two colour classes now belong  to a single colour class.

\begin{table}[ht]
\centering
\begin{tabular}{|c||c|c|c|}\hline
\raisebox{6mm}{$\G$} 
&
\includegraphics[scale=.25]{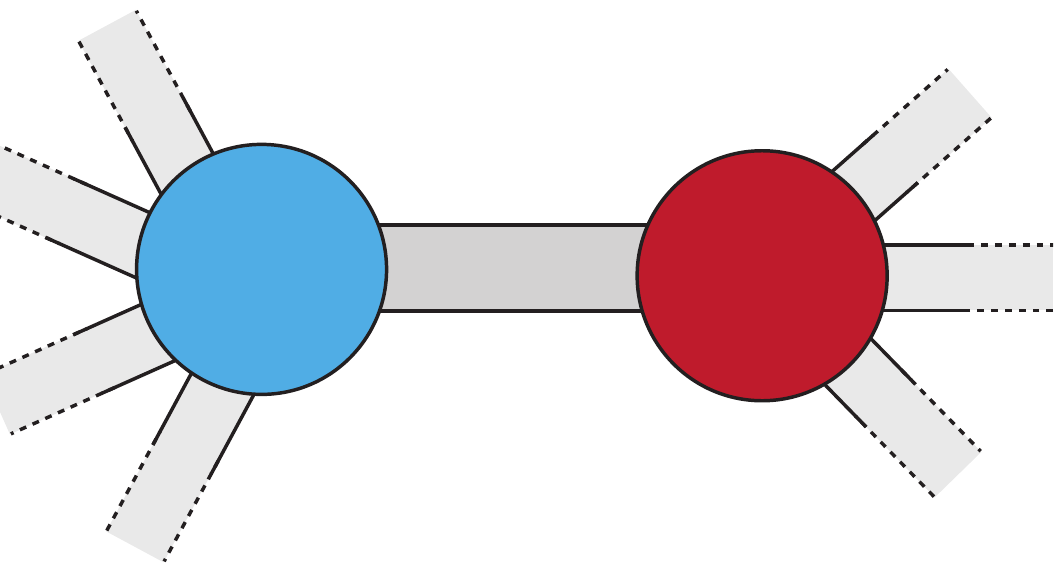} &
\includegraphics[scale=.25]{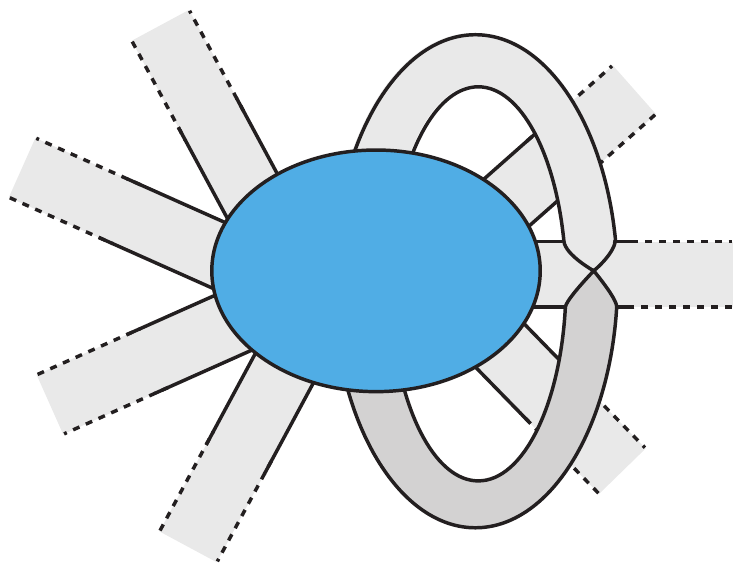} &
\includegraphics[scale=.25]{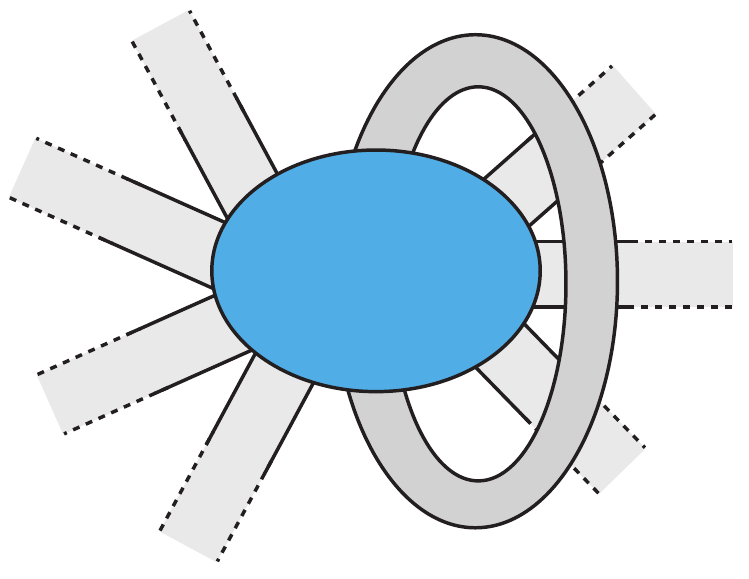}
\\ \hline
\raisebox{6mm}{$\G/e$} 
&\includegraphics[scale=.25]{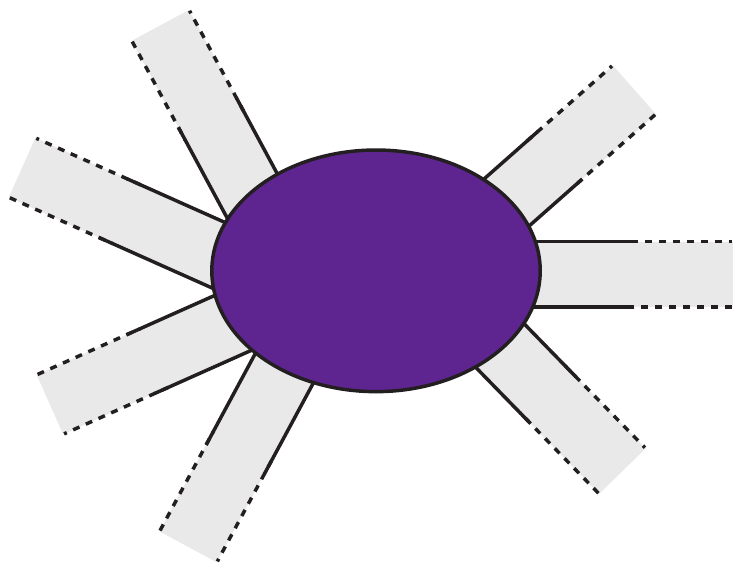} &
\includegraphics[scale=.25]{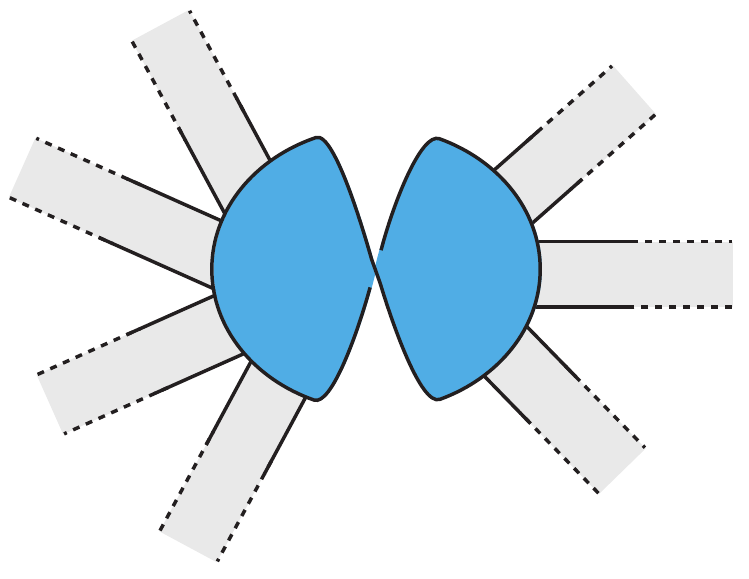}&
\includegraphics[scale=.25]{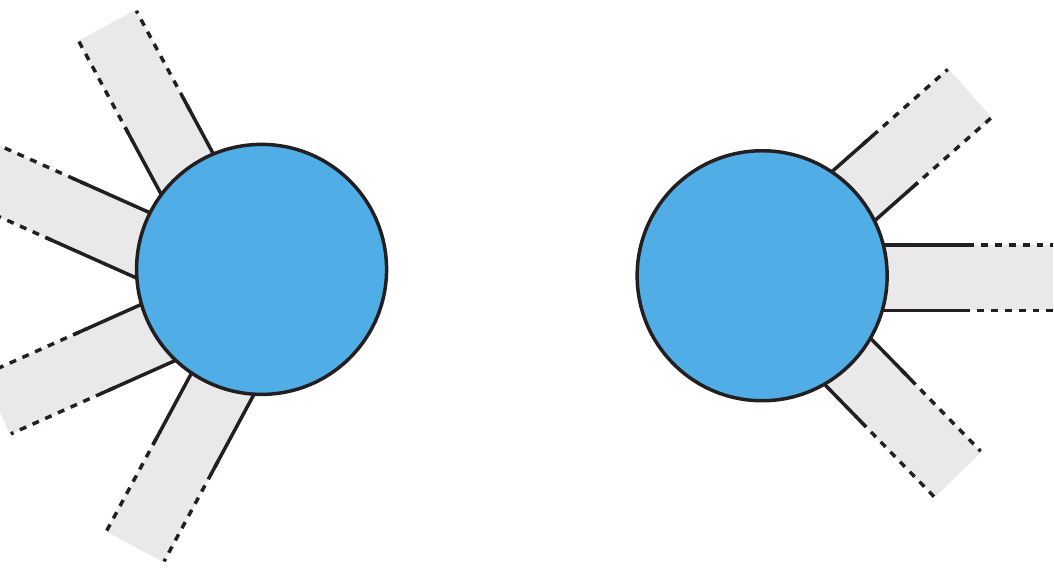} \\ \hline
\end{tabular}
\caption{Contracting an edge in a vertex-coloured ribbon graph}
\label{tablevp}
\end{table}

\begin{definition}
A \emph{boundary colouring} of a ribbon graph $\G$ is a mapping from $F(\G)$, the set of boundary components, to a colouring set. Equivalently, it is a partition of $F(\G)$ into \emph{colour classes}. The colour class of the boundary component $b$ is denoted $[b]^*_{\G}$.
\end{definition}

\begin{definition}
Two boundary-coloured ribbon graphs $\G$ and $\G'$ are \emph{equivalent} if they are {equivalent} as ribbon graphs, with the induced mapping $F(\G)\rightarrow F(\G')$ preserving (boundary) colour classes.
\end{definition}

Now we define deletion and contraction for boundary-coloured ribbon graphs. This time, contraction is clear, since it does not change the number of boundary components of a ribbon graph (as can be seen from Table~\ref{tablecontractrg}).

For deletion, of an edge $e$, suppose that $a$ and $b$ are the boundary components touching $e$ with colour classes $[a]^*_{\G},[b]^*_{\G}$. We obtain the ribbon graph $\G/e$, with (boundary) colour classes determined as follows.
\begin{itemize}
    \item If the deletion does not change the number of boundary components and creates a boundary component $r$ (in which case $a=b$), then   
    $$[r]^*_{\G\ba e}=[a]^*_{\G}\cup\{r\}\setminus\{a\} = [b]^*_{\G}\cup\{r\}\setminus\{b\} .$$
    \item If the deletion merges $a$ and $b$ into a single boundary component $r$, then
    $$[r]^*_{\G\ba e}=[a]^*_{\G}\cup [b]^*_{\G}\cup\{r\}\setminus\{a,b\}.$$
    \item If the deletion creates boundary components $r$ and $s$ (in which case $a=b$), then 
    \begin{eqnarray*}[r]^*_{\G\ba e}=[s]^*_{\G\ba e}&=&[a]^*_{\G}\cup\{r,s\}\setminus\{a\} \\
    &=&[b]^*_{\G}\cup\{r,s\}\setminus\{b\}.
    \end{eqnarray*}
\end{itemize}
The local effect of deletion on boundary colour classes is shown in Table~\ref{tablefp}.
Note that in the case in which deletion merges two colour classes, the effect on the graph is global in the sense that 
 all boundary components in those two colour classes now belong to a single colour class.

\begin{table}[ht]
\centering
\begin{tabular}{|c||c|c|c|}\hline
\raisebox{6mm}{$\G$} 
&\includegraphics[scale=.25]{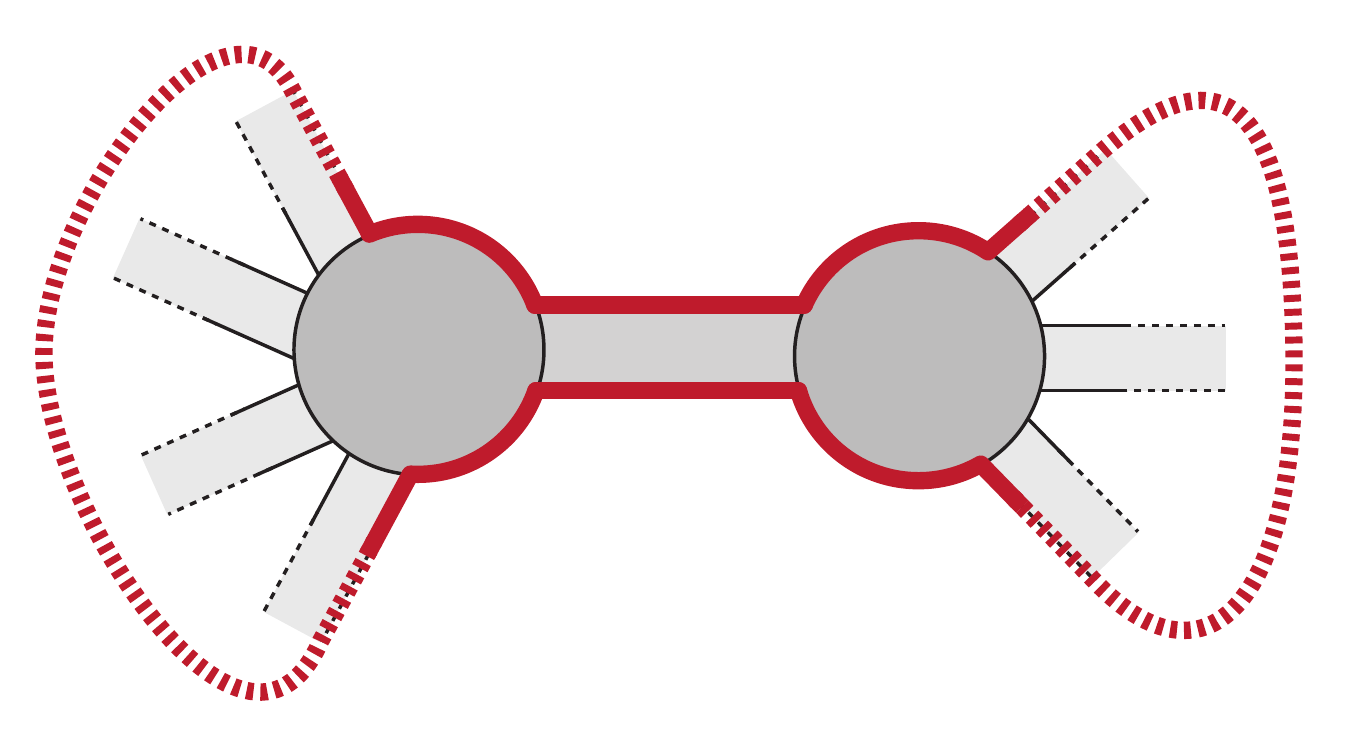} &\includegraphics[scale=.25]{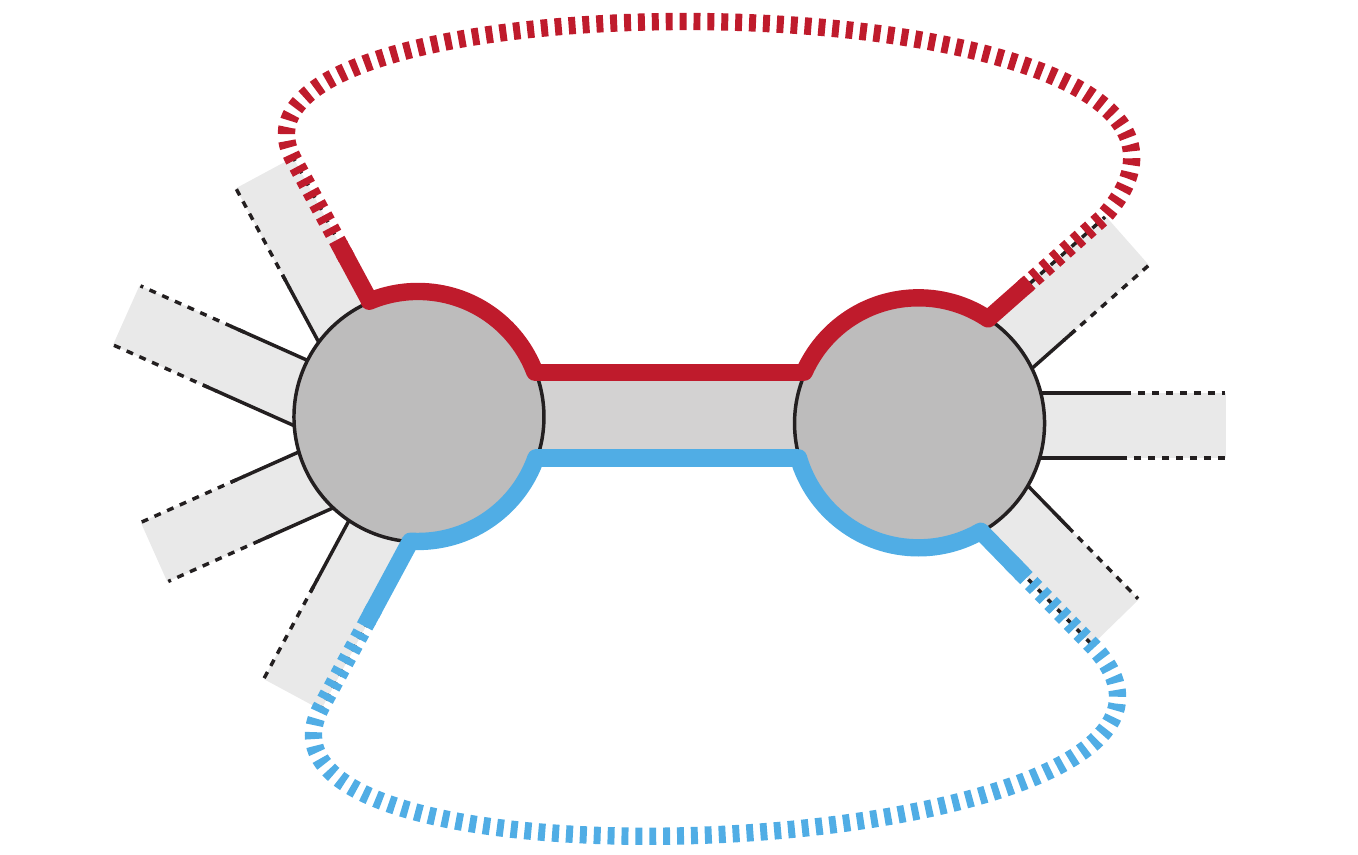} &\includegraphics[scale=.25]{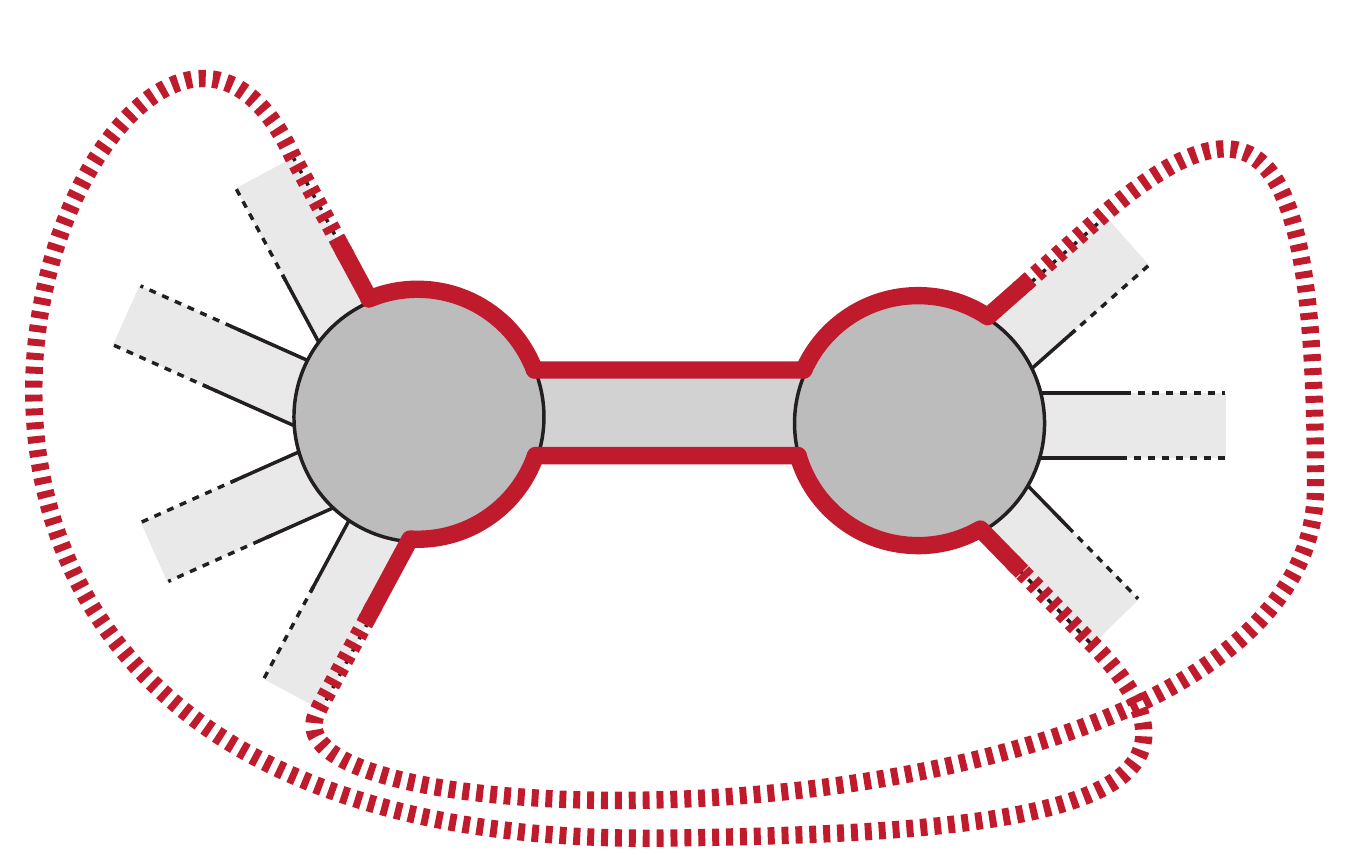}
\\ \hline
\raisebox{6mm}{$\G\ba e$} 
&\includegraphics[scale=.25]{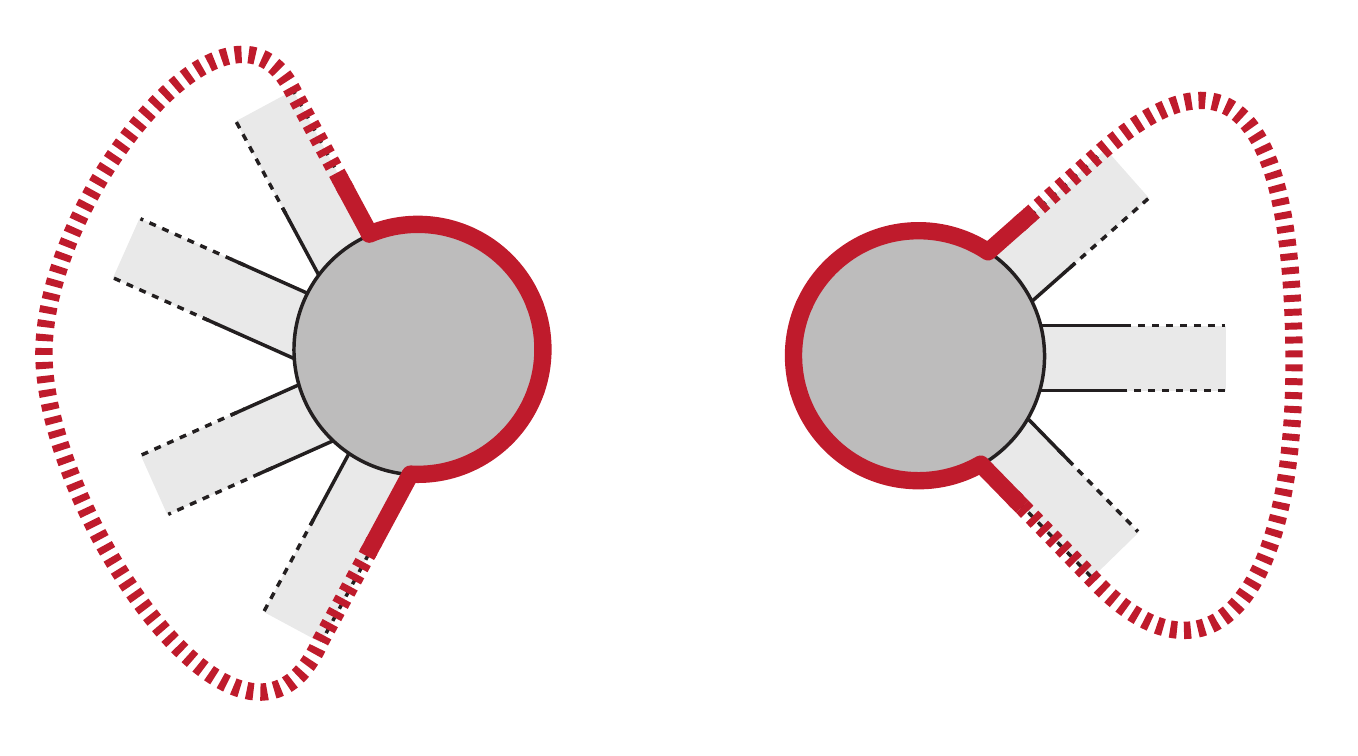} &\includegraphics[scale=.25]{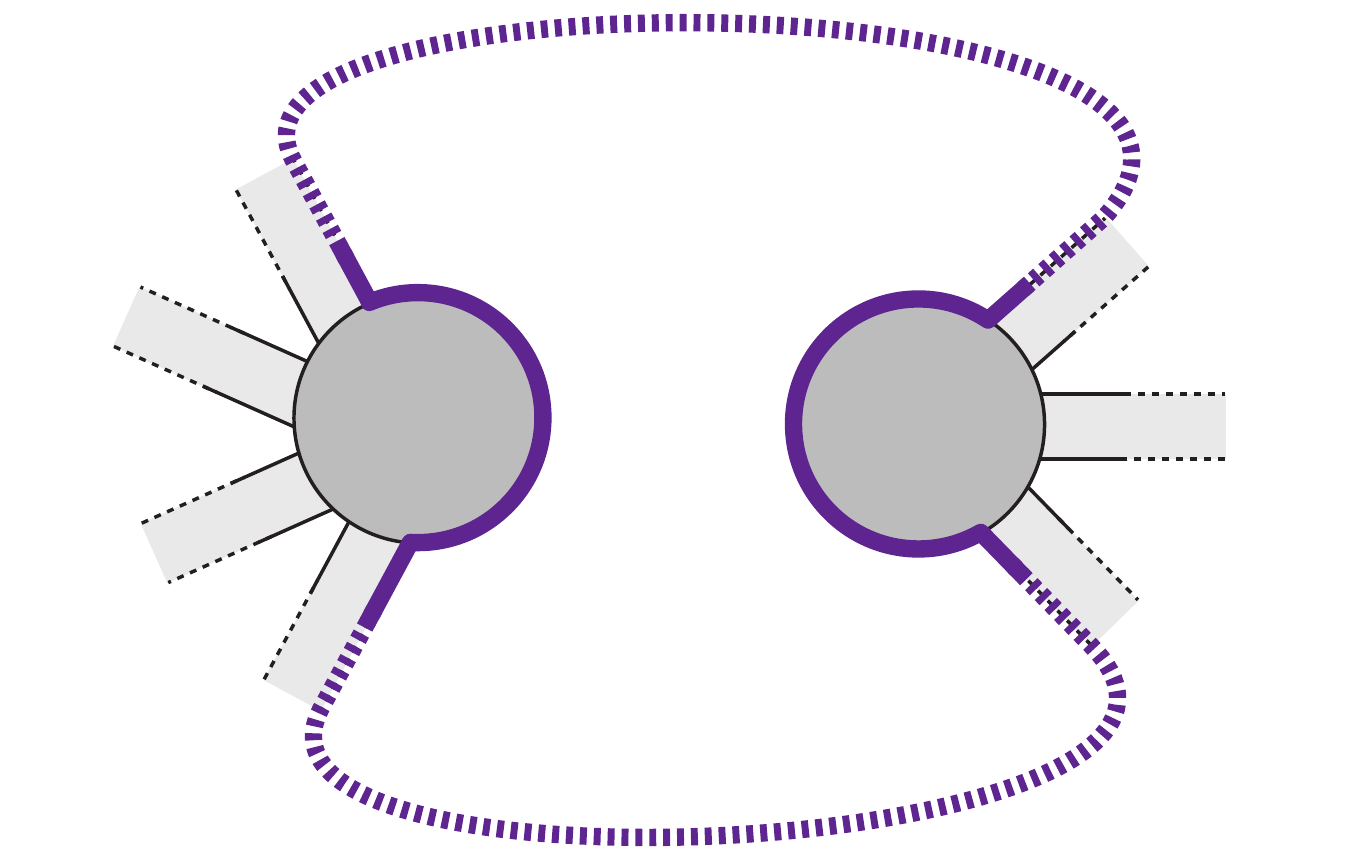}&\includegraphics[scale=.25]{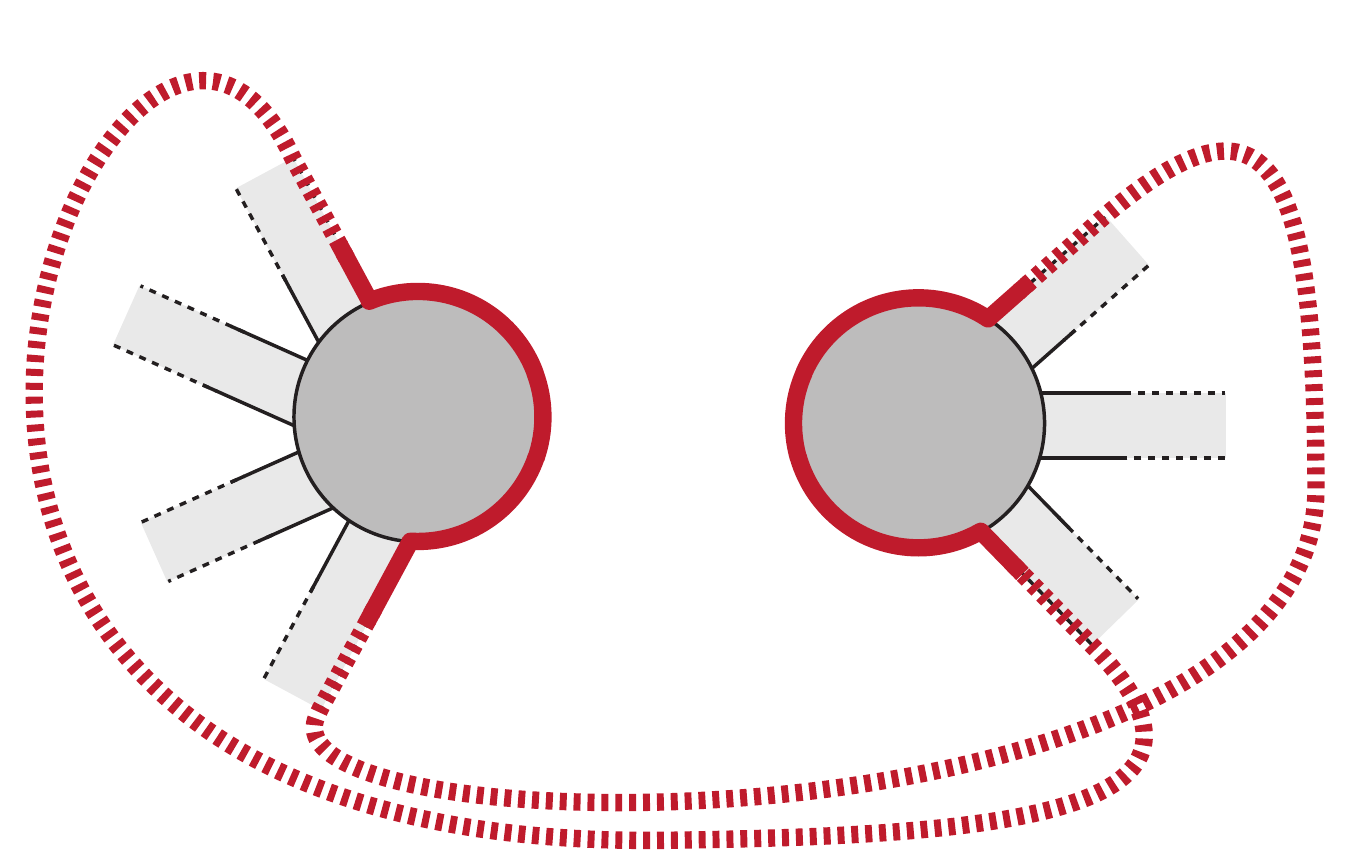} \\ \hline
\end{tabular}
\caption{Deleting an edge in a boundary-coloured ribbon graph.}
\label{tablefp}
\end{table}

\begin{definition}
A \emph{coloured  ribbon graph} $\G$ is a ribbon graph that is simultaneously vertex coloured and boundary coloured. 
\end{definition}

\begin{definition}
Two coloured ribbon graphs are \emph{equivalent} if they are equivalent as both vertex coloured ribbon graph and boundary coloured ribbon graphs.
\end{definition}

\begin{definition}
If $\G$ is a  coloured ribbon graph with vertex colour classes $\{[v]_\G\}_{v\in V(\G)}$ and boundary colour classes $\{[b]^*_\G\}_{b\in F(\G)}$ then, for an edge $e$ of $\G$:
\begin{enumerate}
    \item $\,\G$ with $e$ \emph{deleted}, written $\G\ba e$, is the ribbon graph $\G\ba e$ with vertex colour classes $\{[v]_{\G\ba e}\}_{v\in V(\G\ba e)}$ and boundary colour classes $\{[b]^*_{\G \ba e}\}_{b\in F(\G\ba e)}$; and
    \item $\,\G$ with $e$ \emph{contracted}, written $\G/ e$, is the ribbon graph $\G/e$ with vertex colour classes $\{[v]_{\G/e}\}_{v\in V(\G/e)}$ and boundary colour classes $\{[b]^*_{\G/e}\}_{b\in F(\G/e)}$.
    \end{enumerate}
\end{definition}

It is well-known that ribbon graphs are equivalent to cellularly embedded graphs in surfaces. (Ribbon graphs arise naturally from neighbourhoods of cellularly embedded graphs. On the other hand, topologically a ribbon graph is a  surface with boundary, and capping the holes gives rise to a cellularly embedded graph in the obvious way. See \cite{EMMbook,GT87} for details.)

A graph $G$ embedded in a pseudo-surface $\check{\Sigma}$ gives rise to a unique coloured ribbon graph as follows. Firstly, resolve all the pinch points to obtain a graph embedded on a surface. Then take a neighbourhood of this graph in the surface to obtain a ribbon graph. For the colour classes, go back to $G$ and assign a distinct colour to each vertex and a distinct colour to each region. Then given two vertices $u,v$ in the ribbon graph, $[u]=[v]$ if and only if $u$ and $v$ arose from the same pinch point in $\check{\Sigma}$. Finally, the boundary colours in the ribbon graph come from the colouring of the regions in the embedding of $G$ in $\check{\Sigma}$.

\begin{figure}[ht]
\centering
\subfigure[A graph in a pseudo-surface]{
\includegraphics[height=40mm]{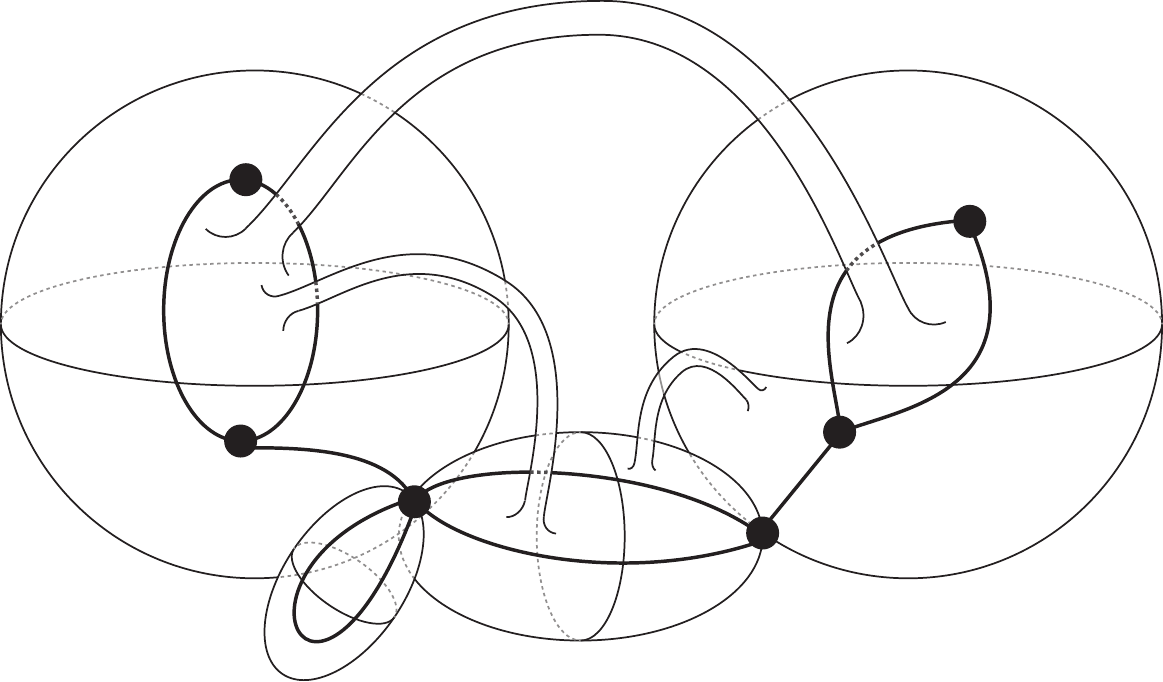}
}
\hspace{10mm}
\subfigure[The corresponding coloured ribbon graph]{
\includegraphics[height=40mm]{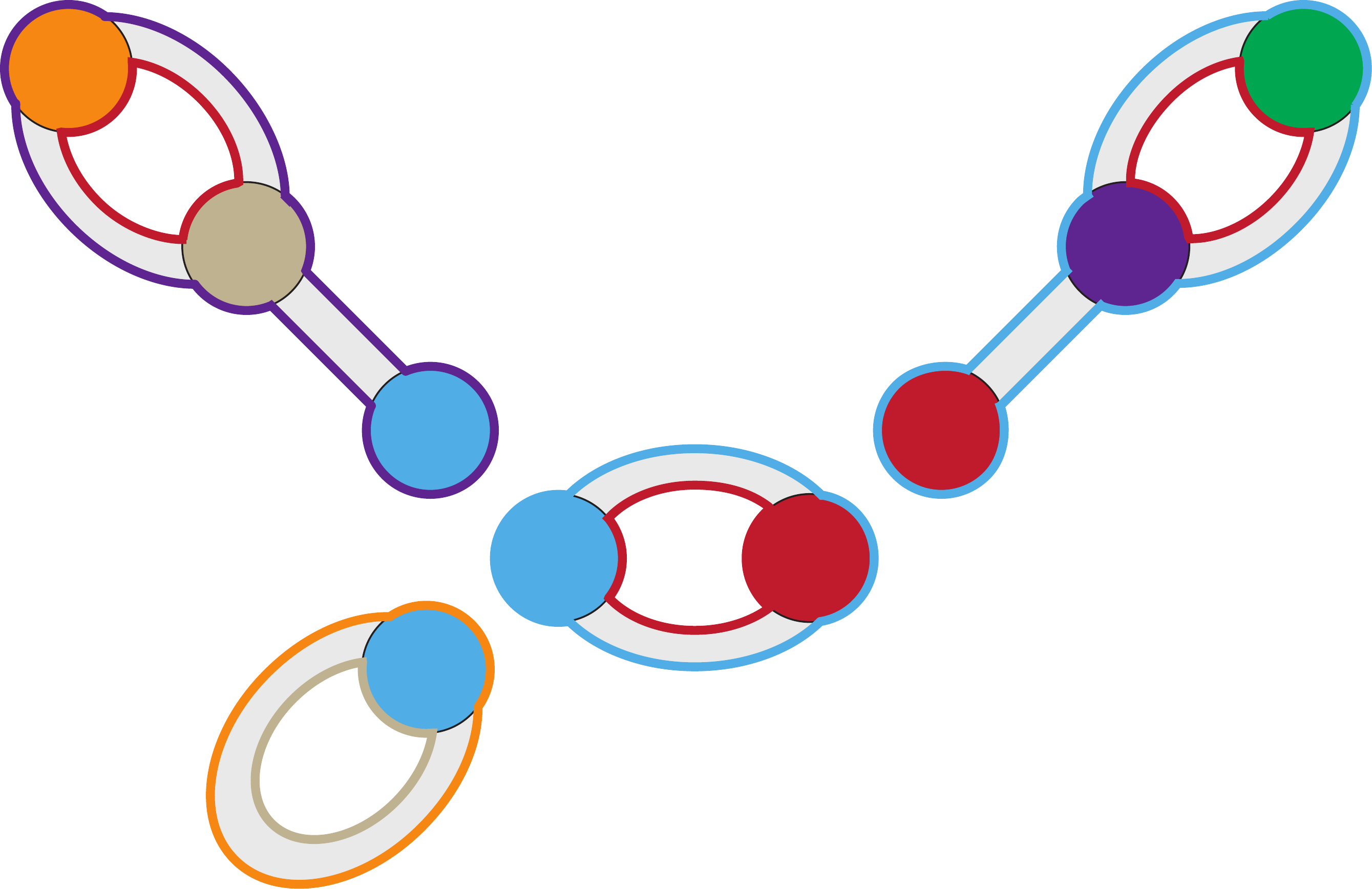}
}
\caption{Moving from a graph in a pseudo-surface to a coloured ribbon graph}
\label{ps2rg}
\end{figure}

On the other hand, given a coloured ribbon graph we can recover a graph embedded on a pseudo-surface as follows. The ribbon graph can be thought of as a graph cellularly embedded on a surface, as usual. The ribbon graph's vertex and boundary colourings give colourings of the vertices and faces of this cellularly embedded graph. Now identify vertices in the same colour class, to obtain a pseudo-surface, and add exactly one handle between each pair of faces of the same colour, thus spoiling the cellular embedding.  
(We note that Theorem~\ref{dh} below will allow us some flexibility in the exact way that handles are added. All that will matter is that handles are added in a way that merges all faces in the same colour class into one.)

\begin{figure}[ht]
\centering
\subfigure[A coloured ribbon graph, $\G$]{
\includegraphics[width=40mm]{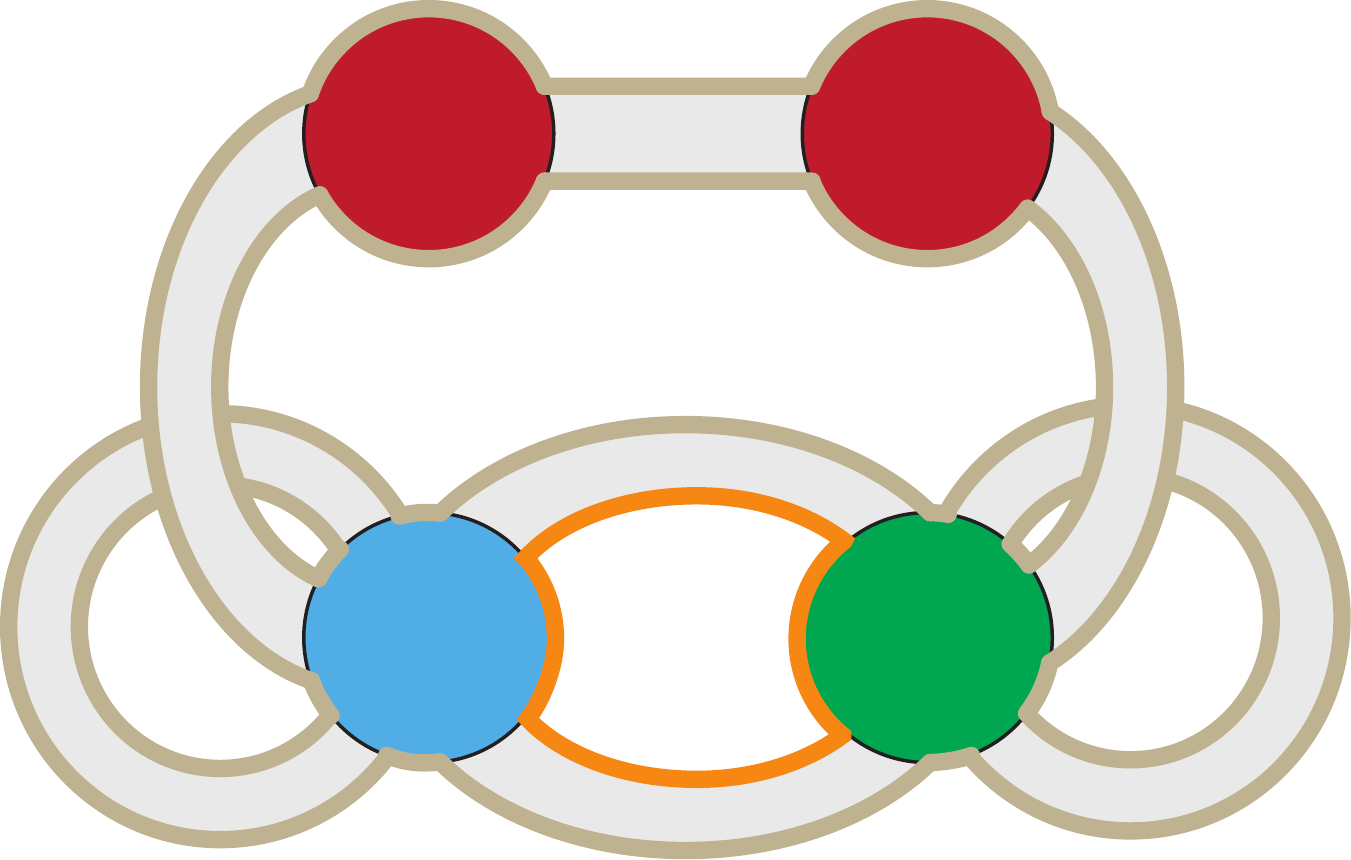}
}
\hspace{1cm}
\subfigure[$\G$ as a face-coloured cellularly embedded graph]{
\includegraphics[width=40mm]{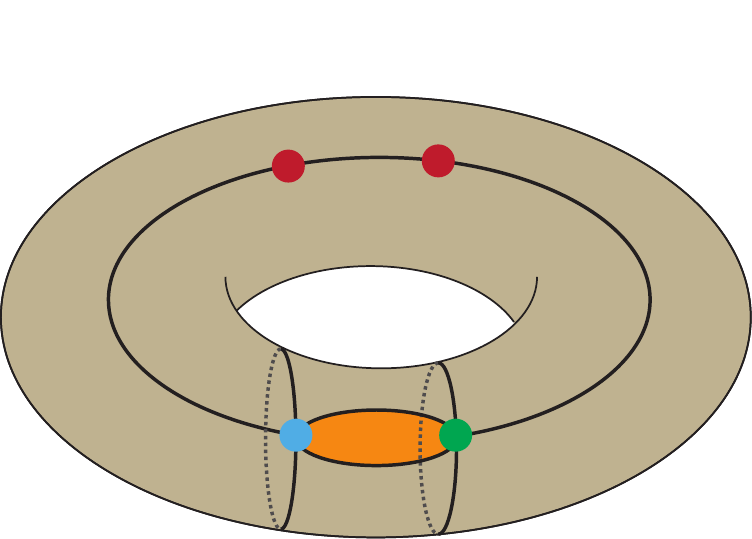}
}
\hspace{1cm}
\subfigure[A corresponding graph in a pseudo-surface]{
\includegraphics[width=40mm]{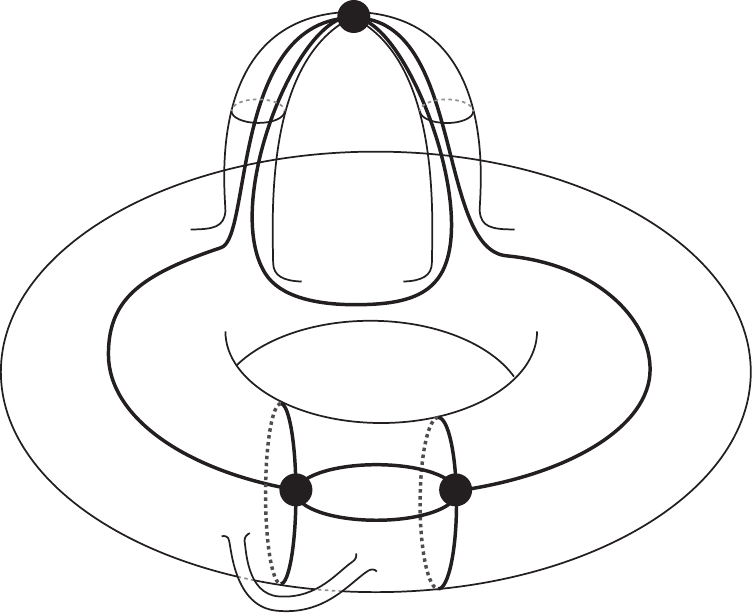}
}
\caption{Moving from a coloured ribbon graph to a graph in a pseudo-surface}
\label{rg2ps}
\end{figure}

Note that this relation between coloured ribbon graphs and graphs embedded on pseudo-surfaces is not a bijection. For example, a single vertex in a torus and a single vertex in the sphere are represented by the same coloured ribbon graph.

\begin{definition}\label{hfdjskfbu}
We say that two graphs embedded in a pseudo-surface (or surface) are related by \emph{stabilization} if one can be obtained from the other by a finite sequence of removal and addition of handles which does not disconnect any region or coalesce any two regions, and any discs or annuli involved in adding or removing handles are disjoint from the graph.
\end{definition}

\begin{theorem}\label{dh}
Two graphs embedded in pseudo-surfaces correspond to the same coloured ribbon graph if and only if they are related by stabilization.
\end{theorem}
\begin{proof}
Starting with two graphs embedded in pseudo-surfaces, related by stabilization, consider the stage in the formation of the two coloured ribbon graphs in which we consider  graphs embedded in surfaces. Since stabilization only removes or adds handles in such a way that regions are neither disconnected nor coalesced, the colour classes of the boundary components of the two ribbon graphs will be equivalent. It follows that the two ribbon graphs are equivalent as boundary coloured ribbon graphs. Since the pinch points are unchanged by stabilization, it then follows that the two coloured ribbon graphs are equivalent.

Conversely, the only choice in the construction of a graph in a pseudo-surface from a coloured ribbon graph is in how the faces of the cellularly embedded graph in a pseudo-surface in the same colour class are connected to each other by handles. This is preserved by stabilization. 
\end{proof}

\begin{corollary}\label{hd}
The set of coloured ribbon graphs is in 1-1 correspondence with the set of stabilization equivalence classes of graphs embedded in pseudo-surfaces.
\end{corollary}

We use $[G\subset\check{\Sigma}]_{\mathrm{stab}}$ to denote the stabilization equivalence class.

We use $(G\subset\check{\Sigma})\ba e$ and $(G\subset\check{\Sigma})/e$ to denote the result of deleting and contracting, respectively, an edge $e$ of a graph in a pseudo-surface $G\subset\check{\Sigma}$ using contraction C1 and deletion D1. 

\begin{theorem}\label{pg}
Let $\G$ be a coloured ribbon graph and $[G\subset \check{\Sigma}]_{\mathrm{stab}}$ be its corresponding class of graphs in pseudo-surfaces, and let $e$ denote corresponding edges. Then 
\begin{enumerate}
    \item $\,\G\ba e \leftrightarrow [(G\subset\check{\Sigma})\ba e]_{\mathrm{stab}}$, and 
    \item $\,\G/e\leftrightarrow[(G\subset\check{\Sigma})/e]_{\mathrm{stab}}$.
\end{enumerate}
    That is, the following diagrams commute.
  \[
    \begin{array}{ccc}
    [G\subset\check{\Sigma} ]_{\mathrm{stab}}& \rightarrow & [(G\subset\check{\Sigma})/e]_{\mathrm{stab}} \\
    \updownarrow & & \updownarrow \\
    {\G} & \rightarrow & {\G}/e
    \end{array}
    \qquad
    \begin{array}{ccc}
    [G\subset\check{\Sigma} ]_{\mathrm{stab}}& \rightarrow & [(G\subset\check{\Sigma})\ba e]_{\mathrm{stab}} \\
    \updownarrow & & \updownarrow \\
    {\G} & \rightarrow & {\G}\ba e
    \end{array}
    \]
\end{theorem}
\begin{proof}
Edge deletion in $G\subset\check{\Sigma}$ does not change the pseudo-surface or create any pinch points. However, it may merge the two regions adjacent to the edge $e$. This corresponds to merging boundary components as in Figure~\ref{compdel}.

For contraction there are three cases to consider: when $e$ is not a loop, when it is a loop that has an orientable neighbourhood, and when it is a loop that has a non-orientable neighbourhood. These three cases are considered in Figures~\ref{compdel2}--\ref{compdel4}.
\end{proof}

\begin{figure}[!ht]
\centering
\subfigure[Compatibility of deletion]{
\includegraphics[height=30mm]{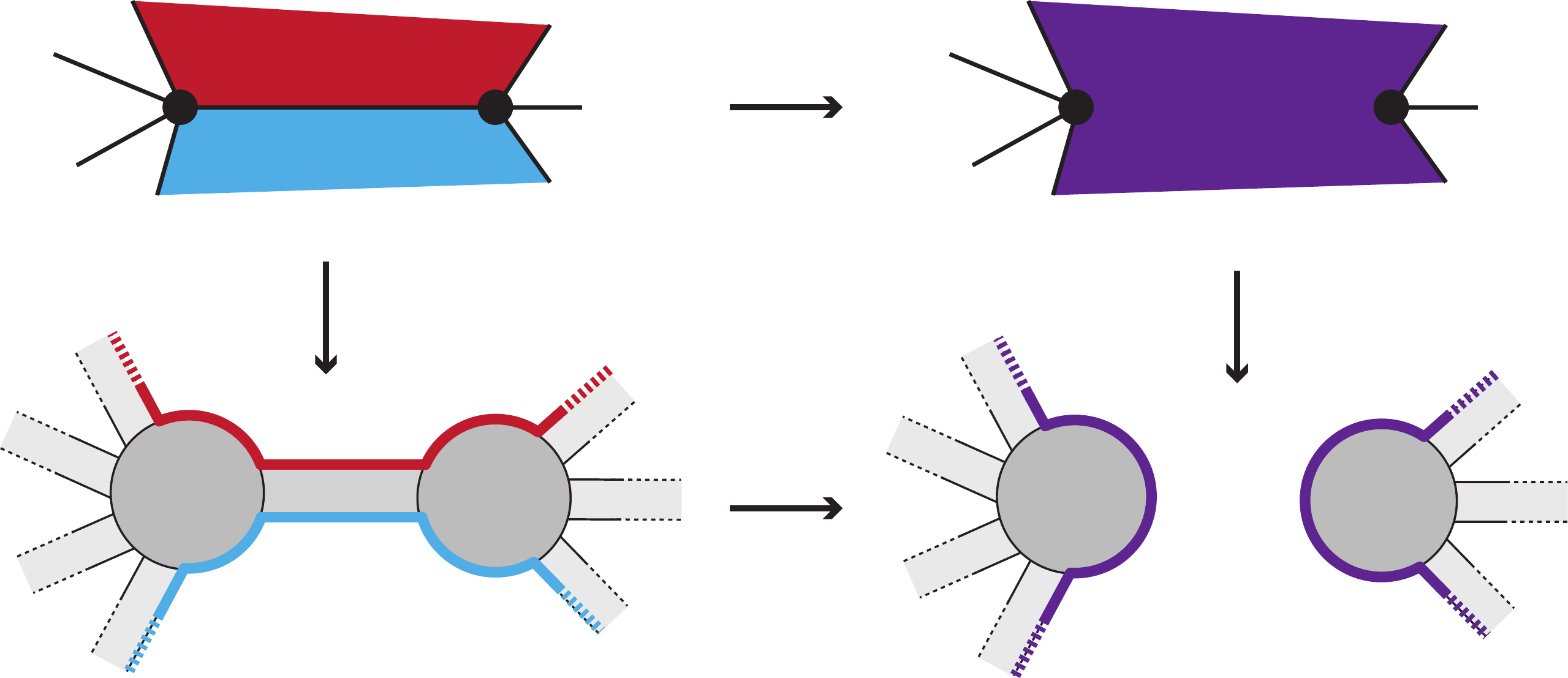}
\label{compdel}
}
\quad
\subfigure[Compatibility of contraction: Case 1]{
\includegraphics[height=30mm]{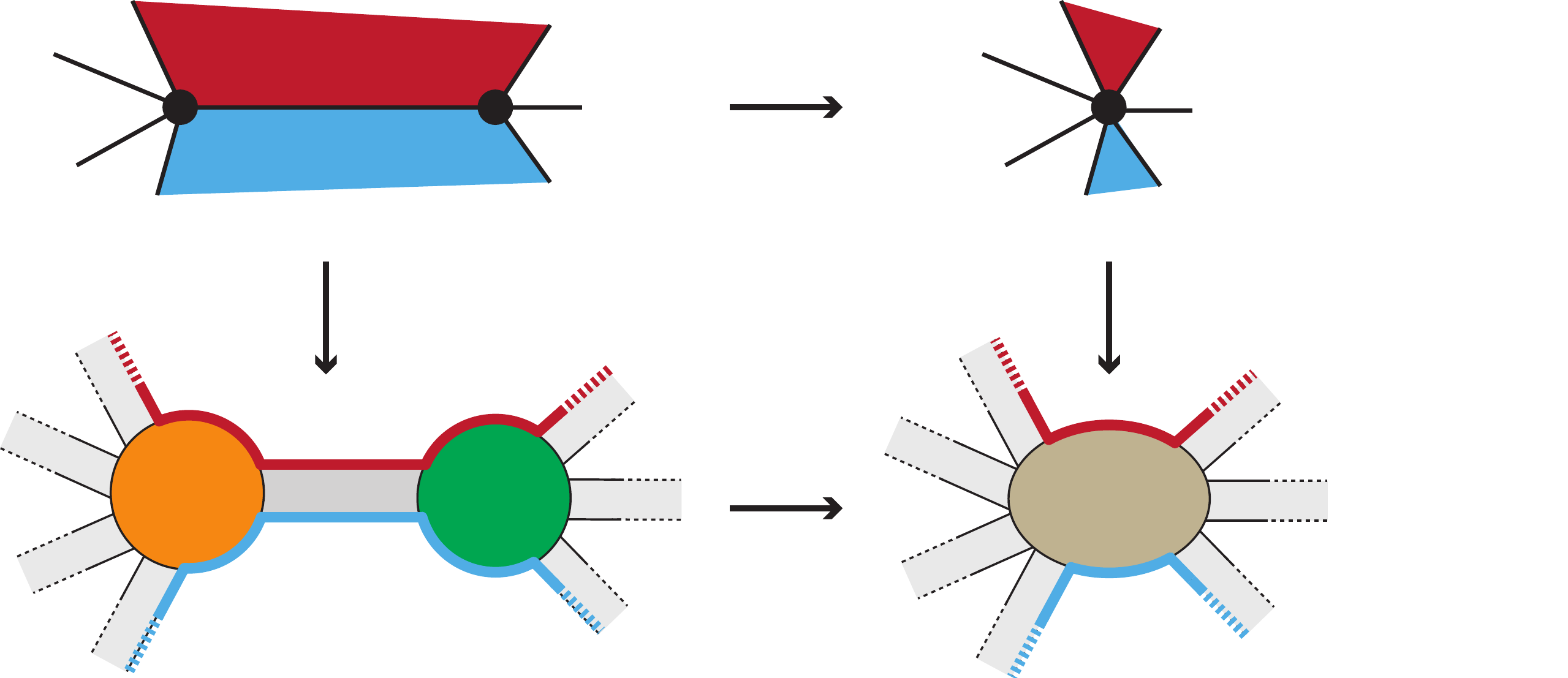}
\label{compdel2}
}
\subfigure[Compatibility of contraction: Case 2]{
\includegraphics[height=50mm]{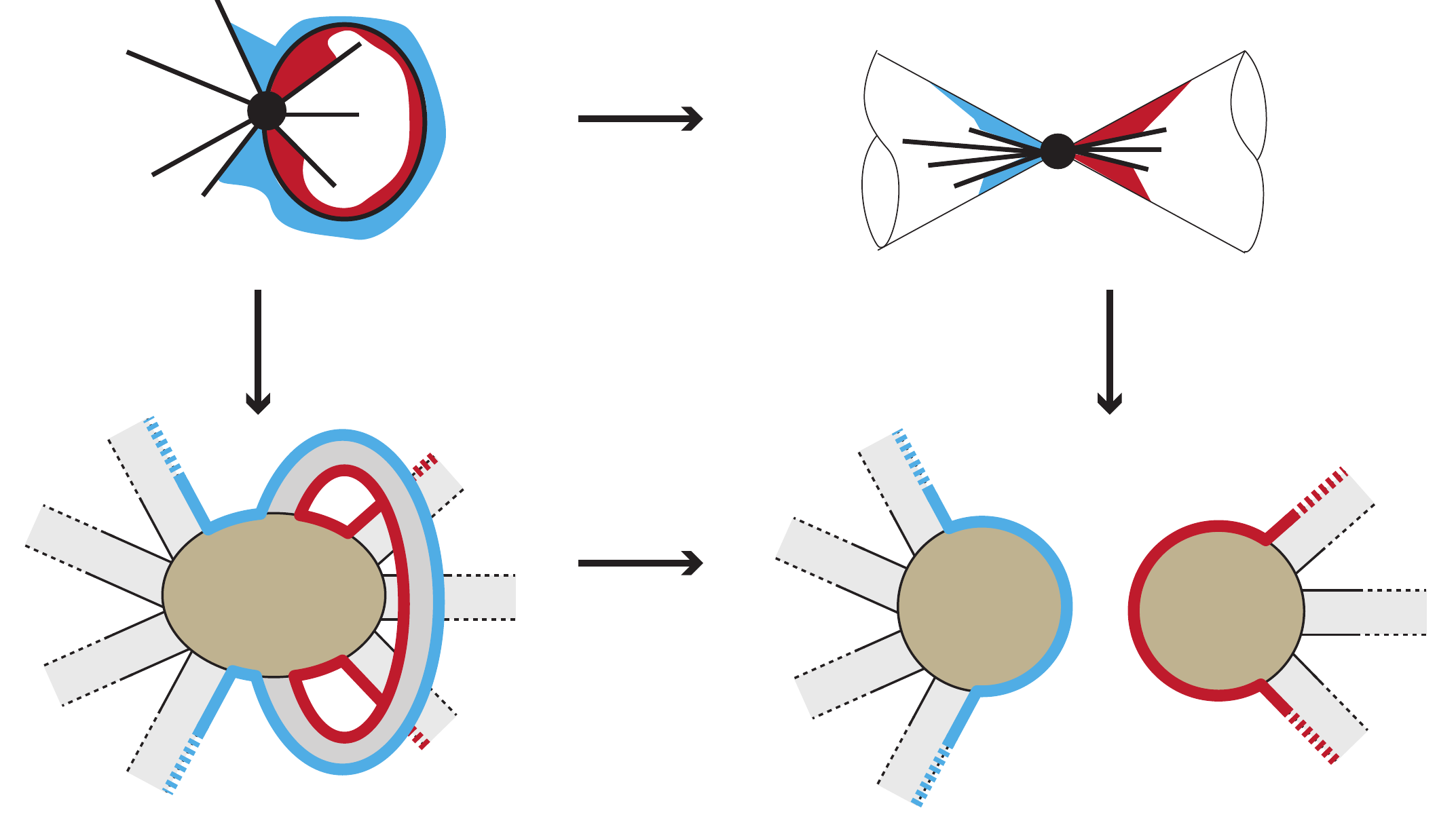}
\label{compdel3}
}

\subfigure[Compatibility of contraction: Case 3]{
\labellist
\small\hair 2pt
\pinlabel  {\tiny cut} at 202 320
\pinlabel  {\tiny surface} at 202 295
\pinlabel  {\tiny contract} at 380 320
\pinlabel  {\tiny edge} at 380 295
\pinlabel  {\tiny flip} at 590 320
\pinlabel  {\tiny glue} at 650 220
\pinlabel  {\tiny surface} at 650 200
\endlabellist
\includegraphics[height=60mm]{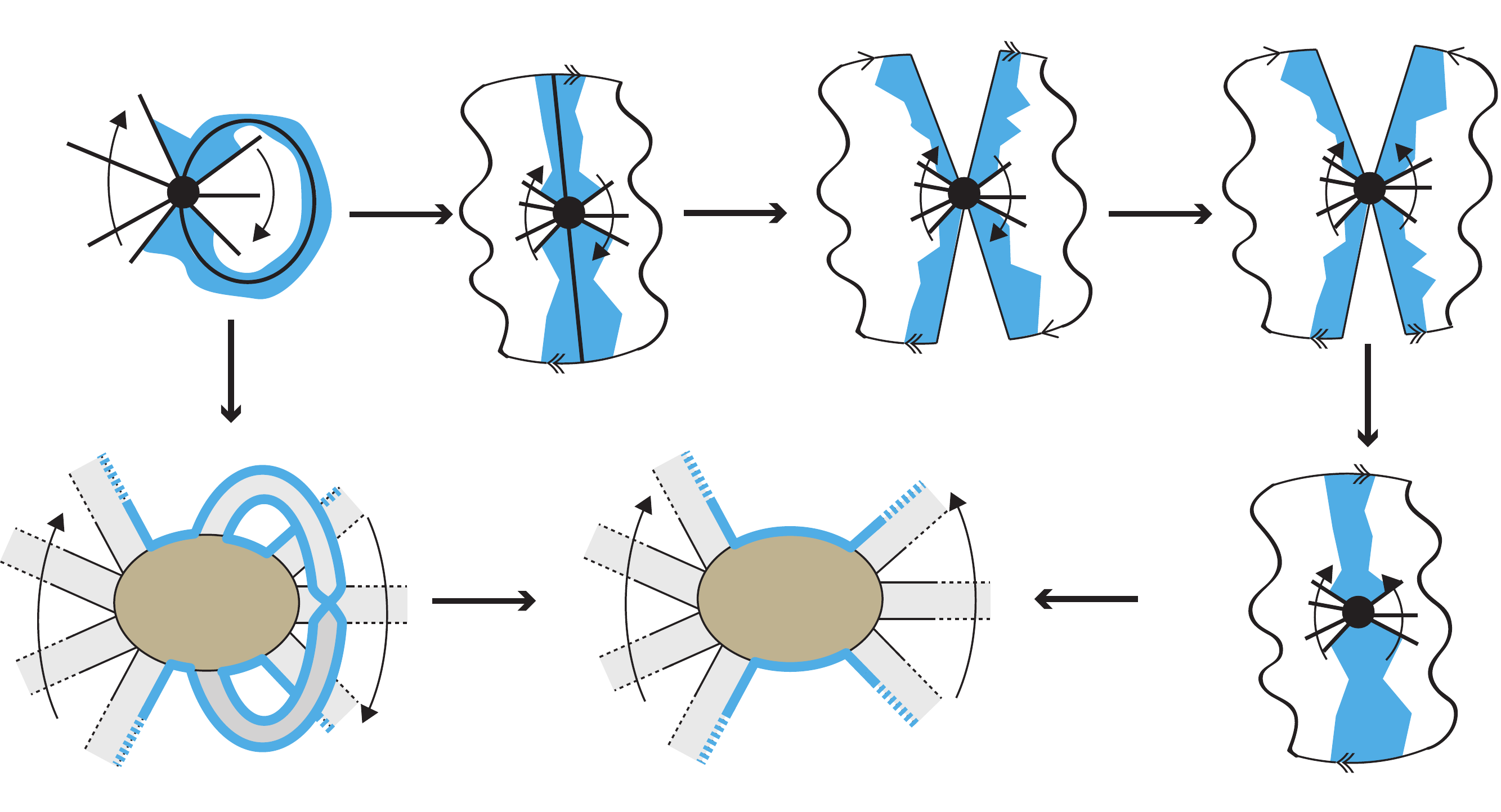}
\label{compdel4}
}
\caption{Compatibility of deletion and contraction for Theorem~\ref{pg}}
\end{figure}
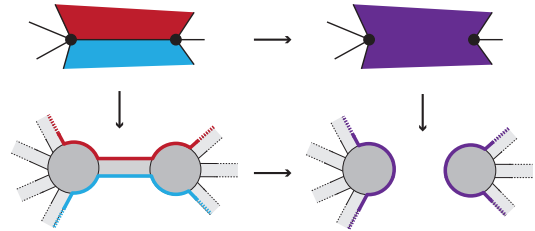
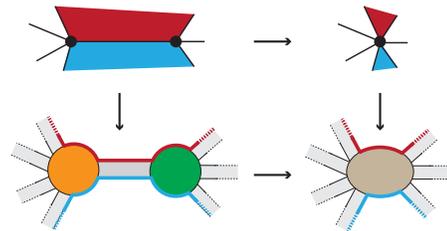
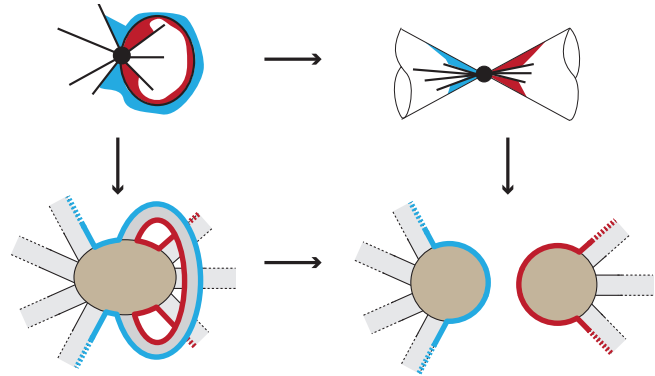
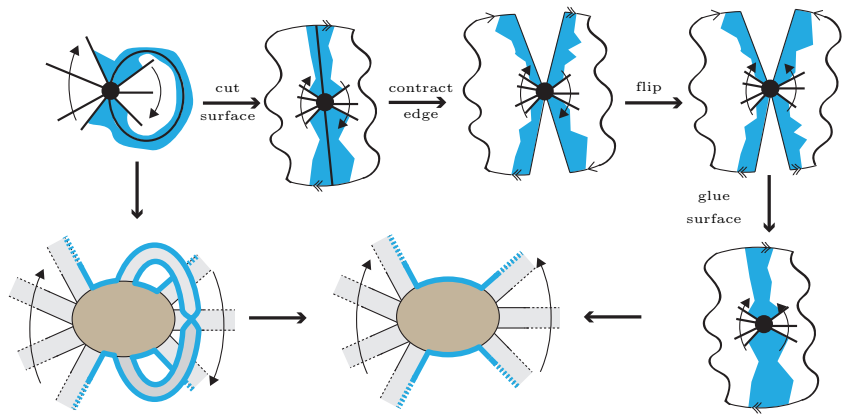

\begin{corollary}[Corollary of Theorem~\ref{dh}]\label{wk}
\begin{enumerate} 
\item \label{wk1}
 The set of coloured ribbon graphs is in 1-1 correspondence with the set of stabilization equivalence classes of graphs embedded in pseudo-surfaces.
\item \label{wk2}
 The set of boundary coloured ribbon graphs is in 1-1 correspondence with the set of stabilization equivalence classes of graphs embedded in surfaces.
\item \label{wk3} 
The set of vertex coloured ribbon graphs is in 1-1 correspondence with the set of graphs cellularly embedded in pseudo-surfaces.
\item \label{wk4}
The set of ribbon graphs is in 1-1 correspondence with the set of graphs cellularly embedded in surfaces.
\end{enumerate}
\end{corollary}

\begin{proof}
Item~\ref{wk1} is a restatement of Corollary~\ref{hd}, and item~\ref{wk4} is the classical result mentioned at the start of Section~\ref{s.4}. 

Item~\ref{wk2} follows from Corollary~\ref{hd} since a graph embedded in a surface is also a graph embedded in a pseudo-surface. Since there are no pinch points, each vertex of the ribbon graph belongs to a distinct colour class and so the vertex colour classes are redundant.

Similarly, item~\ref{wk3} follows from Corollary~\ref{hd} since a graph cellularly embedded in a surface is also a graph embedded in a surface. As the embedding is cellular, each boundary component of the ribbon graph corresponds to a distinct region of the graph in the pseudo-surface.  Thus every boundary component of the ribbon graph belongs to a distinct colour class and so the boundary colour classes are redundant.
\end{proof}

\begin{corollary}[Corollary of Theorem~\ref{pg}]\label{gp}
\begin{enumerate} 
\item \label{gp1} 
If $\mathbb{G}$ is a coloured ribbon graph and $[G\subset \check{\Sigma}]_{\mathrm{stab}}$ its corresponding class of graphs  in  pseudo-surfaces, then
$\mathbb{G}\ba e\leftrightarrow[(G\subset\check{\Sigma})\ba e]_{\mathrm{stab}}$ and $\mathbb{G}/e\leftrightarrow[(G\subset \check{\Sigma})/e]_{\mathrm{stab}},$
where contraction C1 and deletion D1 are used.
\item \label{gp2} 
If $\mathbb{G}$ is a boundary coloured ribbon graph and $[G\subset \Sigma]_{\mathrm{stab}}$ its corresponding class of graphs  in surfaces, then
$\mathbb{G}\ba e\leftrightarrow[(G\subset\Sigma)\ba e]_{\mathrm{stab}}$ and $\mathbb{G}/e\leftrightarrow[(G\subset \check{\Sigma})/e]_{\mathrm{stab}},$
where contraction C2 and deletion D1 are used.
\item \label{gp3} 
If $\mathbb{G}$ is a vertex coloured ribbon graph and $G\subset \check{\Sigma}$ its corresponding graph cellularly embedded in a pseudo-surface, then
$\mathbb{G}\ba e\leftrightarrow(G\subset\check{\Sigma})\ba e$ and $\mathbb{G}/e\leftrightarrow(G\subset\check{\Sigma})/e,$
where contraction C1 and deletion D2 are used.
\item \label{gp4} 
If $\mathbb{G}$ is a ribbon graph and $G\subset \Sigma$ its corresponding graph cellularly embedded in a surface, then
$\mathbb{G}\ba e\leftrightarrow(G\subset{\Sigma})\ba e$ and $\mathbb{G}/e\leftrightarrow(G\subset{\Sigma})/e,$
where contraction C2 and deletion D2 are used.
\end{enumerate}
\end{corollary}
\begin{proof}
Item~\ref{gp1} is a restatement of Theorem~\ref{pg}. The remaining items also follow from  Theorem~\ref{pg}.

For item~\ref{gp2}, the only difference between the deletion and contraction operations for graphs on surfaces and those for graphs on  pseudo-surfaces is that if contraction of an edge on a surface creates a pinch point, then it is resolved. Thus this is the only case we need to examine. However, when converting  a graph on a pseudo-surface to a coloured ribbon graph the first step is to resolve any pinch points.  Thus if every vertex of the ribbon graph $\mathbb{G}/e$ is in a distinct colour class, we see that the corresponding graph on a (pseudo-)surface is $(G\subset\check{\Sigma})/e$. 

For item~\ref{gp3}, the only difference between the deletion and contraction operations for graphs cellularly embedded in pseudo-surfaces and those embedded in  pseudo-surfaces is that redundant handles should be removed after deleting an edge. This corresponds to placing each boundary component of the ribbon graph in a distinct colour class. It follows that $\mathbb{G}\ba e$ corresponds to the graph cellularly embedded in a pseudo-surface $(G\subset\check{\Sigma})\ba e$.

Item~\ref{gp4} follows by combining the arguments for items~\ref{gp2} and~\ref{gp3}. 
\end{proof}

\subsection{Duality}

The construction of the \emph{geometric dual}, $G^*\subset \Sigma$, of a cellularly embedded graph $G\subset \Sigma$ is well known: $V(G^*)$ is obtained by placing one vertex in each face of $G$, and $E(G^*)$ is obtained by embedding an edge of $G^*$ between two vertices whenever the faces of $G$ in which they lie are adjacent. Geometric duality has a particularly neat description when described in the language of ribbon graphs. Let $\G=(V(\G),E(\G))$ be a ribbon graph. Recalling that, topologically, a ribbon graph is a surface with boundary, we cap off the holes using a set of discs, denoted by $V(\G^*)$, to obtain a surface without boundary. The \emph{geometric dual} of $\G$ is the ribbon graph $\G^*=(V(\G^*),E(\G))$. Observe that there is a 1-1 correspondence between the vertices of $\G$ (respectively, $\G^*$) and the boundary components of $\G^*$ (respectively, $\G$). If $\G$ is a coloured ribbon graph then this provides a way to transfer the vertex and boundary colourings between a ribbon graph and its dual.
\begin{definition}
Let $\G$ be a coloured ribbon graph. Its \emph{dual}, $\G^*$, is the coloured ribbon graph consisting of the ribbon graph $\G^*$ with vertex colouring induced from the boundary colouring of $\G$, and boundary colouring induced from the vertex colouring of $\G$. \end{definition}

The definition of a dual of a coloured ribbon graph induces, by forgetting the appropriate colour classes, duals of boundary coloured ribbon graphs and vertex coloured ribbon graphs.
Observe that the dual of a  boundary coloured ribbon graph is a vertex coloured ribbon graph, and vice versa. Thus neither class is closed under duality.

\begin{theorem}\label{zxz}
Let $\mathbb{G}$ be a coloured ribbon graph and let $e$ be an edge of $\G$. Then $(\G^*)^*=\G$, $(\G/e)^*=(\G^*)\ba e$, and $(\G\ba e)^*=(\G^*)/e$.
\end{theorem}
\begin{proof}
The three identities are known to hold for ribbon graphs (see, e.g., \cite{EMMbook}).
The result then follows by observing the effects of the operations on the boundary components and vertices.
\end{proof}

\subsection{Loops in ribbon graphs}

An edge of a ribbon graph is a \emph{loop} if it is incident with exactly one vertex. A loop is said to be \emph{non-orientable} if that edge together with its incident vertex is homeomorphic to a M\"obius band, and otherwise it is said to be \emph{orientable}. See Table~\ref{tablecontractrg}. An edge is a \emph{bridge} if its removal increases the number of components of the ribbon graph. 

In plane graphs, bridges and loops are dual in the sense that an edge of a plane graph $G$ is a loop if and only if the corresponding edge in $G^*$ is a bridge. This leads to the name co-loop for a bridge, in this context. Such terminology would be inappropriate in the context of ribbon graphs, however, where there is more than one type of loop, so we use a new word.

\begin{definition}\label{doopdef}
Let $e$ be an edge of a ribbon graph $\G$ and $e^*$ be its corresponding edge in $\G^*$. Then $e$ is a \emph{dual-loop}, or more concisely a \emph{doop}, if $e^*$ is a loop in $\G^*$. A doop is said to be \emph{non-orientable} if $e^*$ is non-orientable, and is \emph{orientable} otherwise.
\end{definition}

Doops can be recognised directly in $\G$ by looking to see how the boundary components of $\G$ touch the edge, as Figure~\ref{dahkj}.

\begin{figure}[ht]
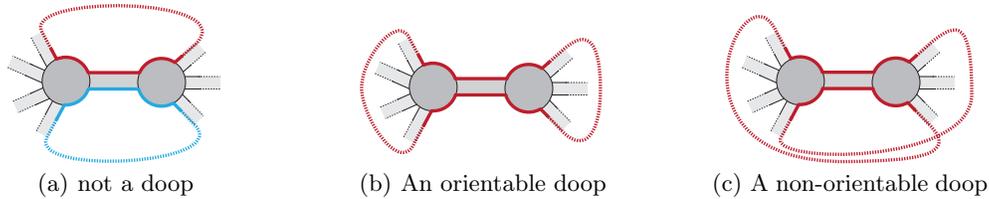

\centering
\subfigure[not a doop]{
\includegraphics[scale=.25]{figs/bc4} 
}
\hspace{1cm}
\subfigure[An orientable doop]{
\includegraphics[scale=.25]{figs/bc1} 
}
\hspace{1cm}
\subfigure[A non-orientable doop]{
\includegraphics[scale=.25]{figs/bc5}
}
\caption{Recognising doops}
\label{dahkj}
\end{figure}

The loop and doop terminology extends to coloured ribbon graphs.

\section{Constructing topological Tutte polynomials}\label{s.5}

Recall that our goal here is the extension of the deletion-contraction definition of the Tutte polynomial to the setting of graphs on surfaces. To do this, following \ref{t1}--\ref{t3}, we need the objects, deletion and contraction, and the cases. Following Section~\ref{s.3}, we now know our objects and our deletion and contraction operations for them. Moreover, we have just seen that there are four different settings to consider, as in Table~\ref{rst}. Since the class of coloured ribbon graphs is the most general, with the other three classes of objects being obtained from it by forgetting information, we will work with coloured ribbon graphs as our primary class.

\begin{table}
\begin{tabular}{| l | l |}
\hline
\textbf{Ribbon graphs} & \textbf{Graphs on surfaces} \\ \hline
 ribbon graphs &   graphs cellularly embedded in surfaces, \\ 
 \hline
 boundary coloured ribbon graphs &   graphs embedded in surfaces, \\  \hline
 vertex coloured ribbon graphs &   graphs cellularly embedded in pseudo-surfaces, \\
 \hline
 coloured ribbon graphs &   graphs embedded in pseudo-surfaces. 
 \\ \hline
\end{tabular}
\caption{The correspondence between ribbon graphs and graphs on surfaces}
\label{rst}
\end{table}

We now consider the problem in \ref{t3}: that of constructing the cases for the deletion-contraction relation.

The deletion-contraction relations \eqref{d1} for the classical Tutte polynomial are divided into cases according to whether an edge is a bridge, a loop, or neither. But these terms are specific to the setting of graphs, and so any canonical construction that we want to apply to a broader class of objects will need to avoid them.

We know that the specific deletion-contraction definition we need will depend upon the type of objects (graphs, ribbon graphs, etc.) under consideration. We also know that the recursion relation will, in general, be subdivided into various cases depending upon edge types, as in \eqref{d1}. Our first task is therefore to divide the edges up into different types. 

In general, it is far from obvious what edge types should be used. For example, one might try to define a Tutte polynomial for ribbon graphs by using bridges, loops, and ordinary edges as the edge types, and apply \eqref{d1} to ribbon graphs. However, the resulting polynomial is just the classical Tutte polynomial of the underlying graph. (In fact the situation is a little more subtle. For graphs, $G/e=G\ba e$ when $e$ is a loop. This is not true for ribbon graphs so, for example, changing $yT(G\ba e)$ to $yT(G/ e)$ in \eqref{d1}, which often happens in definitions of the graph polynomial, would require $x=y$ in order for this $T$ to be well-defined.) 

\subsection{A canonical approach to the cases}\label{s.heo}

We want a canonical way of defining edge types, and all we have to work with are the objects themselves, deletion, and contraction. We also require that any definition or construction we adopt should result in the classical Tutte polynomial when applied to graphs. This requirement, in fact, provides us with the insight enabling us to construct a general framework: let us start by seeing how to characterise edge types in this classical case, using only the concepts of deletion and contraction.

First observe that there are two connected graphs on one edge, \raisebox{-3mm}{\includegraphics[scale=.4]{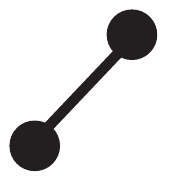}} and \raisebox{-3mm}{\includegraphics[scale=.4]{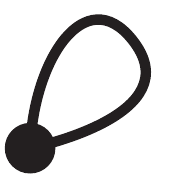}}, 
and that every graph on one edge consists of one of these together with some number of isolated vertices.

For an edge $e$ of a graph $G=(V,E)$ recall that $e^c$ denotes $E\ba e$. The pair \[(G/ e^c,G\ba e^c),\] after disregarding isolated vertices, is one of 
\[ \left(    \raisebox{-2mm}{\includegraphics[scale=.3]{figs/b14}}, \raisebox{-2mm}{\includegraphics[scale=.3]{figs/b13}}\right), 
\quad 
\left(    \raisebox{-2mm}{\includegraphics[scale=.3]{figs/b14}}, \raisebox{-2mm}{\includegraphics[scale=.3]{figs/b14}} \right),
\quad
\left(    \raisebox{-2mm}{\includegraphics[scale=.3]{figs/b13}}, \raisebox{-2mm}{\includegraphics[scale=.3]{figs/b13}}\right),
\quad 
\left(    \raisebox{-2mm}{\includegraphics[scale=.3]{figs/b13}}, \raisebox{-2mm}{\includegraphics[scale=.3]{figs/b14}} \right).
\]
Moreover these pairs classify edge types:
\begin{framed}
\begin{center}
\begin{tabular}{rl}
  
     $(G/ e^c,G\ba e^c)=\left(    \raisebox{-2mm}{\includegraphics[scale=.3]{figs/b14}}, \raisebox{-2mm}{\includegraphics[scale=.3]{figs/b13}} \right) $&$\iff e$ is ordinary       
   \\
    $(G/ e^c,G\ba e^c)=\left(    \raisebox{-2mm}{\includegraphics[scale=.3]{figs/b14}}, \raisebox{-2mm}{\includegraphics[scale=.3]{figs/b14}} \right) $&$\iff e$ is a loop   
   \\ 

    $(G/ e^c,G\ba e^c)=\left(    \raisebox{-2mm}{\includegraphics[scale=.3]{figs/b13}}, \raisebox{-2mm}{\includegraphics[scale=.3]{figs/b13}}\right) $&$\iff e$ is a bridge   
       \\
   $(G/ e^c,G\ba e^c)=\left(    \raisebox{-2mm}{\includegraphics[scale=.3]{figs/b13}}, \raisebox{-2mm}{\includegraphics[scale=.3]{figs/b14}}\right)$ & is impossible
\end{tabular}
\end{center}
\end{framed}

We say that an edge $e$ is of \emph{type} $(i,j)$, for $i,j\in \left\{    \raisebox{-2mm}{\includegraphics[scale=.3]{figs/b13}}, \raisebox{-2mm}{\includegraphics[scale=.3]{figs/b14}}\right\}$, if the pair $(G/ e^c ,G\ba e^c)$ is the pair $(i,j)$ after disregarding isolated vertices. Let $a_{i}$ and $b_{j}$ be indeterminates, and define a deletion-contraction relation by
\begin{equation}\label{hjad}
U(G) = 
  \begin{cases} 
   a_{i}U(G\ba e)+b_{j}U(G/e) & \text{if } e \text{ is of type }(i,j) \\
   \gamma^{n} & \text{if } G \text{ is edgeless, with } n\text{ vertices}.
  \end{cases}
\end{equation}
(Note that if it is preferred not to have to make use of the notion of vertices, then $\gamma$ can be taken to be 1.)

Rewriting \eqref{hjad}  in standard graph terminology gives
\begin{equation}\label{hjad2}
U(G) = 
\begin{cases}
a_{\includegraphics[height=3mm]{figs/b13}} U(G\ba e) + b_{\includegraphics[height=3mm]{figs/b13}} U(G/e) & \text{ if $e$ is a bridge,}
\\
a_{\includegraphics[height=3mm]{figs/b14}} U(G\ba e) + b_{\includegraphics[height=3mm]{figs/b14}} U(G/e) & \text{ if $e$ is a loop,}
\\
a_{\includegraphics[height=3mm]{figs/b14}} U(G\ba e) + b_{\includegraphics[height=3mm]{figs/b13}} U(G/e) & \text{ if $e$ is ordinary,}
\\
\gamma^n & \text{ if $G$ is edgeless on $n$ vertices.}
\end{cases}
\end{equation}

If $e$ is a loop $G/e=G\ba e$. If $e$ is a bridge, then $\gamma \,U(G/e)=U(G\ba e)$. (This latter result follows from the readily verified equation $\gamma \,U(G*H)=U(G\sqcup H)$, in which $*$ denotes the one-point join and $\sqcup$ the disjoint union.) Equation~\eqref{hjad2} can then be written as
\[U(G) = 
\begin{cases}
(\gamma a_{\includegraphics[height=3mm]{figs/b13}}  + b_{\includegraphics[height=3mm]{figs/b13}}) U(G/ e)  & \text{ if $e$ is a bridge,}
\\
(a_{\includegraphics[height=3mm]{figs/b14}}  + b_{\includegraphics[height=3mm]{figs/b14}}) U(G\ba e) & \text{ if $e$ is a loop,}
\\
a_{\includegraphics[height=3mm]{figs/b14}} U(G\ba e) + b_{\includegraphics[height=3mm]{figs/b13}} U(G/e) & \text{ if $e$ is ordinary,}
\\
\gamma^n & \text{ if $G$ is edgeless on $n$ vertices.}
\end{cases}
\]

Now we observe from the universality property of the Tutte polynomial (Theorem~\ref{thm1}) that $U(G)$ is the Tutte polynomial:
\[  U(G)= \gamma^{k(G)} 
a_{\includegraphics[height=3mm]{figs/b14}}^{|E|-r(G)}
b_{\includegraphics[height=3mm]{figs/b13}}^{r(G)} T(G;
\tfrac{ \gamma a_{\includegraphics[height=3mm]{figs/b13}}+ b_{\includegraphics[height=3mm]{figs/b13} }}
{b_{\includegraphics[height=3mm]{figs/b13}} },
\tfrac{a_{\includegraphics[height=3mm]{figs/b14}}+ b_{\includegraphics[height=3mm]{figs/b14}} }{ a_{\includegraphics[height=3mm]{figs/b14}}}).
  \]  
By setting 
$\gamma=a_{\includegraphics[height=3mm]{figs/b14}}=b_{\includegraphics[height=3mm]{figs/b13}}=1$, 
$a_{\includegraphics[height=3mm]{figs/b13}}=x-1$, and
$b_{\includegraphics[height=3mm]{figs/b14}}=y-1$
we recover the Tutte polynomial $T(G;x,y)$.

\medskip

The point of this discussion is that we have recovered the classical Tutte polynomial without having to refer to loops or bridges: these terms only appeared when we interpreted the general procedure in the terminology of graph theory. Thus we have a canonical procedure that we can apply to other classes of object to construct a `Tutte polynomial'. Let us now do this to define topological Tutte polynomials.

\begin{remark}\label{rem1}
The approach to Tutte polynomials that we have taken has its origins in the theory of canonical Tutte polynomials defined by T.~Krajewski, I.~Moffatt, and A.~Tanasa in \cite{KMT}, and the work on canonical Tutte polynomials of delta-matroid perspectives by I.~Moffatt and B.~Smith in \cite{BS}. In \cite{KMT} a `Tutte polynomial' for a connected graded  Hopf algebra is defined as a convolution product of exponentials of certain infinitesimals. Classes of combinatorial objects with suitable notions of deletion and contraction give rise to Hopf algebras and so have a \emph{canonical} Tutte polynomial associated with them. Under suitable conditions, these polynomials have recursive deletion-contraction formulae of the type found in \eqref{hjad} and \eqref{fhuf}. Canonical Tutte polynomials of Hopf algebras of `delta-matroid perspectives' are studied in \cite{BS}. Delta-matroid perspectives are introduced to offer a matroid theoretic framework for topological Tutte polynomials. In particular it is proposed in \cite{BS} that the graphical counter-part of delta-matroid perspectives are `vertex partitioned graphs in surfaces', which are essentially graphs in pseudo-surfaces. The polynomials presented here are compatible with those arising from the canonical Tutte polynomials of delta-matroid perspectives (see Section~\ref{dhjk}). The overall approach that is presented here is the result of our attempt to decouple the theory of canonical Tutte polynomials from the Hopf algebraic framework, and to decouple the topological graph theory from the matroid theoretic framework. We note that because we restrict our work to the setting of graphs in surfaces, many of our results here are  more general than what can be deduced from the present general theory of canonical Tutte polynomials.
\end{remark}

\subsection{The Tutte polynomial of coloured ribbon graphs: a detailed analysis}\label{po}

We apply the process of Section~\ref{s.heo} to coloured ribbon graphs, but postpone the more technical proofs until Section~\ref{tsr} to avoid  interrupting the narrative.
There are five connected coloured ribbon graphs on one edge.
Every ribbon graph on one edge consists of one of these together with some number of isolated vertices.

\begin{figure}[ht]
\centering
\begin{tabular}{ccccc}
\includegraphics[height=20mm]{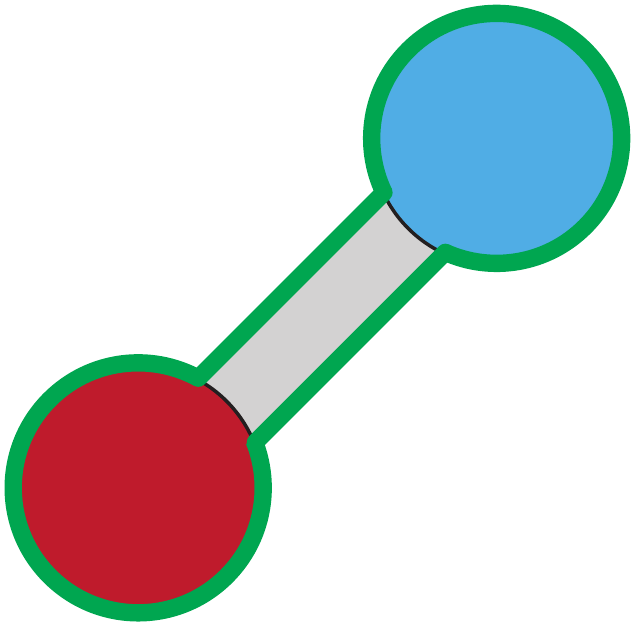}
&
\includegraphics[height=20mm]{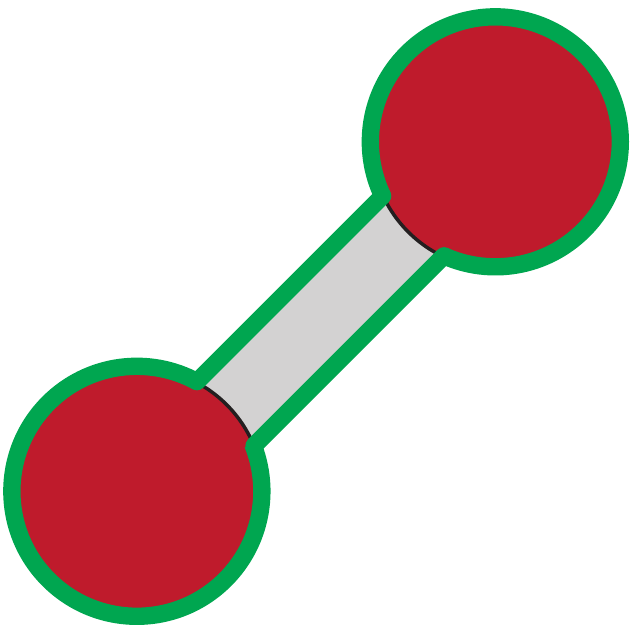}
&
\includegraphics[height=20mm]{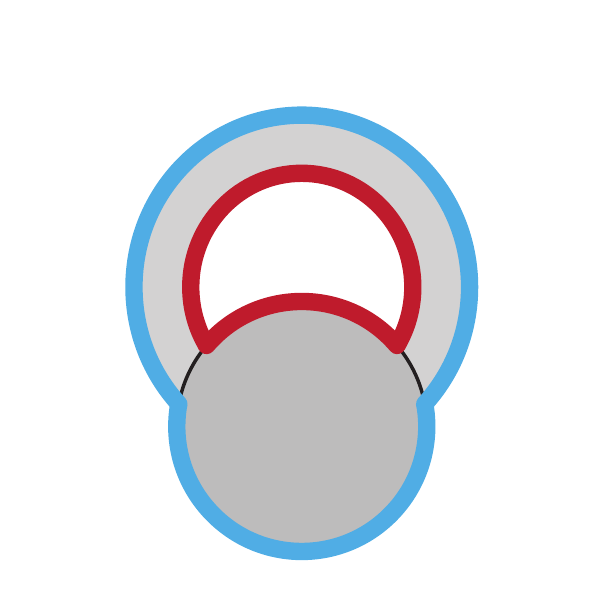}
&
\includegraphics[height=20mm]{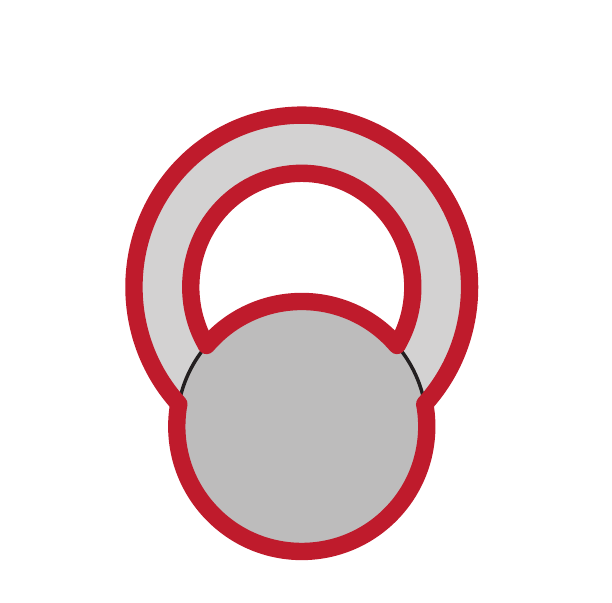}
&
\includegraphics[height=20mm]{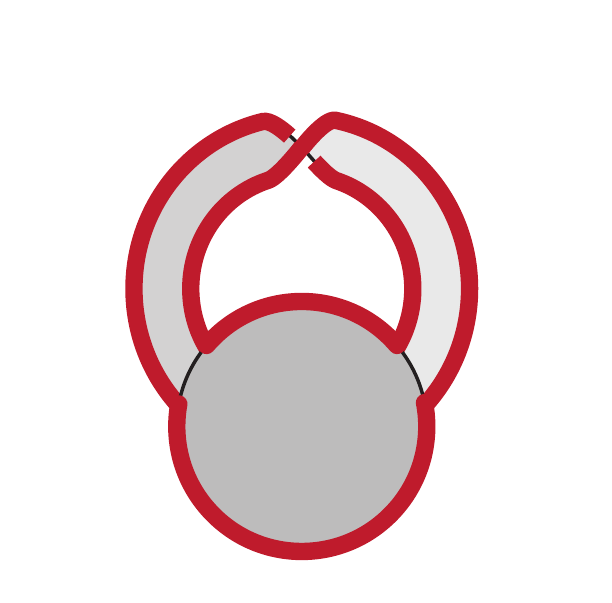}
\\
$\bs$&$\bp$&$\olc$&$\olh$&$\nl$
\end{tabular}
\caption{The coloured ribbon graphs $\bs,\bp,\olc, \olh,\nl$}
\label{fjl}
\end{figure}

For ease of notation we refer to these five coloured ribbon graphs, as in the figure, by 
\begin{itemize}
\item $\bs$ (bridge surface),
\item $\bp$ (bridge pseudo-surface),
\item $\olc$ (orientable loop cellular),
\item $\olh$ (orientable loop handle),
\item $\nl$ (non-orientable loop).
\end{itemize}

As before, we say that an edge $e$ is of \emph{type} $(i,j)$, for $i,j\in\{\bs,\bp,\olc, \olh,\nl\}$, if the pair \[(\G/ e^c,\G\ba e^c),\] is the pair $(i,j)$ after disregarding vertices. Let $a_{i}$ and $b_{j}$ be indeterminates. Set
\begin{equation}\label{fhuf}
P(\G) = 
  \begin{cases} 
   a_{i}P(\G\ba e)+b_{j}P(\G/e) & \text{if } e \text{ is of type }(i,j); \\
   \alpha^{n}\beta^m \gamma^v & \text{if } \G \text{ is edgeless, with }v \text{ vertices, } n\text{ vertex} 
   \\ & \quad \text{colour classes and  } m \text{ boundary colour classes}.
  \end{cases}
\end{equation}

The next step is to specialise the variables so that we obtain a well-defined deletion-contraction invariant. 
As it stands, \eqref{fhuf} does not lead to a well-defined recursion relation for a polynomial $P(\G)$ because the result depends on the order in which the edges of $\G$ are dealt with. 
This can be seen by applying, in the two different ways, the deletion-contraction relations to the ribbon graph consisting of one vertex,  one orientable loop $e$, and one non-orientable loop $f$, the loops in the cyclic order $efef$ at the vertex.
It can be observed in this example that setting $a_{\nl}=\sqrt{a_{\bp}a_{\olh}}$ and $b_{\nl}=\sqrt{b_{\bp}b_{\olh}}$ results in the two computations giving the same answer. In fact, we will see that imposing these conditions makes \eqref{fhuf} a well-defined recursion relation for a graph polynomial.

\begin{samepage}
\begin{theorem}\label{fdgu}
There is a well-defined function $U$ from the set of  coloured ribbon graphs to 
$\mathbb{Z}[ \alpha, \beta, \gamma, a_{\bs},a_{\bp}^{1/2},a_{\olc}, a_{\olh}^{1/2}, b_{\bs},b_{\bp}^{1/2},b_{\olc}, b_{\olh}^{1/2}]$
given by
\begin{equation}\label{efdgu}
U(\G) = 
  \begin{cases} 
   a_{i}U(\G\ba e)+b_{j}U(\G/e) & \text{if } e \text{ is of type }(i,j) \\
   \alpha^{n}\beta^{m} \gamma^v & \text{if } \G \text{ is edgeless, with } v \text{ vertices, } n\text{ vertex} 
   \\ & \quad \text{colour classes and  } m \text{ boundary colour classes}
  \end{cases}
\end{equation}
where, in the recursion,  $a_{\nl}=a_{\bp}^{1/2}a_{\olh}^{1/2}$ and $b_{\nl}=b_{\bp}^{1/2}b_{\olh}^{1/2}$.
\end{theorem}
\end{samepage}
This theorem will follow from Theorem~\ref{zxa} below, in which we will prove that this function $U$ is well-defined by showing that it has a state-sum formulation. For this we will need some more notation. Recall that in a graph $G=(V,E)$ the \emph{rank function} is $r(G)=v(G) -k(G),$
where $v(G)$ denotes the number of vertices of $G$, and  $k(G)$  the number of components. Then for $A\subseteq E$, $r(A)$ is defined to be the rank of the spanning subgraph of $G$ with edge set $A$.

If $\G=(V,E)$ is a ribbon graph and $A\subseteq E$, then $r(\G)$, $r(A)$, $k(\G)$ and $k(A)$ are the parameters of its underlying abstract graph. The number of boundary components of $\G$ is denoted by $b(\G)$, and $b(A):=b(\G\ba A^c)$. $\G$ is \emph{orientable} if it is orientable when regarded as a surface, and the \emph{genus} of $\G$ is its genus when regarded as a surface. The \emph{Euler genus}, $\gamma(\G)$, of  $\G$ is the genus of $\G$ if $\G$ is non-orientable, and is twice its genus if $\G$ is orientable. $\gamma(A):= \gamma(\G \ba A^c)$. \emph{Euler's formula} is $\gamma(\G)=2k(\G)-|V|+|E|-b(\G)$, so \[ \gamma(A)=2k(A)-|V|+|A|-b(A).\]
Where there is any ambiguity over which ribbon graph we are considering we use a subscript, for example writing $r_{\G}(A)$.
\begin{definition}
For a ribbon graph $\G=(V,E)$, with $A\subseteq E$,
$$\rho(\G):=r(\G)+\tfrac{1}{2}\gamma(\G)$$
and
$$\rho(A):=\rho(\G\ba A^{c}).$$
\end{definition}

Observe that when $\G$ is of genus 0 we have $\rho(\G)=r(\G)$. Euler's formula can be used to show that 
\begin{equation}\label{e.vpa}
  \rho_{\G}(A) =  \tfrac{1}{2}\left( |A|+ |V| -b(A)\right).
\end{equation}

For the various coloured ribbon graphs, all these parameters refer to the underlying ribbon graph.

Let $\G$ be a coloured ribbon graph, and denote the vertex colouring by $\V$ and the boundary colouring by $\B$. Now we define the graph (not a ribbon graph) $\G/\V$ as follows. Its vertex set is the set of vertex colour classes, and its adjacency is induced from $\G$. Similarly, the graph $\G^*/\B$ has vertex set the set of boundary colour classes, and for each edge in $\G$ put an edge between the colour classes of its boundary components. We will consider the rank functions $r_{\G/\V}$ and $r_{\G^*/\B}$ of these graphs.

Note that $\G/\V$ can be formed by taking the corresponding graph embedded on a pseudo-surface.

\begin{figure}[ht]
\centering
\begin{tabular}{ccccc}
\includegraphics[height=40mm]{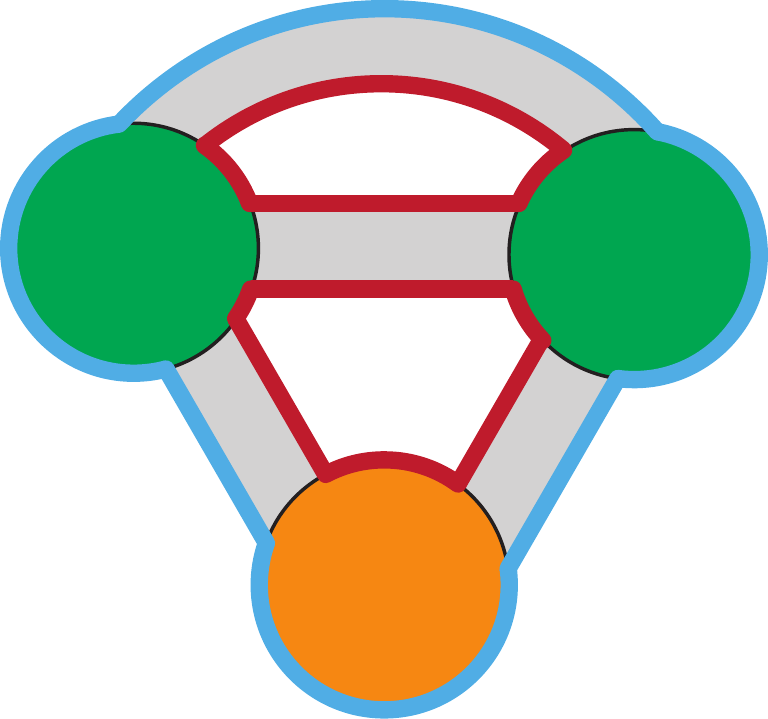} &&
\includegraphics[height=40mm]{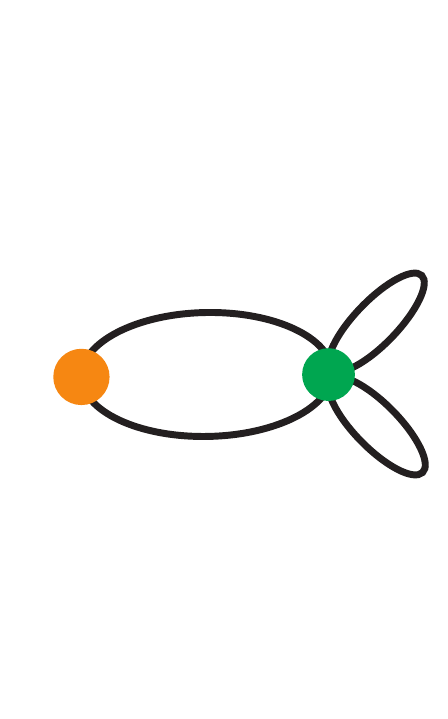} &&
\includegraphics[height=40mm]{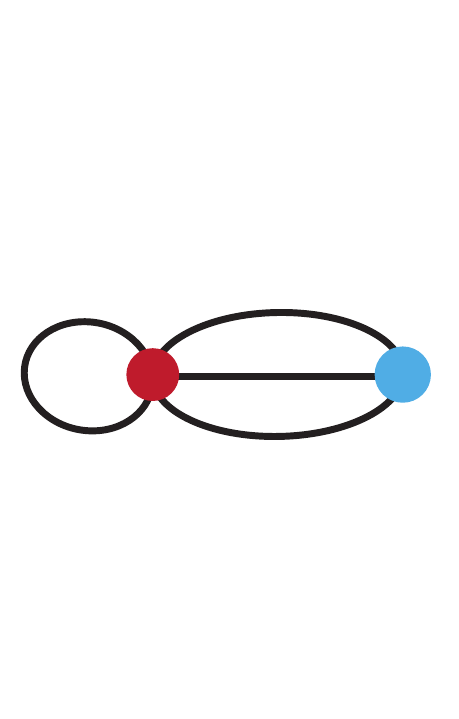} 
\\
$\G$ &\hspace{1cm}&$\G / \V$ &\hspace{1cm}& $\G^* / \B$
\end{tabular}
\caption{Forming $\G / \V$ and $\G^* / \B$ from $\G$}
\label{wgj}
\end{figure}

\begin{theorem}\label{zxa}
Let $\G=(V,E)$ be a coloured ribbon graph with vertex colouring $\V$ and boundary colouring $\B$, and let $U$ be defined as in Theorem~\ref{fdgu}. Then 
\begin{multline}\label{udd}
U(\G)=  
\alpha^{k(\G/\V)} \beta^{k(\G^*/\B)} \gamma^{v(\G)}
(\alpha \, a_{\bs})^{r_1(\G)} 
 a_{\bp}^{r_2(\G)} 
a_{\olc}^{r_3(\G)} 
 a_{\olh}^{r_4(\G)} 
\\
\sum_{A\subseteq E} 
\left(\frac{b_{\bs}}{\alpha \gamma\,  a_{\bs}}\right)^{r_1(A) }
\left(\frac{b_{\bp}}{\gamma\, a_{\bp}}\right)^{r_2(A)}
\left(\frac{\beta\gamma\,  b_{\olc}}{a_{\olc}}\right)^{r_3(A)}
\left(\frac{\gamma\, b_{\olh}}{a_{\olh}}\right)^{ r_4(A)}
\end{multline}
where
\begin{align*}
r_1(A) &:= r_{\G/\V} (A),    & r_3(A) &:=    r_{\G^*/\B} (E)-r_{\G^*/\B} (A^c),  \\
r_2(A) &:= \rho(A)-r_{\G/\V}(A),  & r_4(A) &:= |A| +r_{\G^*/\B} (A^c) - r_{\G^*/\B} (E) -\rho(A),
\end{align*}
and $r_i(\G):=r_i(E)$.
\end{theorem}
To avoid interrupting the narrative, we have put the proof of this theorem in Section~\ref{tsr}. Theorem~\ref{fdgu} follows easily from this one.

The notation  $\G^*/\B$ and $\G/\V$ allows us to express the deletion-contraction relations of Theorem~\ref{fdgu} in a more convenient form as follows. 
\begin{theorem}[Deletion-contraction relations] \label{mnc}
The polynomial $U(\G)$  of Theorem~\ref{fdgu} is uniquely defined by the following deletion-contraction relations. 

\begin{equation}\label{sahj}
U(\G) = 
  \begin{cases} 
   f(e)\,U(\G\ba e)+g(e)\,U(\G/e) & \\
   \alpha^{n}\beta^{m} \gamma^v & \text{if } \G \text{ is edgeless, with } v \text{ vertices, } n\text{ vertex} 
   \\ & \quad \text{colour classes  and  } m \text{ boundary colour classes;}
  \end{cases}
\end{equation}
where
\[f(e) = 
\begin{cases}
a_{\bs} &  \text{if $e$ is an orientable doop in $\G$, a bridge in $\G/\V$} \\
a_{\bp} & \text{if $e$ is an orientable doop in $\G$, not a bridge in $\G/\V$}\\
a_{\olc}& \text{if  $e$ is not a loop in $\G^*/\B$, not a doop in $\G$}\\
a_{\olh}& \text{if $e$ is  a loop in $\G^*/\B$, not a doop in $\G$}\\
\sqrt{a_{\bp}a_{\olh}}& \text{if $e$ is a non-orientable doop in $\G$;}
\end{cases}\]
and 
\[g(e) = 
\begin{cases}
b_{\bs} &  \text{if $e$ is not a loop in $\G$, not a loop in $\G/\V$}\\
b_{\bp} & \text{if $e$ is not a loop in $\G$, a loop  in $\G/\V$}\\
b_{\olc}& \text{if $e$ is a bridge in $\G^*/\B$, an orientable loop in $\G$}\\
b_{\olh}& \text{if $e$ is not a bridge in $\G^*/\B$, an orientable loop in $\G$}\\
\sqrt{b_{\bp}b_{\olh}}& \text{if $e$ is a non-orientable loop in $\G$.}
\end{cases}\]

\end{theorem}
We postpone the proof of this theorem until Section~\ref{tsr}.

It is clear that, up to normalisation, there is some redundancy in the numbers of variables in \eqref{udd}, and four variables suffice. Each selection of four variables has its own advantages and disadvantages (for example, some lead to a smaller number of deletion-contraction relations). Here, motivated by the duality formula \eqref{eld} for the Tutte polynomial, we choose a form that gives the cleanest duality relation.
\begin{definition}\label{dhu}
Let $\G=(V,E)$ be a coloured ribbon graph with vertex colouring $\V$ and boundary colouring $\B$. Then 
\[  T_{ps}(\G;w,x,y,z):=  \sum_{A\subseteq E}   w^{r_1(E) -r_1(A) } x^{r_2(E)-r_2(A)} y^{r_3(A)}  z^{r_4(A)} ,  \]
is the \emph{Tutte polynomial of a coloured ribbon graph}. (Recall that graphs on pseudo-surfaces correspond to coloured ribbon graphs, and hence the subscript $ps$.)
Here
\begin{align*}
r_1(A) &:= r_{\G/\V} (A),    & r_3(A) &:=    r_{\G^*/\B} (E)-r_{\G^*/\B} (A^c),  \\
r_2(A) &:= \rho(A)-r_{\G/\V}(A),  & r_4(A) &:= |A| +r_{\G^*/\B} (A^c) - r_{\G^*/\B} (E) -\rho(A).
\end{align*}
The polynomial is in the ring $\mathbb{Z}[ w,x^{1/2},y,z^{1/2} ]$.
\end{definition}

\begin{theorem}[Universality] \label{sds}
Let $\mathcal{G}$ be a minor-closed class of coloured ribbon graphs. Then there is a unique map $U: \mathcal{G}\rightarrow \mathbb{Z}[\alpha, \beta, \gamma,  a_{\bs},a_{\bp}^{1/2},a_{\olc}, a_{\olh}^{1/2}, b_{\bs},b_{\bp}^{1/2},b_{\olc}, b_{\olh}^{1/2}]$ that satisfies \eqref{efdgu}.
Moreover,
\begin{multline*}
U(\G)=  
\alpha^{k(\G/\V)} \beta^{k(\G^*/\B)}\gamma^{v(\G)-\rho(\G)}
b_{\bs}^{r_1(\G)} 
b_{\bp}^{r_2(\G)}  
a_{\olc}^{r_3(\G)}
a_{\olh}^{r_4(\G)} 
 \\
T_{ps} \left(\G ; 
\frac{\alpha \gamma\, a_{\bs}}{b_{\bs}}, 
\frac{\gamma\,a_{\bp}}{b_{\bp}},
\frac{\beta\gamma\, b_{\olc}}{a_{\olc}},
\frac{\gamma\,b_{\olh}}{a_{\olh}} 
\right).
\end{multline*}
\end{theorem}
\begin{proof}
The result follows routinely from Definition~\ref{dhu} and  Theorems~\ref{fdgu} and~\ref{zxa}.
\end{proof}

\begin{theorem}[Duality] \label{sdr}
\[ T_{ps}(\G^*; w,x,y,z) = T_{ps}(\G; y,z,w,x)\]
\end{theorem}
\begin{proof}
Consider $U(\G^*)$ as a map $U^*:\G\mapsto \G^* \mapsto U(\G^*)$. For the map $U^*$, and using Theorem~\ref{zxz}, we have 
\begin{align*}
U^*(\G) &= U(\G^*) \\
&=    a_{i}U(\G^*\ba e)+b_{j}U(\G^*/e)  \text{ if } e \text{ is of type }(i,j) \text{ in } \G^* \\
&=    a_{i}U((\G/ e)^*)+b_{j}U((\G\ba e)^*)  \text{ if } e \text{ is of type }(i,j) \text{ in } \G^* \\
&=    a_{i}U^*(\G/ e)+b_{j}U^*(\G\ba e)  \text{ if } e \text{ is of type }(i,j) \text{ in } \G^* \\
&=    a_{i}U^*(\G/ e)+b_{j}U^*(\G\ba e)  \text{ if } e \text{ is of type }(i^*,j^*) \text{ in }  \G.
\end{align*}
Since duality interchanges $\bs$ and $\olc$ edges, and $\bp$ and $\olh$ edges, the result follows by universality and specialising to $T_{ps}$.
\end{proof}

\begin{remark}
Theorem~\ref{sdr} can also be proven using the state-sums since 
$r_{1,\G}(E)-r_{1,\G}(A) =  r_{3,\G^*}(A^c) $, and 
$r_{2,\G}(E)-r_{2,\G}(A) =  r_{4,\G^*}(A^c) $.
\end{remark}

\section{The full family of Topological Tutte polynomials}\label{djkda}
We have just described, in Section~\ref{po}, the  Tutte polynomial of coloured ribbon graphs, or graphs in pseudo-surfaces. This is just one of the four minor-closed classes of topological graphs, as given in Table~\ref{rst}. In this section we describe the Tutte polynomials of the remaining classes of topological graphs. 

To obtain these polynomials one could either follow the approach of  Section~\ref{po} for each class of topological graph, or one could observe that the various ribbon graph classes are obtained from coloured ribbon graphs by forgetting the boundary colourings, vertex colourings, or both. (Algebraically, this would correspond to setting $\olc=\olh$, $\bs=\bp$, or both.) Accordingly we omit proofs from this section.

\subsection{The Tutte polynomial of ribbon graphs (or graphs cellularly embedded in surfaces)}

There are three connected ribbon graphs on one edge, shown in Figure~\ref{fwe} with the names we use for them,
\begin{figure}[ht]
\centering
\begin{tabular}{ccc}
\includegraphics[height=20mm]{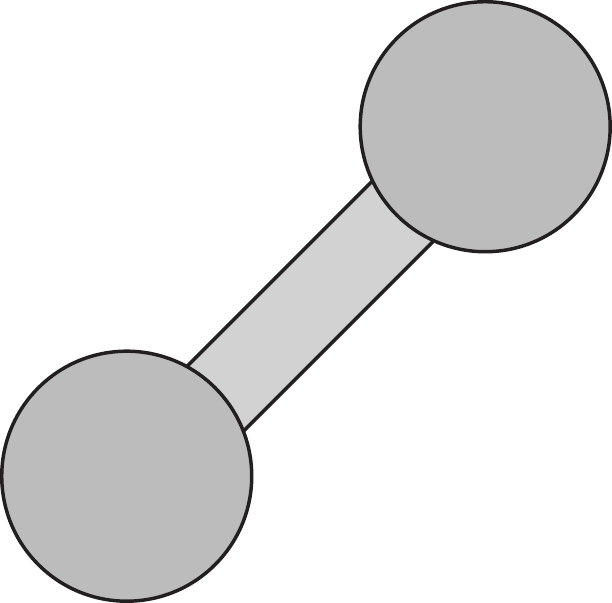}
&
\includegraphics[height=20mm]{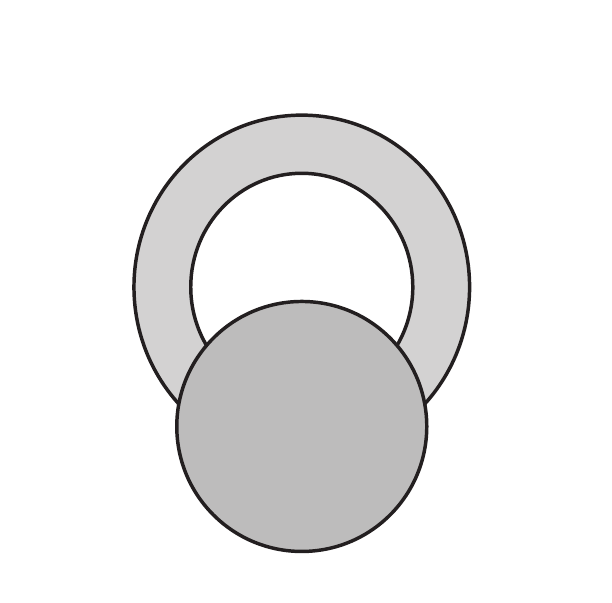}
&
\includegraphics[height=20mm]{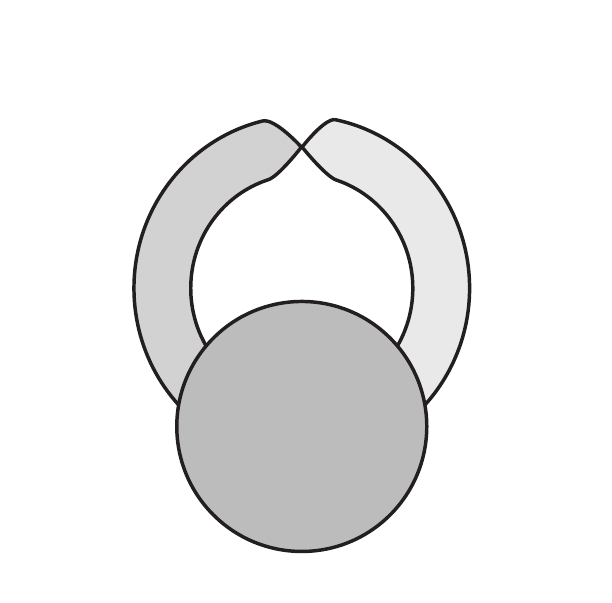}
\\
$\br$&$\ol$&$\nl$
\end{tabular}
\caption{The ribbon graphs $\br,\ol,\nl$}
\label{fwe}
\end{figure}
and every ribbon graph on one edge consists of one of these together with some number of isolated vertices.
An edge $e$ of a ribbon graph $\G$ is of type $(i,j)$, for $i,j\in\{\br,\ol,\nl\}$, if the pair $(\G/ e^c,\G\ba e^c)=(i,j)$ after disregarding isolated vertices.

Define a function $U$ from the set of  ribbon graphs to
$\mathbb{Z}[\gamma, a_{\br}^{1/2}, a_{\ol}^{1/2}, b_{\br}^{1/2},b_{\ol}^{1/2}]$
 by
\begin{equation}\label{zxzx}
U(\G) = 
  \begin{cases} 
   a_{i}U(\G\ba e)+b_{j}U(\G/e) & \text{if } e \text{ is of type }(i,j) \\
   \gamma^v & \text{if } \G \text{ is edgeless, with } v \text{ vertices} 
     \end{cases}
\end{equation}
where, in the recursion,  $a_{\nl}=a_{\br}^{1/2}a_{\ol}^{1/2}$ and $b_{\nl}=b_{\br}^{1/2}b_{\ol}^{1/2}$.

By Proposition~\ref{bhj} the deletion-contraction relations can be rephrased as
\[  
U(\G):= f(e) \,U(\G \ba e) + g(e)\, U(\G /e) ,
\]
where
\[
f(e) =  
\begin{cases}
a_{\br} & \text{if $e$ is an orientable doop} \\ 
a_{\ol} & \text{if $e$ is not a doop} \\ 
\sqrt{a_{\br}a_{\ol}} & \text{if $e$ is a non-orientable doop} \\
\end{cases}
\]
and\footnote{An unfortunate typo means that the  conditions for the first two cases of the following are transposed    in the published version of this paper.} 
\[
g(e) =  
\begin{cases}
b_{\br} & \text{if $e$ is not a loop} \\ 
b_{\ol} &  \text{if $e$ is an orientable loop}\\ 
\sqrt{b_{\br}b_{\ol}} & \text{if $e$ is a non-orientable loop} \\
 \end{cases}
\]

We have 
\begin{equation}\label{uddr}
U(\G)=  
 \gamma^{v(\G)}
a_{\br}^{\rho(\G)} 
a_{\ol}^{|E|-\rho(\G)} 
\sum_{A\subseteq E} 
\left(\frac{b_{\br}}{ \gamma\,  a_{\br}}\right)^{\rho(A) }
\left(\frac{\gamma\, b_{\ol}}{a_{\ol}}\right)^{ |A|-\rho (A)}.
\end{equation}

We then define the Tutte polynomial of  ribbon graphs or cellularly embedded graphs as follows.
 \begin{definition}
Let $\G=(V,E)$ be a  ribbon graph.  Then 
\[T_{cs}(\G;x,y) = 
\sum_{A\subseteq E(\G)} 
x^{\rho(\G)-\rho(A) }
y^{|A|-\rho(A)}
. \] is the \emph{Tutte polynomial of the ribbon graph} $\G$. 
\end{definition}

Note that $T_{cs}$ is the 2-variable Bollob\'as--Riordan polynomial. (We use the subscript $cs$ since ribbon graphs correspond to cellularly embedded graphs on surfaces.) See Section~\ref{dhjk} for details.

\begin{theorem}[Universality]
Let $\mathcal{G}$ be a minor-closed class of ribbon graphs. Then there is a unique map $U: \mathcal{G}\rightarrow \mathbb{Z}[\gamma, a_{\br}^{1/2}, a_{\ol}^{1/2}, b_{\br}^{1/2},b_{\ol}^{1/2}]$ that satisfies \eqref{zxzx}.
Moreover,
\[
U(\G)=  
 \gamma^{v(\G)-\rho(\G)}
b_{\br}^{\rho(\G)} 
a_{\ol}^{|E|-\rho(\G)} 
T_{cs}
\left(\frac{b_{\br}}{ \gamma\,  a_{\br}},\frac{\gamma\, b_{\ol}}{a_{\ol}}\right).
\]
\end{theorem}

\begin{theorem}[Duality]
\[T_{cs}(\G;x,y)  =  T_{cs}(\G^*;y,x)  \]
\end{theorem}

\subsection{The Tutte polynomial of  boundary coloured ribbon graphs (or graphs  embedded in surfaces)}
There are four connected boundary coloured ribbon graphs on one edge, shown in Figure~\ref{fwef} with the names we use for them,
and every boundary coloured ribbon graph on one edge consists of one of these together with some number of isolated vertices.
An edge $e$ of a ribbon graph $\G$ is of type $(i,j)$, for $i,j\in\{\br,\olc,  \olh, \nl\}$, if the pair $(\G/ e^c,\G\ba e^c)=(i,j)$ after disregarding isolated vertices.

\begin{figure}[ht]
\centering
\begin{tabular}{cccc}
\includegraphics[height=20mm]{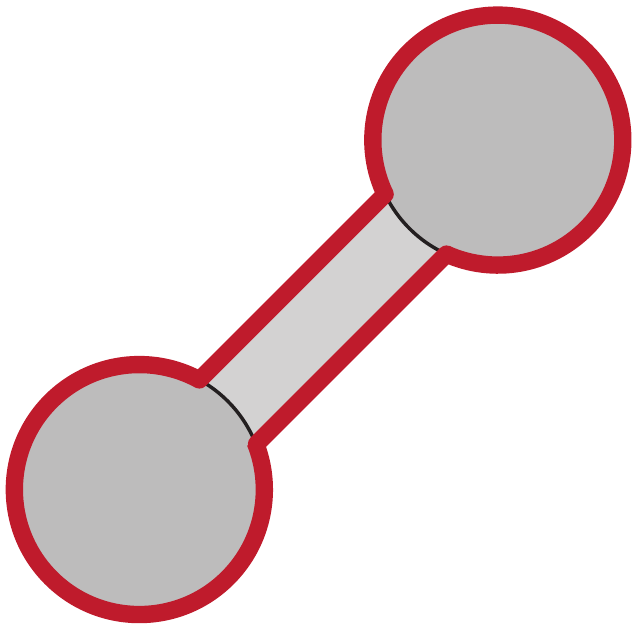}
&
\includegraphics[height=20mm]{figs/olc}
&
\includegraphics[height=20mm]{figs/olh}
&
\includegraphics[height=20mm]{figs/nl}
\\
$\br$&$\olc$&$\olh$&$\nl$
\end{tabular}
\caption{The boundary coloured ribbon graphs $\br,\olc,\olh,\nl$}
\label{fwef}
\end{figure}

Define a function $U$ from the set of boundary coloured ribbon graphs to the ring 
$\mathbb{Z}[  \beta, \gamma, a_{\br}^{1/2},a_{\olc}, a_{\olh}^{1/2},b_{\br}^{1/2},b_{\olc}, b_{\olh}^{1/2}]$
 by
\begin{equation}\label{ers}
U(\G) = 
  \begin{cases} 
   a_{i}U(\G\ba e)+b_{j}U(\G/e) & \text{if } e \text{ is of type }(i,j) \\
  \beta^{m} \gamma^v & \text{if } \G \text{ is edgeless, with } v \text{ vertices, }
   \\ & \quad  \text{and  } m \text{ boundary colour classes}
  \end{cases}
\end{equation}
where $a_{\nl}=a_{\br}^{1/2}a_{\olh}^{1/2}$ and $b_{\nl}=b_{\br}^{1/2}b_{\olh}^{1/2}$.

By Lemma~\ref{asr} the deletion-contraction relations can be rephrased as
\[  
U(\G):= f(e) \,U(\G \ba e) + g(e)\, U(\G /e) ,
\]
where
\[
f(e) =  
\begin{cases}
a_{\br} & \text{if $e$ is an orientable doop} \\ 
a_{\olc} & \text{if $e$  is not a loop in $\G^*/\B$ and not a doop in $\G$} \\ 
a_{\olh} & \text{if $e$ is  a loop in $\G^*/\B$ and not a doop in $\G$} \\ 
\sqrt{a_{\br}a_{\olh}} & \text{if $e$ is a non-orientable doop;} \\
\end{cases}
\]
and 
\[
g(e) =  
\begin{cases}
b_{\br} & \text{if $e$ is an orientable loop} \\ 
b_{\olc} & \text{if $e$ is a bridge in $\G^*/\B$ and  an orientable loop in $\G$} \\ 
b_{\olh} & \text{if $e$ is is not a bridge in $\G^*/\B$ and an orientable loop in $\G$} \\ 
\sqrt{b_{\br}b_{\olh}} & \text{if $e$ is a non-orientable loop.} \\
 \end{cases}
\]

We have 
\begin{equation}\label{plo}
U(\G)=  
 \beta^{k(\G^*/\B)} \gamma^{v(\G)}
a_{\br}^{\rho(\G)} 
a_{\olc}^{r_3(\G)} 
 a_{\olh}^{r_4(\G)} 
\sum_{A\subseteq E} 
\left(\frac{b_{\br}}{ \gamma\,  a_{\br}}\right)^{\rho(A) }
\left(\frac{\beta\gamma\,  b_{\olc}}{a_{\olc}}\right)^{r_3(A)}
\left(\frac{\gamma\, b_{\olh}}{a_{\olh}}\right)^{ r_4(A)}
\end{equation}
where
\begin{equation}\label{r3andr4}
 r_3(A) =    r_{\G^*/\B} (E)-r_{\G^*/\B} (A^c), 
\quad\text{and}\quad
 r_4(A) := |A| +r_{\G^*/\B} (A^c) - r_{\G^*/\B} (E) -\rho(A).
\end{equation}

We then define the Tutte polynomial of boundary coloured ribbon graphs or graphs embedded in surfaces as follows.
\begin{definition}\label{ghj}
Let $\G=(V,E)$ be a ribbon graph with  boundary colouring $\B$. Then 
\[  T_{s}(\G;w,x,y,z):=  \sum_{A\subseteq E}   x^{\rho(E) -\rho(A) }  y^{r_3(A)}  z^{r_4(A)} ,  \]
is the \emph{Tutte polynomial of a boundary coloured ribbon graph}.
Here $ r_3(A)$ and $ r_3(A)$ are as given in \eqref{r3andr4}.
\end{definition}

\begin{theorem}[Universality]
Let $\mathcal{G}$ be a minor-closed class of boundary coloured ribbon graphs. Then there is a unique map $U: \mathcal{G}\rightarrow \mathbb{Z}[  \beta, \gamma, a_{\br}^{1/2},a_{\olc}, a_{\olh}^{1/2},b_{\br}^{1/2},b_{\olc}, b_{\olh}^{1/2}]$ that satisfies \eqref{ers}.
Moreover,
\[
U(\G)=  
 \beta^{k(\G^*/\B)}\gamma^{v(\G)-\rho(\G)}
b_{\br}^{\rho(\G)} 
a_{\olc}^{r_3(\G)}
a_{\olh}^{r_4(\G)} 
 \;
T_{s} \left(\G ; 
\frac{ \gamma\, a_{\br}}{b_{\br}}, 
\frac{\beta\gamma\, b_{\olc}}{a_{\olc}},
\frac{\gamma\,b_{\olh}}{a_{\olh}} 
\right).
\]
\end{theorem}

The dual of a boundary coloured ribbon graph is a vertex coloured ribbon graph and so $T_s$ cannot satisfy a three variable duality relation. However, it is related to the Tutte polynomial of a vertex coloured ribbon graph, $T_{cps}$ as defined below, through duality.
\begin{theorem}[Duality]
Let $\G$ be a vertex coloured ribbon graph. Then
\[ T_{s}(\G; x,y,z) = T_{cps}(\G^*; y,z,x).\]
\end{theorem}

\subsection{The Tutte polynomial of vertex coloured ribbon graphs (or graphs cellularly embedded in pseudo-surfaces)}

There are four connected vertex coloured ribbon graphs on one edge, shown in Figure~\ref{fjlr} with the names we use for them,
and every vertex coloured ribbon graph on one edge consists of one of these together with some number of isolated vertices.
An edge $e$ of a ribbon graph $\G$ is of type $(i,j)$, for $i,j\in\{\bs,\bp,\ol,\nl\}$, if the pair $(\G/ e^c,\G\ba e^c)=(i,j)$ after disregarding isolated vertices.

\begin{figure}[ht]
\centering
\begin{tabular}{cccc}
\includegraphics[height=20mm]{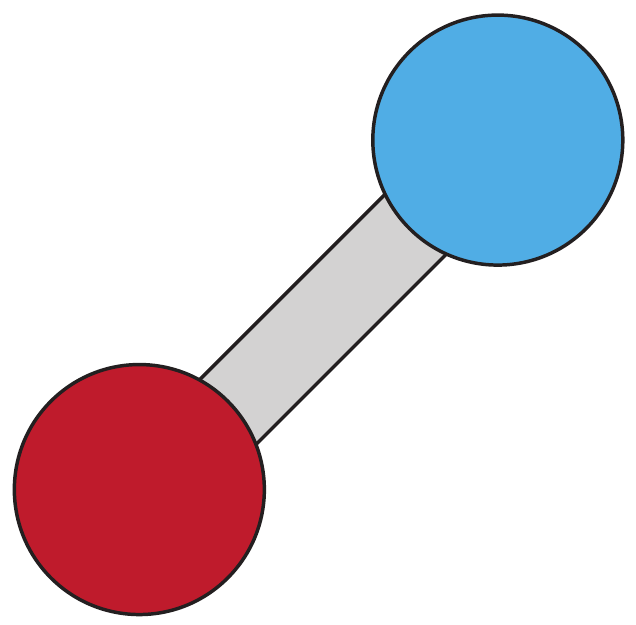}
&
\includegraphics[height=20mm]{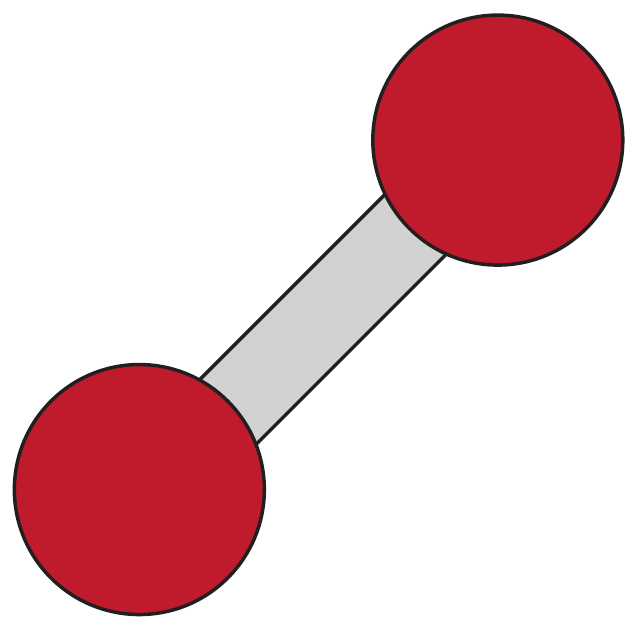}
&
\includegraphics[height=20mm]{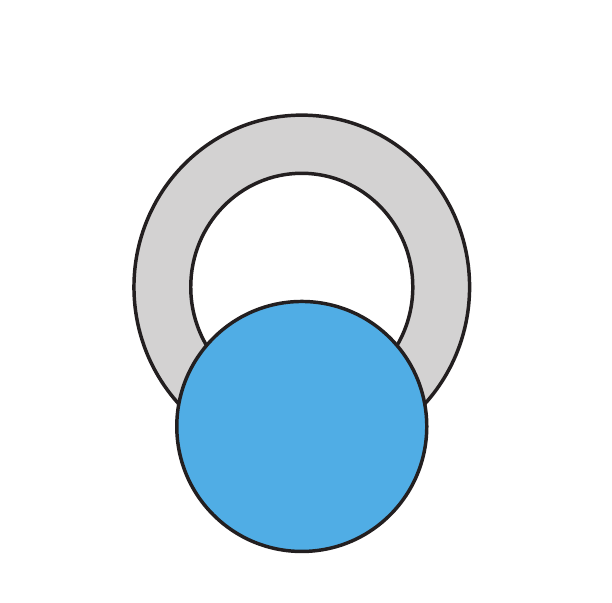}
&
\includegraphics[height=20mm]{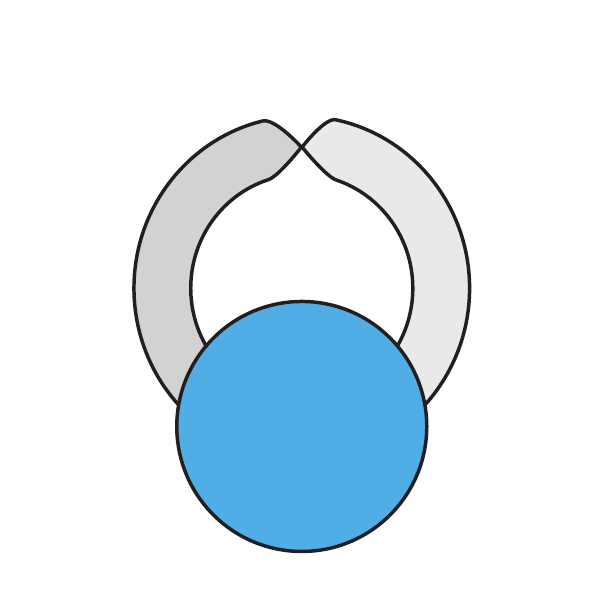}
\\
$\bs$&$\bp$&$\ol$&$\nl$
\end{tabular}
\caption{The vertex coloured ribbon graphs $\bs,\bp,\ol,\nl$}
\label{fjlr}
\end{figure}

Define a function $U$ on  ribbon graphs to
$\mathbb{Z}[ \alpha, \beta, \gamma, a_{\bs},a_{\bp}^{1/2}, a_{\ol}^{1/2}, b_{\bs},b_{\bp}^{1/2}, b_{\ol}^{1/2}]$
given by
\begin{equation}\label{rrg}
U(\G) = 
  \begin{cases} 
   a_{i}U(\G\ba e)+b_{j}U(\G/e) & \text{if } e \text{ is of type }(i,j) \\
   \alpha^{n} \gamma^v & \text{if } \G \text{ is edgeless, with } v  \\ & \quad \text{ vertices, and } n\text{ vertex colour classes  } 
  \end{cases}
\end{equation}
where   $a_{\nl}=a_{\bp}^{1/2}a_{\olh}^{1/2}$ and $b_{\nl}=b_{\bp}^{1/2}b_{\olh}^{1/2}$.

By Lemma~\ref{asr} the deletion-contraction relations can be rephrased as
\[  
U(\G):= f(e) \,U(\G \ba e) + g(e)\, U(\G /e) ,
\]
where
\[
f(e) =  
\begin{cases}
a_{\bs} & \text{if $e$ is an orientable doop in $\G$ and  a bridge in $\G/\V$} \\ 
a_{\bp} & \text{if $e$ is an orientable doop in $\G$ and  not a bridge in $\G/\V$} \\ 
a_{\ol} & \text{if $e$ is  not a doop in $\G$} \\ 
\sqrt{a_{\bp}a_{\ol}} & \text{if $e$ is a non-orientable doop,} \\
\end{cases}
\]
and 
\[
g(e) =  
\begin{cases}
b_{\bs} & \text{if $e$ is not a loop in $\G$ and not a loop in $\G/\V$} \\ 
b_{\bp} & \text{if $e$ is not a loop in $\G$ and a loop  in $\G/\V$} \\ 
b_{\ol} & \text{if $e$ is an orientable loop in $\G$} \\ 
\sqrt{b_{\bp}b_{\ol}} & \text{if $e$ is a non-orientable loop.} \\
 \end{cases}
\]

We have
\begin{multline*}
U(\G)=  
\alpha^{k(\G/\V)}  \gamma^{v(\G)}
(\alpha \, a_{\bs})^{r_1(\G)} 
 a_{\bp}^{r_2(\G)} 
a_{\ol}^{|E|-\rho(\G)}  \\
\sum_{A\subseteq E} 
\left(\frac{b_{\bs}}{\alpha \gamma\,  a_{\bs}}\right)^{r_1(A) }
\left(\frac{b_{\bp}}{\gamma\, a_{\bp}}\right)^{r_2(A)}
\left(\frac{\gamma\,  b_{\ol}}{a_{\ol}}\right)^{|A|-\rho(A)}
\end{multline*}
where 
\begin{equation}\label{r1andr2}
r_1(A)= r_{\G/\V} (A),  \quad\text{and}\quad 
r_2(A)= \rho(A)-r_{\G/\V}(A).
\end{equation}

We then define the Tutte polynomial of vertex coloured ribbon graphs or graphs cellularly embedded in pseudo-surfaces as follows.
\begin{definition}\label{fds}
Let $\G=(V,E)$ be a vertex coloured ribbon graph with  vertex colouring $\V$. Then 
\[T_{cps}(\G;w,x,y) = 
\sum_{A\subseteq E(\G)} 
 w^{r_1(E) -r_1(A) } x^{r_2(E)-r_2(A)} y^{|A|  -\rho(A)},  \]
where $r_1(A)$ and $r_2(A)$ are as given in \eqref{r1andr2}.
\end{definition}

\begin{theorem}[Universality]
Let $\mathcal{G}$ be a minor-closed class of vertex coloured ribbon graphs. Then there is a unique map $U: \mathcal{G}\rightarrow \mathbb{Z}[\alpha, \gamma,  a_{\bs},a_{\bp}^{1/2}, a_{\ol}^{1/2}, b_{\bs},b_{\bp}^{1/2}, b_{\ol}^{1/2}]$ that satisfies \eqref{rrg}.
Moreover,
\begin{multline*}
U(\G)=  
\alpha^{k(\G/\V)} \gamma^{v(\G)-\rho(\G)}
b_{\bs}^{r_1(\G)} 
 b_{\bp}^{r_2(\G)} 
a_{\ol}^{|E|-\rho(\G)} 
\\
\sum_{A\subseteq E} 
\left(\frac{\alpha \gamma\, a_{\bs}}{b_{\bs}}\right)^{r_1(\G)-r_1(A) }
\left(\frac{\gamma\, a_{\bp}}{b_{\bp}}\right)^{r_2(\G)-r_2(A)}
\left(\frac{\gamma\,  b_{\ol}}{a_{\ol}}\right)^{|A|-\rho(A)}.
\end{multline*}
\end{theorem}

The dual of a vertex coloured ribbon graph is a boundary coloured ribbon graph and so $T_{cps}$ cannot satisfy a three variable duality relation. However, it is related to the Tutte polynomial of a boundary coloured ribbon graph, $T_{s}$ through duality.
\begin{theorem}[Duality]
Let $\G$ be a vertex coloured ribbon graph. Then
\[ T_{cps}(\G; w,x,y) = T_{s}(\G^*; y,w,x)\]
\end{theorem}

\subsection{The Tutte polynomial of coloured ribbon graphs (or graphs embedded in pseudo-surfaces)}
This was discussed in Section~\ref{po}. However, for ease of reference we recall its definition:
\[  T_{ps}(\G;w,x,y,z):=  \sum_{A\subseteq E}   w^{r_1(E) -r_1(A) } x^{r_2(E)-r_2(A)} y^{r_3(A)}  z^{r_4(A)} ,  \]
where
\begin{align*}
r_1(A) &:= r_{\G/\V} (A),    & r_3(A) &:=    r_{\G^*/\B} (E)-r_{\G^*/\B} (A^c),  \\
r_2(A) &:= \rho(A)-r_{\G/\V}(A),  & r_4(A) &:= |A| +r_{\G^*/\B} (A^c) - r_{\G^*/\B} (E) -\rho(A).
\end{align*}

\subsection{Relating the four polynomials}
Since there is a hierarchy of ribbon graph structures given by forgetting particular types of colouring,  the resulting Tutte polynomials have a corresponding hierarchy given by specialisation of variables. See Figure~\ref{fffs}.

\begin{figure}[ht]
\begin{center}
\begin{tikzpicture}
\tikzset{node distance=2cm, auto}
  \node (A) {\framebox{coloured rg.}};
  \node (B) [below of=A] {};
  \node (C) [right of=B] {\framebox{vert. col. rg.}};
   \node (D) [left of=B] {\framebox{bound. col. rg.}};
    \node (E) [below of=B] {\framebox{ribbon graph}};
     \draw[->] (A) to node {} (D);
    \draw[->] (A) to node {} (C);
    \draw[<->, dashed] (D) to node {duality} (C);
   \draw[->] (D) to node {} (E);
  \draw[->] (C) to node {} (E);
\end{tikzpicture}
\hspace{20mm}
\begin{tikzpicture}
\tikzset{node distance=2cm, auto}
  \node (A) {$T_{ps}$};
  \node (B) [below of=A] {};
  \node (C) [right of=B] {$T_{cps}$};
   \node (D) [left of=B] {$T_{s}$};
    \node (E) [below of=B] {$T_{cs}$};
     \draw[->] (A) to node [swap] {$w=x$} (D);
    \draw[->] (A) to node {$y=z$} (C);
    \draw[<->, dashed] (D) to node {duality} (C);
   \draw[->] (D) to node [swap] {$y=z$} (E);
  \draw[->] (C) to node {$w=x$} (E);
\end{tikzpicture}
\end{center}
\caption{The hierarchy of ribbon graph structures and the corresponding hierarchy of polynomials}
\label{fffs}
\end{figure}
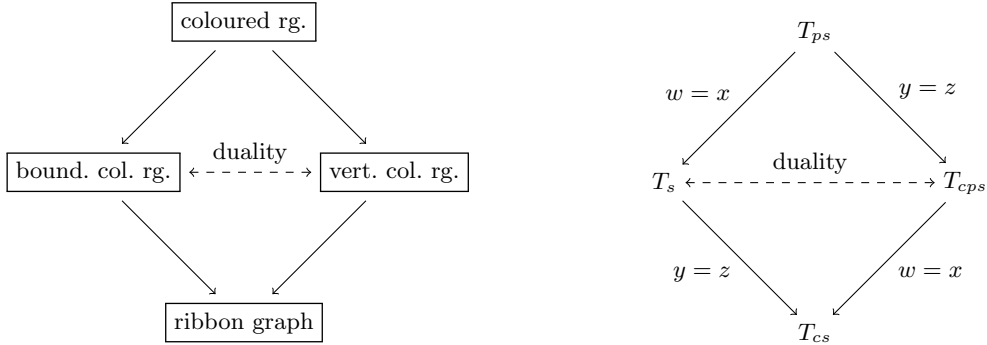

One might question why we should bother with $T_{cs}$, $T_s$ and $T_{cps}$ when they are all just special cases of $T_{ps}$. The philosophy we take here is not one of generalisation, but rather one of finding the correct Tutte polynomial to use for a given setting. So for example, if we have a property of cellularly embedded graphs that we wish to relate to a Tutte polynomial, then we should relate it to the Tutte polynomial for ribbon graphs, rather than the more general Tutte polynomial of coloured ribbon graphs. As a concrete illustration, by \cite{Mo08}, the Tutte polynomial of a ribbon graph can be obtained from the homfly polynomial of a link in a surface, but this result does not (or at least has not been) extended to links embedded in pseudo-surfaces. As an analogy, if we are interested in relating a property of graphs to the Tutte polynomial, then it makes sense to relate it to the classical Tutte polynomial of a graph, rather than the more general Tutte polynomial of, say, a delta-matroid.

\subsection{Relating to  polynomials in the literature}\label{dhjk}
The Tutte polynomials we have constructed have appeared in specialised or restricted forms in the literature. We give a brief overview of these 
connections here.  
The key observation in this direction is that upon setting $\alpha=\beta=\gamma=1$ the  topological Tutte polynomials presented here coincide with those obtained as  canonical Tutte polynomials of Hopf algebras from \cite{KMT} (see Remark~\ref{rem1}). This follows since both sets of invariants satisfy the same deletion-contraction relations.  Consequently we can identify our graph polynomials in that work.

From \cite{KMT}, let $G=(V,E)$ be a graph embedded in a surface $\Sigma$ equipped with a vertex partition $\V$, and let $A\subseteq E$. 
Using $N(X)$ to denote a regular neighbourhood of a subset  $X$ of $\Sigma$, set 
\begin{equation}\label{e.vgs1}
\kappa(A):= \#\mathrm{components}(\Sigma  \backslash N(V \cup  A))  - \#\mathrm{components}(\Sigma ).
\end{equation}
Introduced in \cite{KMT}, the \emph{Krushkal Polynomial} of a vertex partitioned graph in a surface  is 
\[
\widetilde{K}_{(G\subset \Sigma, \mathcal{V})}(x,y,a,b):=\sum_{A\subseteq E(G)}x^{r_{\G/\V}(E)-r_{\G/\V}(A)}y^{\kappa(A)}a^{\rho(A)-r_{\G/\V}(A)}b^{|A|-\rho(A)-\kappa(A)}.
\]
Also introduced in  \cite{KMT}, the \emph{Bollob\'as--Riordan polynomial} of a ribbon graph $\G$ with vertex partition $\V$ is 
\[R_{ ( \G,\mathcal{V})}(x,y,z) := \sum_{A\subseteq E(G)}  (x-1)^{  r_{\G/\V}(E)-   r_{\G/\V}(A) }  y^{|A|-r_{\G/\V}(A)} z^{2( \rho(A)-r_{\G/\V}(A) )}.\]

The above two polynomials are generalisations of the far better-known Krushkal and Bollob\'as--Riordan polynomials.

Introduced in \cite{BR01,BR02}, the \emph{Bollob\'as--Riordan polynomial} of a ribbon graph $\G=(V,E)$ is defined by 
\[ R_{\G}(x,y,z):= \sum_{A \subseteq E}   (x-1)^{r( G ) - r( A )}   y^{|A|-r(A)} z^{\gamma(A)}.\] 

For a graph $G=(V,E)$ embedded in a surface $\Sigma$ (but not necessarily cellularly embedded) the {\em Krushkal polynomial}, introduced by S.~Krushkal in \cite{Kr} for graphs in orientable surfaces,  and extended by C.~Butler in \cite{But} to graphs in non-orientable surfaces, is defined by
\begin{equation}\label{d.kru}
 K_{G\subset \Sigma} ( x,y,a,b ) :=  \sum_{A\subseteq E}  x^{r(G)-r(A)} y^{ \kappa(A)} a^{ \frac{1}{2}s(A)} b^{ \frac{1}{2} s^{\perp}(A)},
   \end{equation}
 where   $s(A):= \gamma(N(V \cup A))$,  $s^{\perp}(A) := \gamma(\Sigma \backslash N(V \cup A))$, and $\kappa$ is as in \eqref{e.vgs1}.
 
Note that we use here the form of the exponent of $y$ from the proof of Lemma~4.1 of \cite{ACEMS}  rather than the homological definition given in \cite{Kr}.

When $\V$ assigns each vertex to its own part of the partition $R_{ ( \G,\mathcal{V})}(x,y,z) =   R_{\G}(x,y,z)$  and $K_{G\subset \Sigma}(x,y,a,b)  = b^{\tfrac{1}{2}\gamma(\Sigma)} \widetilde{K}_{(G\subset \Sigma, \mathcal{V})}(x,y,a,1/b)$. (See \cite{KMT} for details.)

\begin{theorem}\label{nm}
The following hold. 
\begin{enumerate}
\item For a vertex partitioned graph in a surface $(G\subset \Sigma, \mathcal{V})$ and its corresponding coloured ribbon graph $\G$,
  \[\widetilde{K}_{(G\subset \Sigma, \mathcal{V})}(x,y,a,b) = a^{r_2(\G)}  T_{ps} \left(\G; x,\tfrac{1}{a},y,b\right) .\]
\item  For a ribbon graph $\G$ with vertex partition $\V$,
\[  R_{ ( \G,\mathcal{V})}(x,y,z) = (yz^2)^{r_2(\G)} T_{cps} \left(  \G; x-1, \tfrac{1}{yz^2} ,y \right).   \]
\item \label{adsfg}  For a ribbon graph $\G$,
\[   x^{\tfrac{1}{2} \gamma(\G)} R\left(x+1,y, \tfrac{1}{\sqrt{xy}}\right) =T_{cs}(\G;x,y). \]
\end{enumerate}
\end{theorem}
\begin{proof}
The results follow by writing out the state-sum expressions for the polynomials, translating between the language of graphs in pseudo-surfaces and coloured ribbon graphs, and collecting terms. As this is mostly straightforward we omit the details, with the following exception.  
Recall the notation of \eqref{e.vgs1}.
Suppose that $\G$ is the coloured ribbon graph corresponding to a vertex partitioned graph in a surface, and that the boundary colouring of $\G$ is given by $\B$.
By Theorem~\ref{dh},  the connected components of $\Sigma  \backslash N(V)$ are in 1-1 correspondence with the elements of  $V(\G^*/\B)=(\G^*/\B)\ba E$. From this observation, it is readily seen that the connected components of $\Sigma  \backslash N(V \cup A)$, which arise by adding edges to connect components of $\Sigma  \backslash N(V)$, are in 1-1  correspondence with the elements of  $V(\G^*/\B)\cup A=(\G^*/\B)\ba A^c$. It follows that $\kappa (A)=k_{\G^*/\B}(A^c) - k_{\G^*/\B}(E) =   r_{\G^*/\B}(E) - r_{\G^*/\B}(A^c) =r_3(A)$, and so $|A|-\rho(A)-\kappa(A)=r_4(A)$.
\end{proof}

The relations between $R_{ ( \G,\mathcal{V})}$ and $R_{\G}$, and between $K_{G\subset \Sigma}$ and  $\widetilde{K}_{(G\subset \Sigma, \mathcal{V})}$, immediately give the following corollary.
\begin{corollary}\label{dshk}
The following hold. 
\begin{enumerate}
\item For a  graph in a surface $(G\subset \Sigma, \mathcal{V})$ described as coloured ribbon graph $\G$,
\[K_{G\subset \Sigma } (x,y,a,b) = a^{r_2(\G)}   b^{\tfrac{1}{2}\gamma(\Sigma)} T_{ps} \left(\G; x,\tfrac{1}{a},y,\tfrac{1}{b}\right). \]

\item For a ribbon graph $\G$ and for a vertex colouring $\V$ that assigns a unique colour to each vertex, 
\[  R_{ \G}(x,y,z) = (yz^2)^{\tfrac{1}{2}\gamma(\G)} T_{cps} \left(  \G; x-1, \tfrac{1}{yz^2} ,y \right).  \]
\end{enumerate}
\end{corollary}

\begin{remark}
The polynomials $R_{ \G}(x,y,z)$ and $K_{G\subset \Sigma } (x,y,a,b)$ are the two most studied topological graph polynomials in the literature. However, there is a problematic aspect to both polynomials  in that neither has a `full' recursive deletion-contraction definition that reduces the polynomial to a linear combination of polynomials of trivial graphs (as is the case with the classical Tutte polynomial of a graph). Instead the known deletion-contraction relations reduce the polynomials to those of graphs in surfaces on one vertex.  

The significance of Corollary~\ref{dshk}  is that it tells us that by extending the domains of the polynomials to graphs (cellularly) embedded in pseudo-surfaces, we obtain versions of the polynomials with `full' recursive deletion-contraction definitions. This indicates that the Bollob\'as--Riordan polynomial is not a Tutte polynomial for cellularly embedded graphs in \emph{surfaces}, as it has been considered to be, but in fact a Tutte polynomial    for cellularly embedded graphs in \emph{pseudo-surfaces}. A similar comment holds for the Krushkal polynomial. Moreover, that most of the known properties of the Bollob\'as--Riordan polynomial only hold for the specialisation $x^{\tfrac{1}{2} \gamma(\G)} R\left(x+1,y, \tfrac{1}{\sqrt{xy}}\right)$ is explained by the fact that, by Theorem~\ref{nm}\ref{adsfg},  this specialisation is the Tutte polynomial for a ribbon graph, which is where the ribbon graph results naturally belong.
\end{remark}

\section{Activities expansions}\label{s.7}
In this remaining section we consider analogues of the activities expansions of the Tutte polynomial, as in \eqref{d3}. In the classical case, the activities expansion for the Tutte polynomial expresses it as a sum over spanning trees (in the case where the graph is connected). In the setting of topological graph polynomials, we instead consider quasi-trees. 

A ribbon graph is a \emph{quasi-tree} if it has exactly one boundary component. It is a \emph{quasi-forest} if each of its components has exactly one boundary component. A ribbon subgraph of $\G$ is \emph{spanning} if it contains each vertex of $\G$. Note that a genus 0 quasi-tree is a tree and a genus 0 quasi-forest is a forest. In this section we work with connected ribbon graphs and  spanning quasi-trees, for simplicity, but the results  extend to the non-connected case by considering maximal spanning quasi-forests instead.

An activities expansion for the Bollob\'as--Riordan polynomial of orientable ribbon graphs was given by A.~Champanerkar, I.~Kofman, and N.~Stoltzfus in \cite{Champ}. This was quickly extended to non-orientable ribbon graphs by  F.~Vignes-Tourneret in \cite{VT}, and independently by E.~Dewey in unpublished work~\cite{dewey}.  C.~Butler in \cite{But} then extended this to give a quasi-tree expansion for the Krushkal polynomial. Each of these expansions expresses the graph polynomial as a sum over quasi-trees, but they all include a Tutte polynomial of an associated graph as a summand. Most recently, A.~Morse in \cite{Morse:2017aa} gave a spanning tree expansion for the 2-variable Bollob\'as--Riordan polynomial  of a delta-matroid that specialised to one for the 2-variable Bollob\'as--Riordan polynomial  of a ribbon graph. 

In this section, we give a spanning tree expansion for $T_{ps}$. However, it is more convenient to work with a normalisation of $T_{ps}$, as follows.

We consider the polynomial $U$ of coloured ribbon graphs (see Theorems~\ref{fdgu}, \ref{zxa}, and \ref{mnc}) specialised at $ \alpha= \beta= \gamma, a_{\bs}=a_{\bp}^{1/2}=a_{\olc}= a_{\olh}^{1/2}=1$. For convenience we denote the resulting polynomial by $P$. It is given by 
\begin{equation}\label{das}
 P(\G)=  
\sum_{A\subseteq E} 
\left(b_{\bs}\right)^{r_1(A) }
\left(b_{\bp}\right)^{r_2(A)}
\left( b_{\olc}\right)^{r_3(A)}
\left(b_{\olh}\right)^{ r_4(A)}
\end{equation}
and, by Theorem~\ref{sahj}, satisfies the deletion-contraction relations
\begin{equation}\label{sdfh}
P(\G) = 
  \begin{cases} 
   P(\G\ba e)+g(e)P(\G/e) & \\
   1&  \text{if } \G \text{ is edgeless}
  \end{cases}
\end{equation}
where
\[
g(e) = 
\begin{cases}
b_{\bs} &  \text{if $e$ is not a loop in $\G$, not a loop in $\G/\V$} \\
b_{\bp} & \text{if $e$ is not a loop in $\G$, a loop  in $\G/\V$}\\
b_{\olc}& \text{if $e$ is a bridge in $\G^*/\B$,  an orientable loop in $\G$}\\
b_{\olh}& \text{if $e$ is not a bridge in $\G^*/\B$,  an orientable loop in $\G$}\\
\sqrt{b_{\bp}b_{\olh}}& \text{if $e$ is a non-orientable loop in $\G$}.
\end{cases}
\]

By comparing \eqref{udd} and \eqref{das} it is easily seen how  $U(\G)$ (and hence $T_{ps}$) can be recovered from $P(\G)$.

We say that a ribbon graph $\G$ is the  \emph{join} of ribbon graphs $\G'$ and $\G''$, written $\G' \vee \G''$, if $\G$ can be obtained by identifying an arc on the boundary of a vertex of  $\G'$ with an arc on the boundary of a vertex of $\G''$. The two  vertices with identified arcs make a single vertex of $\G$. (See, for example, \cite{EMMbook,Mo12} for  elaboration of this operation.) If $\G$ is coloured then the colour class of the vertices of $\G'$ and $\G''$ should be identified in forming $\G$, as should their boundary colour classes.

\begin{proposition}\label{veesqcup}
Let $\G$ and $\G'$ be coloured ribbon graphs. Then $P(\G \vee \G')=P(\G \sqcup \G')$.
\end{proposition}
\begin{proof}
The result follows easily by computing the expressions using the deletion-contraction relations and noting that $P$ takes the value 1 on all edgeless coloured ribbon graphs.
\end{proof}

Equation~(\ref{d.kru}) and Proposition~\ref{veesqcup} immediately give the following lemma.

\begin{lemma}\label{hdda}
\[P(\G)=
\begin{cases}
(1+b_{\olc})P(\G\backslash e) & \text{if $e$ is a trivial orientable loop with} \\
& \quad\quad\quad \quad\quad\quad\quad\quad\quad \text{two boundary colours} \\
(1+b_{\olh})P(\G\backslash e) & \text{if $e$ is a trivial orientable loop with} \\
&\quad\quad\quad \quad\quad\quad\quad\quad\quad \text{one boundary colour} \\
(1+\sqrt{b_{\bp}b_{\olh}})P(\G\backslash e) & \text{if $e$ is a trivial non-orientable loop} \\
(1+b_{\bs})P(\G/e) & \text{if $e$ is a bridge with two vertex colours} \\
(1+b_{\bp})P(\G/e) & \text{if $e$ is a bridge with one vertex colour} \\
P(\G\backslash e)+g(e)P(\G/e) & \text{if $e$ is otherwise} \\
1 & \text{if $\G$ is edgeless}
\end{cases}
\]
\end{lemma}

We consider resolution trees for the computation of $P(\G)$ via the deletion-contraction relations of Lemma~\ref{hdda}. An example of one  is given in Figure~\ref{gsdh}. 
We use the following terminological conventions in our resolution trees for $P(\G)$.  
The \emph{root} is the node corresponding to the original graph. 
A \emph{branch} is a path from the root to a leaf, and the branches are in 1-1 correspondence with the leaves.
The leaves are of height 0, with the height of the other nodes given by the distance from a leaf (so the root is of height $|E(G)|$). 
Note that the advantage of using $P(\G)$ rather than $U(\G)$ is that there are fewer leaves in the resolution tree.

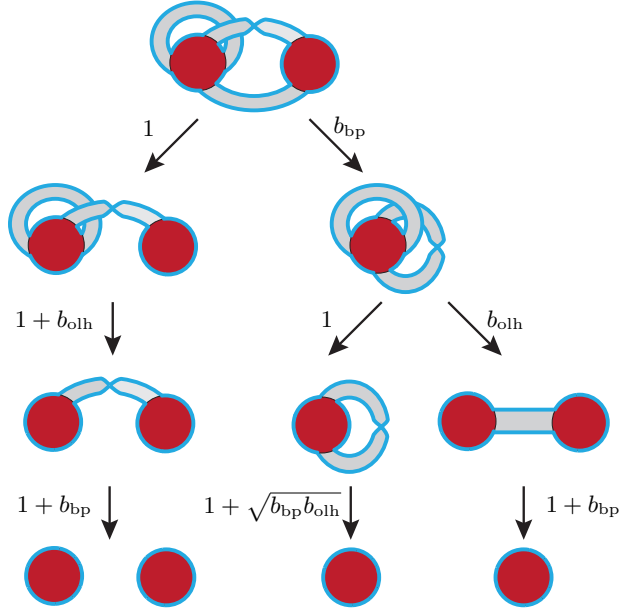
\begin{figure}[ht]
\centering
\labellist
\small\hair 2pt
\pinlabel $1$ at 181 617
\pinlabel $b_{\bp}$ at 440 617
\pinlabel $1+b_{\olh}$ at  60 370
\pinlabel $1$ at  410 370
\pinlabel $b_{\olh}$ at  640 370
\pinlabel $1+b_{\bp}$ at  60 130
\pinlabel $1+\sqrt{b_{\bp}b_{\olh}}$ at  340 130
\pinlabel $1+b_{\bp}$ at  740 130
\endlabellist
\includegraphics[width=80mm]{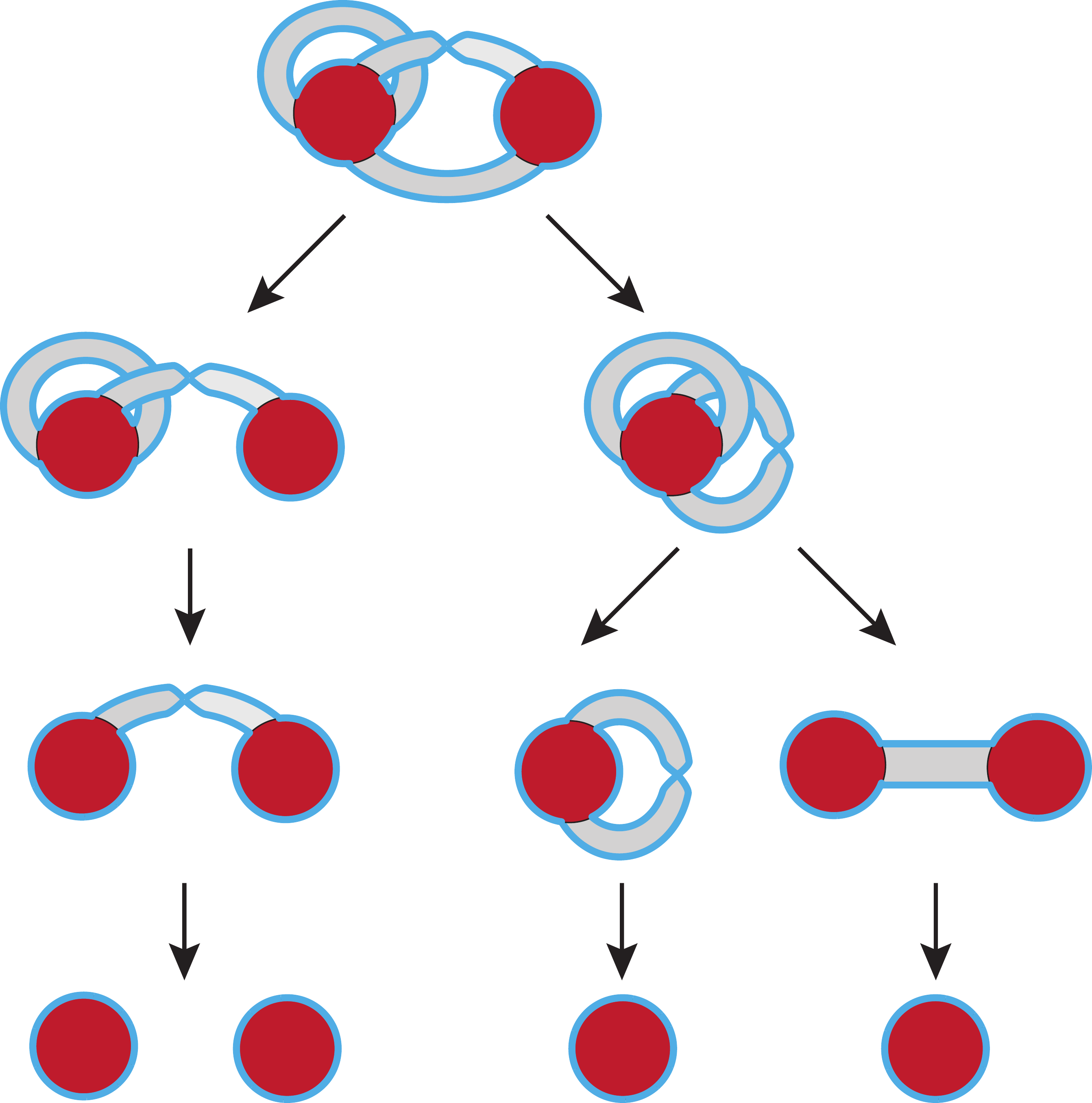}
\caption{A resolution tree. In this example, each ribbon graph has a single vertex colour class and a single boundary colour class }
\label{gsdh}
\end{figure}

\begin{lemma}\label{zxzc}
If $\G$ is a connected coloured ribbon graph then the set of spanning quasi-trees of $\G$ is in one-one correspondence with the set of leaves of the resolution tree of $P(\G)$. Furthermore, the correspondence is given by deleting the set of edges of $\G$ that are deleted in the branch that terminates in the node. 
\end{lemma}
\begin{proof}
First consider a leaf of the resolution tree. Let $D$ be the set of edges that are deleted in the branch containing that leaf, and let $C$ be the set of edges that are contracted in that branch. Then, since the order of deletion and contraction of edges does not matter, the ribbon graph at the node is give by $\G\ba D /C$. 
Since, in the resolution tree, bridges are never deleted and trivial loops are never contracted we have that  $\G$ and $\G\ba D /C$ have the same number of components, and so $\G\ba D /C$ consists of a single vertex and no edges. Then $\G\ba D /C$ has exactly one boundary component. Since contraction does not change the number of boundary components, it follows that $\G\ba D$ has exactly one boundary component and is hence a spanning quasi-tree of $\G$.

Now suppose that $\T$ is a spanning quasi-tree of $\G$. We need to show that $\T= \G \ba D$ where $D$ is the set of edges that are deleted in some branch of the resolution tree.  Let $C=E(\G) \ba D$.
Since $\G\ba D=\T$ is a quasi-tree and contraction does not change the number of boundary components, $\G\ba D /C$ consists of a single vertex. 
The order in which the edges of $\G$ are deleted and contracted does not change the resulting ribbon graph and so if we compute $\G\ba D /C$  applying deletion and contraction in any order we will never delete a bridge or contract a trivial orientable loop (otherwise $\G\ba D /C$ would not be a single vertex). It follows that there is a branch in the resolution tree in which the edges in $D$ are deleted and the edges in $C$ are contracted.  Thus $\T= \G \ba D$ where $D$ is the set of edges that are deleted in some branch of the resolution tree, completing the proof of the correspondence. 
\end{proof}

We are aiming to obtain an `activities expansion' for $P$ (and hence $T_{ps}$ and its specialisations). To do this we need to be able to express certain properties of an edge in the coloured ribbon graph at a given node of the resolution tree.  These expressions will be the analogues for coloured ribbon graphs of internal and external activities in the classical case.

We  make use of partial duals of ribbon graphs, introduced in \cite{Ch09}.
Let $\G=(V, E)$ be a ribbon graph, $A\subseteq E$ and regard the boundary components of the  ribbon subgraph $(V,A)$ of $\G$ as curves on the surface of $\G$. Glue a disc to $\G$ along each of these curves by identifying the boundary of the disc with the curve, and remove the interior of all vertices of $\G$. The resulting ribbon graph is the {\em partial dual} $\G^{A}$ of $\G$.  (See for example \cite{EMMbook} for further background on partial duals.)
Observe that if $\T$ is a spanning quasi-tree of $\G$ then $\G^{E(\T)}$ has exactly one vertex.

An edge $e$ in a one-vertex ribbon graph is said to be \emph{interlaced} with an edge $f$ if the ends  of $e$ and $f$ are  met in the cyclic order $e,f,e,f$ when travelling round the boundary of the vertex. 

\begin{definition}\label{efvh}
Let $\G$ be a connected coloured ribbon graph with a choice of edge order $<$, and let $\T$ be a spanning quasi-tree of $\G$. Let $\G^{\T}$ be the partial dual of its underlying  ribbon graph. Note that $\G^{\T}$ has exactly one vertex and its edges correspond with those of $\G$.

\begin{enumerate}
\item An edge $i$ is said to be \emph{vertex essential} if it is in a cycle of $\G/\V$ consisting of edges in $\T$ greater or equal to $i$ in the edge order. It is \emph{vertex inessential} otherwise.

\item An edge $i$ is said to be \emph{boundary essential} if after the deletion of all edges of $\G^*/\B$ which are in $\T$ and  greater or equal to $i$ in the edge order, the end vertices of the edge are in different components. It is \emph{boundary inessential} otherwise.

\item An edge $i$ of $\G$ is said to be \emph{internal} if it is in $E(\T)$, and \emph{external} otherwise.

\item An edge $i$ of $\G$ is said to be \emph{live} with respect to $(\T,<)$ if in $\G^{\T}$ it is not interlaced with any lower ordered edges. It is said to be \emph{dead} otherwise. The edge is said to be \emph{live (non-)orientable} if it is live and forms an (non-)orientable loop in  $\G^{\T}$. 

\item  Given an edge $i$ of $G$, consider the ribbon subgraph of  $\G^{\T}$ consisting of  all internal edges that are greater or equal to $i$ in the edge order. Arbitrarily orient each of the boundary components of this subgraph, and consider the corresponding oriented closed curves on  $\G^{\T}$. An edge $i$ of $\G$ is said to be \emph{consistent} (respectively \emph{inconsistent}) with respect to $(\T,<)$ if the boundary of $i$ intersects exactly one of the oriented curves and the directed arcs where they intersect are  consistent (respectively inconsistent) with some orientation of the boundary of the edge $i$.
\end{enumerate}
\end{definition}

\begin{lemma}\label{ashu}
Let $\G$ be a connected coloured ribbon graph with a choice of edge order $<$, and $\T$ be a spanning quasi-tree of $\G$. Let $\h$ be the coloured ribbon graph at the node at height $i$ in the branch determined by $\T$. Then the following hold.
\begin{enumerate}
\item  \label{ashu1} The $i$-th edge is a loop in $\h/\V$ if and only if it is vertex essential.

\item \label{ashu2}The $i$-th edge is a bridge in $\h^*/\B$ if and only if it is boundary essential.

\item \label{ashu3} The $i$-th edge is a trivial orientable loop in $\h$ if and only if it is externally live orientable with respect to $\T$ and $<$.

\item\label{ashu4}  The $i$-th edge is a trivial non-orientable loop in $\h$ if and only if it is  live non-orientable with respect to $\T$ and $<$.

\item\label{ashu5} The $i$-th edge is a bridge in $\h$ if and only if it is internally live orientable with respect to $\T$ and $<$.

\item\label{ashu6} The $i$-th edge is an orientable loop in $\h$ if and only if it is consistent with respect to $\T$ and $<$.

\item\label{ashu7} The $i$-th edge is a non-orientable loop in $\h$ if and only if it is inconsistent with respect to $\T$ and $<$.

\item\label{ashu8} The $i$-th edge is not a loop in $\h$ if and only if it is neither consistent nor inconsistent with respect to $\T$ and $<$.
\end{enumerate}
\end{lemma}
\begin{proof}
For the proof let $C$ denote the set of edges of $\G$ that are in $\T$ and higher than $i$ in the edge order,
and let $D$ denote the set of edges of $\G$ that are not in $\T$ but are higher than $i$ in the edge order.
Note that $\h=\G\ba D /C$.

Item~\ref{ashu1} follows from the observation that an edge $e$ is a loop in the graph $(\G/V)\ba D /C$ if and only if it is in a cycle of the subgraph $C\cup e$ of the graph $\G/\V$.

The proof of item~\ref{ashu2} follows from the observation that 
$  \h^*/\B =   (\G\ba D /C)^*/\B = (\G^*/\B) / D \ba C$.

For the remaining items, start by observing that, with $C$ and $D$ as above, 
\[\h = \G/ C \ba D =   \G^C \ba (D\cup C) = \left(\left( \G^C \ba (D\cup C)\right)^{(  E(\T)\ba  C )} \right)^{(  E(\T)\ba  C )}  =  ( \G^{\T} \ba (D\cup C))^{(  E(\T)\ba  C )} .\]
In particular, $\h^{(  E(\T)\ba  C )}=\G^{\T} \ba (D\cup C)$, and $\h^{(  E(\T)\ba  C )}$ has exactly one vertex (since $\G^{\T}$ does).

For item~\ref{ashu3}, suppose that $i$ is a trivial orientable loop in $\h$. Then $i$ must be external since otherwise $\h^{(  E(\T)\ba  C )}$ would have more than one vertex.
Then since $i\in \h$ but $i\notin  E(\T)\ba  C$ it follows from the properties of partial duals that $i$ must be a  trivial orientable loop in $\h^{(  E(\T)\ba  C )} = \G^{\T} \ba (D\cup C)$. Thus $i$ must be externally live orientable.  

Conversely, if $i$ is externally live orientable then it must be a trivial orientable loop in $ \G^{\T} \ba (D\cup C)= \h^{(  E(\T)\ba  C )}$. As it is not in $\T$, it must then also be a trivial orientable loop in $(\h^{(  E(\T)\ba  C )})^{(  E(\T)\ba  C )}=\h$. This completes the proof of item~\ref{ashu3}.

For item~\ref{ashu4}, $i$ is a trivial non-orientable loop in $\h$ (regardless of whether it is internal or external) if and only if it is a trivial non-orientable loop in $\h^{(  E(\T)\ba  C )}= \G^{\T} \ba (D\cup C)$. Rephrasing this condition tells us that this happens if and only if $i$ is live non-orientable.

For item~\ref{ashu5}, suppose that $i$ is a bridge in $\h$. Then it must be internal as otherwise $\h^{(  E(\T)\ba  C )}$ would have more than one vertex. It follows that $i$ must be a trivial orientable loop in $\h^{(  E(\T)\ba  C )}= \G^{\T} \ba (D\cup C)$. Thus it must be internally live orientable. 

Conversely, suppose that $i$ is internally live orientable. Then it must be a trivial orientable loop in $ \G^{\T} \ba (D\cup C)= \h^{(  E(\T)\ba  C )}$. As it is in $\T$, it must then also be a bridge in $(\h^{(  E(\T)\ba  C )})^{(  E(\T)\ba  C )}=\h$. This completes the proof of item~\ref{ashu5}.

For items~\ref{ashu6}--\ref{ashu8}, the edge $i$ is consistent or inconsistent if and only if in $\G^{\T} \ba (D\cup C)=\h^{(  E(\T)\ba  C )}$ it touches one boundary component of the ribbon subgraph of $ \h^{(  E(\T)\ba  C )}  $ on the edges $E(\T)\ba  C$. Thus  $i$ is consistent or inconsistent if and only if it is a loop in $ (\h^{(  E(\T)\ba  C )} )^{(  E(\T)\ba  C )}  =\h  $. It is not hard to see that the conditions of consistent and inconsistent determine whether the loop is orientable or non-orientable. This completes the proof of the final three items and of the lemma.
\end{proof}

\begin{definition}\label{mku}
Let $\G$ be a connected coloured ribbon graph with a choice of edge order $<$, and $\T$ be a spanning quasi-tree of $\G$. Then we say that an edge $e$ is of \emph{activity type} $i$ with respect to  $<$ and $\T$ according to the following.
\begin{itemize}
\item ~\emph{Activity type} $1$ if it is externally live orientable and boundary essential.
\item ~\emph{Activity type} $2$ if it is  externally live orientable and boundary inessential.
\item ~\emph{Activity type} $3$ if externally live orientable.
\item ~\emph{Activity type} $4$ if it is internally live orientable and vertex inessential.
\item ~\emph{Activity type} $5$ if it is internally live orientable and vertex essential.
\item ~\emph{Activity type} $6$ if is internal, not of activity type 4, boundary and vertex inessential and neither consistent nor inconsistent.
\item ~\emph{Activity type} $7$ if is internal, not of activity type 5, boundary inessential, vertex essential and neither consistent nor inconsistent.
\item ~\emph{Activity type} $8$ if is internal, not of activity type 1,  boundary essential, vertex essential and  consistent.
\item ~\emph{Activity type} $9$ if is internal, not of activity type 2, boundary inessential, vertex essential and  consistent.
\item ~\emph{Activity type} $10$ if is internal, not of activity type 3, but inconsistent.
\end{itemize}
We let $N(\G,<,\T,i)$ be the numbers of edges in $\G$ of  \emph{activity type} $i$ with respect to  $<$ and $\T$.
\end{definition}

\begin{theorem}\label{qtr}
Let $\G$ be a connected coloured ribbon graph with a choice of edge order $<$, and $\T$ be a spanning quasi-tree of $\G$. Then
\[  P(\G)= \sum_{ \substack{  \text{spanning} \\ \text{quasi-trees }  \\ \T \text{ of } \G} }    \prod_{i=1}^{10}   C(i)^{N(i)}  ,  \]

Where $N(i):=N(\G,<,\T,i)$, and   
$C(1)= 1+b_{\olc}  $, 
$C(2)=1+b_{\olh}$,
$C(3)= 1+\sqrt{b_pb_{\olh}}$,
$C(4)= 1+b_{\bs}$,
$C(5)= 1+b_{\bp}$,
$C(6)= b_{\bs}$,
$C(7)= b_{\bp}$,
$C(8)= b_{\olc}$,
$C(9)= b_{\olh}$,
$C(10)= \sqrt{b_{\bp}b_{\olh}}$.
\end{theorem}
\begin{proof}
By Lemma~\ref{zxzc}, the branches of the resolution tree for $P(\G)$ are in 1-1 correspondence with the spanning quasi-trees of $\G$. $P(\G)$ can be obtained by taking the product of the labels of the edges in each branch of the resolution tree then summing over all branches. These labels are determined by looking at a given node of height $i$, then applying the deletion-contraction relation \eqref{hdda}. The particular coefficient obtained (i.e., which of the ten cases of the deletion-contraction relation is used) depends upon the edge $i$ at the node. This type can be rephrased in terms of activities using   Lemma~\ref{ashu}. The activity types of Definition~\ref{mku} are exactly these rephrasings, and the $N(i)$ are the corresponding terms, excluding those that contribute a coefficient of 1.  The theorem follows.
\end{proof}

\begin{remark}
It is instructive to see how Theorem~\ref{qtr} specialises to the activities expansion for the Tutte polynomial when $\G$ is a plane ribbon graph in which each boundary component and each vertex has a distinct colour.
In this case,  using Lemma~\ref{ashu}, we see that only edges with activity types 1,4, and 7 can occur. Thus $P(G)$ specialises to 
\[P(\G)= \sum_{ \substack{  \text{spanning} \\ \text{quasi-trees }  \\ \T \text{ of } \G} }  (1+b_{\olc})^{N(1)}(1+b_{\bs})^{N(4)}b_{\bp}^{N(7)}.\]
Since the total number of terms $1+b_{\bs}$ and $b_{\bp} $ in each summand equals the number of contracted edges in a branch of the resolution tree which, as we are working with plane graphs, is $r(\G)$, we can write this as  
\[P(\G)= b_{\bp}^{r(\G)}  \sum_{ \substack{  \text{spanning} \\ \text{quasi-trees }  \\ \T \text{ of } \G} }  (1+b_{\olc})^{N(1)}  ((1+b_{\bs})/b_{\bp})^{N(4)}.\]
Finally, as in \cite{Champ}, because $\G$ is plane every quasi-tree is a tree, so $N(4)$ equals the number of internally active edges, and  $N(1)$ equals the number of externally active edges.  Hence
\[P(\G)=    b_{\bp}^{r(\G)}    \sum_{ \substack{  \text{spanning trees} \\ \T \text{ of } \G } }  ((1+b_{\bs})/b_{\bp})^{IA} (1+b_{\olc})^{EA} .\]
From the activities expansion and the universality theorem for the Tutte polynomial, \eqref{d3} and Theorem~\ref{thm1}, we find that for plane graphs, $P(\G)$ is given by 
\[P(\G)=
\begin{cases}
(1+b_{olc})P(\G\backslash e) & \text{if $e$ is a trivial orient. loop with two boundary colours} \\
(1+b_{bs})P(\G/e) & \text{if $e$ is a bridge with two vertex colours} \\
P(\G\backslash e)+b_{\bp}P(\G/e) & \text{if $e$ is otherwise}
\end{cases}
\]
and its value of 1 on edgeless ribbon graphs. This is readily verified to coincide with Lemma~\ref{hdda} when it is restricted to plane graphs.
\end{remark}

A quasi-tree expansion for $U(\G)$, and hence for all the polynomials considered here, can be obtained from Theorem~\ref{qtr} since $U(\G)$ can be obtained from $P(\G)$.  For brevity we will omit explicit formulations of these.

\section{Proofs for Section~\ref{po}}\label{tsr}

Recall that for a graph $G$ we have the following.
\begin{equation}\label{ewe}
G/e^c = \raisebox{-3mm}{\includegraphics[scale=.4]{figs/b13}}   \iff e \text{ is a bridge,} \qquad
G\ba e^c = \raisebox{-3mm}{\includegraphics[scale=.4]{figs/b14}}\iff e \text{ is a loop.}
\end{equation}
Similar results hold for ribbon graphs.
\begin{proposition}\label{bhj}
Let $\G$ be a ribbon graph and $e$ be an edge in $\G$. Then, after removing any isolated vertices, 
\[
\begin{array}{l l c l l}
\G/e^c = \raisebox{-3mm}{\includegraphics[height=12mm]{figs/ol}} & \! \iff  \! e \text{ not a doop,}   &  & 
\G\ba e^c = \raisebox{-3mm}{\includegraphics[height=12mm]{figs/ol}} &\! \iff  \! e \text{  an orientable loop,}
\\
\G/e^c = \raisebox{-3mm}{\includegraphics[height=10mm]{figs/b}}  & \!  \iff \!  e \text{  an orientable doop,}   &  & 
\G\ba e^c = \raisebox{-3mm}{\includegraphics[height=10mm]{figs/b}}& \! \iff  \! e \text{  not a loop,}
\\
\G/e^c = \raisebox{-3mm}{\includegraphics[height=12mm]{figs/nlrg}}   &\!  \iff  \! e \text{  a non-orient. doop,}   &  & 
\G\ba e^c = \raisebox{-3mm}{\includegraphics[height=12mm]{figs/nlrg}} & \! \iff  \! e \text{  a non-orient. loop.}
\end{array}
\]
\end{proposition}
\begin{proof}
The results for $\G\ba e^c$ are trivial. The results for $\G/e^c$ follow since the boundary components of $\G/e^c$ are determined by those boundary components of $\G$ which intersect $e$.
\end{proof}

\begin{lemma}\label{asr}
Let $\G=(V,E)$ be a coloured ribbon graph with vertex colouring $\V$ and boundary colouring $\B$. 
For each row in the tables below, after removing any isolated vertices, $\G/e^c$ or $\G\ba e^c$  is the coloured ribbon graph in the first column if and only if $e$ has the properties in $\G^*/\B$, $\G$, and $\G/\V$ given in the remaining columns.
\begin{center}
\begin{tabular}{|c||c|c|c|}
\hline
$\G/e^c$ & $e$ in $\G^*/\B$ & $e$ in $\G$  & $e$ in $\G/\V$ \\
\hline \hline
\raisebox{-3mm}{\includegraphics[height=8mm]{figs/bs}} $(=\bs)$   & loop &orientable doop& bridge \\
\hline
  \raisebox{-3mm}{\includegraphics[height=8mm]{figs/bp}}  $(=\bp)$ &loop &  orientable doop & not a bridge
\\\hline
\raisebox{-3mm}{\includegraphics[height=10mm]{figs/olc}} $(=\olc)$&not a loop& not a doop & not a bridge
  \\\hline
 \raisebox{-3mm}{\includegraphics[height=10mm]{figs/olh}} $(=\olh)$&loop & not a doop & not a bridge
  \\\hline
 \raisebox{-3mm}{\includegraphics[height=10mm]{figs/nl}}  $(=\nl)$ & loop & non-orientable doop & not a bridge
 \\\hline
\end{tabular}
\end{center}

\bigskip

\begin{center}
\begin{tabular}{|c||c|c|c|}
\hline
$\G\ba e^c$ & $e$ in $\G^*/\B$ & $e$ in $\G$  & $e$ in $\G/\V$ \\
\hline \hline
 \raisebox{-3mm}{\includegraphics[height=8mm]{figs/bs}}  $(=\bs)$ & not a bridge & not a loop& not a loop
 \\ \hline
 \raisebox{-3mm}{\includegraphics[height=8mm]{figs/bp}}   $(=\bp)$ &  not a bridge &not a loop&  loop  
  \\\hline
\raisebox{-3mm}{\includegraphics[height=10mm]{figs/olc}} $(=\olc)$  &  bridge & orientable loop & loop
  \\\hline
\raisebox{-3mm}{\includegraphics[height=10mm]{figs/olh}}  $(=\olh)$&   not a bridge&  orientable loop & loop 
  \\\hline
\raisebox{-3mm}{\includegraphics[height=10mm]{figs/nl}}  $(=\nl)$  &  not a bridge & non-orientable loop &loop 
\\\hline
\end{tabular}
\end{center}
\end{lemma}
\begin{proof}
Let $f\neq e$ be an edge of $\G$ and let $\tilde{\G}$ denote the underlying (non-coloured) ribbon graph of $\G$. 

Consider $\G/ f$. 
Since contraction preserves boundary components, if $e$ touches boundary components of a single colour (respectively, two colours) in $\G$, then it does the same in $\G/f$, and hence also in $\G/e^c$.   It follows that $e$ is a loop in $\G^*/\B$ if and only if it is one in $(\G/ e^c)^*/\B$.

Trivially  $\widetilde{(\G/ f)} =\tilde{\G}/ f$. Then by Proposition~\ref{bhj}, the doop-type determines the underlying ribbon graph of $\G/e^c$.

By considering Table~\ref{tablevp}, and in each case forming the abstract graphs $ (\G /f ) /\V$ and   $(\G/\V)/ f$ we see that  $ (\G /f ) /\V=  (\G/\V)/ f$.  (It is worth emphasising that the contraction on the left-hand side is ribbon graph contraction and that on the right is graph contraction.) It follows that  $ (\G / e^c ) /\V=  (\G/\V)/ e^c$ and hence  $(\G / e^c ) /\V$ is a bridge if and only if is a bridge in $(\G/\V)/ e^c$, which happens if and only if it is a bridge in  $\G/\V$.

Collecting these three facts gives the results about $\G/e^c$. The remaining five cases can be proved by an analogous argument, or by duality.
\end{proof}

Note that while Lemma~\ref{asr} appears to imply that there are 25 edge-types, not all types are realisable. For example, type $ (\bs,\olc )$ edges are not possible since if $e$ is a loop in $\G$ it must be a loop in $\G/ \V$. Furthermore, there is redundancy in the descriptions of the cases in Lemma~\ref{asr} since, for example, using again that a loop in $\G$ must be a loop in $\G/ \V$, we can see that the last clause in $\G\ba e^c=\olh$ is not needed.

The following are standard facts about the rank function of a graph.
\begin{equation}\label{e.7a}
r(G) =
  \begin{cases}
    r(G\ba e)       & \quad \text{if } e \text{ is a bridge}\\
    r(G\ba e)+1  & \quad \text{otherwise} \\
  \end{cases}
  \end{equation}
\begin{equation}\label{e.7b}
r(G) =
  \begin{cases}
    r(G/e)       & \quad \text{if } e \text{ is a loop}\\
    r(G/e)+1  & \quad \text{otherwise} \\
  \end{cases}
  \end{equation}

\begin{lemma}\label{asew}
Let $\G$ be a ribbon graph. Then
\begin{equation}\label{e.8a} \rho(\G) =
  \begin{cases}
   \rho(\G\ba e)+1       & \quad \text{if }  e\text{ is a not a doop},\\
   \rho(\G\ba e)  & \quad \text{if }  e\text{ is an orientable doop},\\
   \rho(\G\ba e)+\frac{1}{2}  & \quad \text{if }  e\text{ is a non-orientable doop};\\
  \end{cases}
\end{equation}
and
\begin{equation}\label{e.8b} \rho(\G) =
  \begin{cases}
    \rho(\G/e)+1       & \quad \text{if }  e\text{ is a not a loop},\\
   \rho(\G/e)  & \quad \text{if }  e\text{ is an orientable loop},\\
   \rho(\G/e)+\frac{1}{2}  & \quad \text{if }  e\text{ is a non-orientable loop}.
  \end{cases}
\end{equation}

\end{lemma}
\begin{proof}
The identities are easily verified by writing $\rho(E)= \frac{1}{2}\left( |E|+ |V| -b(E)\right)$, via \eqref{e.vpa}, and considering  how deletion changes the number of boundary components, and how contraction changes the number of vertices (it does not change the number of boundary components), in the various cases. We omit the details.
\end{proof}

\begin{lemma}\label{rsf}
Let $\G=(V,E)$ be a coloured ribbon graph with vertex colouring $\V$ and boundary colouring $\B$, and  $r_1,\ldots, r_4$ be as in Theorem~\ref{zxa}. 
Recall that $r_{k, \G} (A):= r_k(\G \ba A^c) $, for $k=1,2,3,4$.
Then,  for $k=1,2,3,4$ and $i,j \in \{ \bs,\bp,\olc,\olh,\nl \}$ we have the following.

Firstly,
\begin{equation}\label{hasw1}
  r_k(\G) = r_k(\G\ba e)  +  \delta_{k,i}   ,  
  \end{equation}
and if $e\notin A$, 
\begin{equation}\label{hasw2}
   r_{k,\G}(A) = r_{k,\G \ba e} (A) +  \delta_{k,i}   ,  
     \end{equation}
where $ \G /  e^c =i $ and 
\[   \delta_{k,j} =
\begin{cases}
1 & \text{when } (k,j) \text{ is one of } (1,\bs), (2,\bp), (3,\olc),  (4,\olh)  \\
\tfrac{1}{2} &  \text{when } (k,j) \text{ is one of }   (2,\nl), (4,\nl) \\
0 & \text{otherwise}
\end{cases}
 \]

Secondly,
\begin{equation}\label{hasw3}
  r_k(\G) = r_k(\G/e)  +  \epsilon_{k,j} ,  
    \end{equation}
and if $e\in A$, 
\begin{equation}\label{hasw4}
   r_{k,\G}(A) = r_{k,\G/e} (A\ba e) +  \epsilon_{k,j}  ,  
     \end{equation}
where  $\G\ba e^c =j$ and
\[   \epsilon_{k,j} =
\begin{cases}
1 & \text{when } (k,j) \text{ is one of } (1,\bs),  (2,\bp), (3,\olc), (4,\olh) \\
\tfrac{1}{2} &  \text{when } (k,j) \text{ is one of }  (2,\nl), (4,\nl)\\
0 & \text{otherwise.}
\end{cases}
 \]
Thirdly
\[
\kappa(\G/\V) = 
\begin{cases}
\kappa((\G\ba e)/\V)  - 1& \text{if }\G/  e^c = \bs \\
\kappa((\G\ba e)/\V)  & \text{if }\G/ e^c \neq \bs \\
\kappa((\G/e)/\V) &\text{always}
\end{cases}
\]

\[
\kappa(\G^*/\B) =
\begin{cases}
\kappa((\G^*\ba e)/\B) &\text{always}\\
\kappa((\G^*/e)/\B)  - 1& \text{if }\G\ba e^c = \olc  \\
\kappa((\G^*/e)/\B)  & \text{if }\G\ba e^c \neq \olc.  \\
\end{cases}
\]
Finally, 
\[
v(\G) =
\begin{cases}
v(\G/e) &\text{if }\G\ba e^c = \nl \\
v(\G/e) - 1& \text{if }\G\ba e^c = \olc \text{ or } \olh  \\
v(\G/e) +1 & \text{if }\G\ba e^c = \bs \text{ or } \bp  \\
v(\G\ba e) & \text{always.}
\end{cases}
\]

\end{lemma}
\begin{proof}
Each function $r_k(A)$ is expressible in terms of $r_{\G^*/\B}(A)$, $\rho_{\G}(A)$, and $r_{\G/\V}(A)$. Equations~\eqref{e.7a}--\eqref{e.8b} describe how those functions act under deletion and contraction. 
 
Trivially,  $\widetilde{(\G/ e)} =\tilde{\G}/ e$ and $\widetilde{(\G\ba  e)} =\tilde{\G}\ba e$.
From the proof of Lemma~\ref{asr}, $ (\G /e ) /\V=  (\G/e)/ \V$, and, trivially, $ (\G \ba e ) /\V=  (\G\ba e)/ \V$. 
From Table~\ref{tablefp} it is readily seen that $ (\G^* \ba e ) /\B=  (\G^*/ \B) / e$. Similarly, by considering $G^*$ locally at an edge and forming $ (\G^* / e ) /\B$  and $(\G^*/ \B)\ba e$ it is seen that $ (\G^* / e ) /\B=  (\G^*/ \B)\ba e$. These identities are used in the omitted rank calculations.

Then 
\begin{align*}
r_1(\G) = r(\G/\V)  &=   
 \begin{cases}
    r(G\ba e)       & \quad \text{if } e \text{ is a bridge in } \G/\V \\
    r(G\ba e)+1  & \quad \text{otherwise} 
  \end{cases}
  \\&
  = 
   \begin{cases}
    r(G\ba e)       & \quad \text{if } \G/e^c =  \bs\\
    r(G\ba e)+1  & \quad \text{otherwise} ,
  \end{cases}
  \end{align*}
where the last equality is by Lemma~\ref{asr}. 
Similar arguments give \eqref{hasw1} and \eqref{hasw3}.

For \eqref{hasw2}, if $e\notin A$ then 
\[  r_{k,\G} (A)  =   r_k(\G\ba A^c)   =   r_k((\G\ba A^c) \ba e ) +\delta_{k,j}=  r_k((\G \ba e) \ba A^c )+\delta_{k,j} =  r_{k,\G\ba e} (A)+\delta_{k,j}. \]

For \eqref{hasw4}, if $e\in A$ then 
\[  r_{k,\G} (A)  =   r_k(\G\ba A^c)   =   r_k((\G\ba A^c) /e ) +\epsilon_{k,j}=  r_k((\G /e) \ba A^c )+\epsilon_{k,j} =  r_{k,\G/e} (A\ba e)+\epsilon_{k,j} . \]

By \eqref{e.7a},  $\kappa(\G/\V)$ and $\kappa((\G\ba e)/\V)$ differ if and only if $e$ is a bridge in $\G/\V$, in which case  $\kappa(\G/\V)= \kappa((\G\ba e)/\V)  - 1$.  By Lemma~\ref{asr} this will happen if and only if $\G/e^c$ is $\bs$.
 Also, since graph contraction does not change the number of connected components  $\kappa(\G/\V)=\kappa((\G/ e)/\V)$.

By \eqref{e.7b},  and recalling that contraction of an edge of a coloured ribbon graph acts as deletion in $\G^* /\B$,   $\kappa(\G^* /\B)$ and $\kappa((\G^*/ e)/\B)$ differ if and only if $e$ is a bridge in $\G^*/\B$ in which case $\kappa(\G^* /\B)=\kappa((\G^*/ e)/\B)-1$. 
 Also, since  deletion of an edge of a coloured ribbon graph acts as contraction in $\G^* /\B$ and 
  graph contraction does not change the number of connected components  $\kappa(\G^*/\B)=\kappa((\G^*/ e)/\B)$.
  
Finally, the number of vertices will change under contraction if $e$ is a non-loop edge, in which case it  decreases by one, or if $e$ is an orientable loop, in which case it increases by one. By Lemma~\ref{asr}, this happens if  $\G\ba e^c = \olc \text{ or } \olh$, or    $\G\ba e^c = \bs \text{ or } \bp$, respectively. Deletion does not change the number of vertices. 
\end{proof}

\begin{proof}[Proof of Theorem~\ref{zxa}]
The method of proof is a standard one in the theory of the Tutte polynomial. We let
\begin{multline*}
 \theta(\G,A) :=
\alpha^{k(\G/\V)} \beta^{k(\G^*/\B)} \gamma^{v(\G)}
(\alpha \, a_{\bs})^{r_1(\G)} 
a_{\bp}^{r_2(\G)}
a_{\olc}^{r_3(\G)}
 a_{\olh}^{r_4(\G)} \\
\left(\frac{b_{\bs}}{\alpha \gamma\,  a_{\bs}}\right)^{r_1(A) }
\left(\frac{b_{\bp}}{\gamma\, a_{\bp}}\right)^{r_2(A)} 
\left(\frac{\beta\gamma\,  b_{\olc}}{a_{\olc}}\right)^{r_3(A)}
\left(\frac{\gamma\, b_{\olh}}{a_{\olh}}\right)^{ r_4(A)}
  \end{multline*}
so that, with this notation, the theorem claims that $U(\G)=\sum_{A\subseteq E} \theta(\G,A)$.

We prove the result by induction on the number of edges of $\G$. If $\G$ is edgeless then the result is easily seen to hold. Now let $\G$ be a coloured ribbon graph on at least one edge and suppose  that  the theorem holds for all coloured ribbon graphs with fewer edges than $\G$.

Let $e$ be an edge of $\G$. Then we can write 
\begin{equation}\label{vbv}
\sum_{A\subseteq E(\G)} \theta(\G,A) = \sum_{\substack{A\subseteq E(\G) \\ e\notin A}} \theta(\G,A)  +   \sum_{\substack{A\subseteq E(\G) \\ e\in A}} \theta(\G,A).
\end{equation}


Consider the first sum in the right-hand side of \eqref{vbv}. Using Lemma~\ref{rsf} and the inductive hypothesis to rewrite the exponents in terms of $\G\ba e$, we have
\[  \theta(\G,A)  =    
\begin{cases}
a_{\bs} \,  \theta(\G\ba e,A\ba e)  & \text{if }\G/  e^c = \bs \\ 
a_{\bp} \, \theta(\G\ba e,A\ba e)  & \text{if }\G/ e^c = \bp \\ 
a_{\olc} \,  \theta(\G\ba e,A\ba e)  & \text{if }\G/ e^c = \olc \\ 
a_{\olh} \, \theta(\G\ba e,A\ba e)  & \text{if }\G/ e^c = \olh \\ 
a_{\bp}^{1/2} b_{\olh}^{1/2} \, \theta(\G\ba e,A\ba e)  & \text{if }\G/ e^c = \nl. 
\end{cases}
 \]


Now consider the second sum in the right-hand side of \eqref{vbv}. Using Lemma~\ref{rsf} and the inductive hypothesis to rewrite  the exponents in terms of $\G/e$ we have
\[  \theta(\G,A)  =   
\begin{cases}
b_{\bs} \,  \theta(\G/e,A\ba e)  & \text{if }\G\ba e^c = \bs \\ 
b_{\bp} \, \theta(\G/e,A\ba e)  & \text{if }\G\ba e^c = \bp \\ 
b_{\olc} \,  \theta(\G/e,A\ba e)  & \text{if }\G\ba e^c = \olc \\ 
b_{\olh} \, \theta(\G/e,A\ba e)  & \text{if }\G\ba e^c = \olh \\ 
b_{\bp}^{1/2} b_{\olh}^{1/2} \, \theta(\G/e,A\ba e)  & \text{if }\G\ba e^c = \nl. 
\end{cases}
 \]
 
Collecting this together, we have shown that 
\[ \sum_{A\subseteq E(\G)} \theta(\G,A)  = 
   a_{i}\left(   \sum\limits_{A\subseteq E(\G\ba e)} \theta(\G\ba e,A\ba e)  \right) +b_{j}  \left(   \sum\limits_{A\subseteq E(\G/ e)} \theta(\G/ e,A\ba e)  \right)  
\]
if $e$ is of type $(i,j)$, 
and 
$\alpha^{n}\beta^{m} \gamma^v$ if  $\G$  is edgeless, with $v$ vertices, $n$ vertex colour classes, and $m$  boundary colour classes, and
where  $a_{\nl}=a_{\bp}^{1/2}a_{\olh}^{1/2}$ and $b_{\nl}=b_{\bp}^{1/2}b_{\olh}^{1/2}$.
The theorem follows.
\end{proof}

\begin{proof}[Proof of Theorem~\ref{mnc}]
We use Lemma~\ref{asr} to give the alternative descriptions of $(i,j)$-edges.
By Theorem~\ref{zxa}, $T_{ps}(\G; x,y,a,b)$ satisfies the deletion-contraction relations in Theorem~\ref{fdgu} with $a_{\bs}=w$,  $a_{\bp}=x$, $b_{\olc}=y$, $b_{\olh}=z$,  $a_{\nl}=\sqrt{x}$, $b_{\nl}=\sqrt{z}$, and all other variables set to 1. 
Note that if  $e$ is a bridge in $\G^*/\B$ then it must be  an orientable loop in $\G$ and a loop in $\G/\V$.   Also if $e$ is a loop in $\G$ then it must be a loop in $\G/\V$. These observations simplify the descriptions.
\end{proof}

\section*{Acknowledgements}

We are grateful to the London Mathematical Society for its support, and to Chris and Diane Reade for their hospitality.

\end{document}